\documentclass[11pt,a4paper]{amsart}

\RequirePackage[OT1]{fontenc}
\RequirePackage{amsthm,amsmath}
\RequirePackage[numbers]{natbib}
\RequirePackage[colorlinks,citecolor=blue,urlcolor=blue]{hyperref}

\usepackage{graphicx}
\newtheorem{theorem}{Theorem}
\newtheorem{lemma}[theorem]{Lemma}
\newtheorem*{lemma*}{Lemma}
\newtheorem{proposition}[theorem]{Proposition}

\newtheorem{definition}[theorem]{Definition}
\newtheorem{remark}[theorem]{Remark}

\newtheorem*{fact*}{Fact}

\usepackage[utf8]{inputenc}
\usepackage[english]{babel}
\usepackage{latexsym}
\usepackage{amssymb}
\usepackage{amsmath}

\newcommand{\T}{\mathbb{T}}

\newcommand{\Z}{\mathbb{Z}}
\newcommand{\R}{\mathbb{R}}
\newcommand{\C}{\mathbb{C}}
\newcommand{\E}{\mathbb{E}}

\begin{document}

\title[A discrete log gas, discrete FH-singularities, and GMC]{A discrete log gas, discrete Toeplitz determinants with Fisher-Hartwig singularities, and Gaussian Multiplicative Chaos}

\author[C. Webb]{Christian Webb}
\address{Department of mathematics and systems analysis, Aalto University, PO Box 11000, 00076 Aalto, Finland}
\email{christian.webb@aalto.fi}
\date{\today}

\begin{abstract}
We consider a log-gas on a discretization of the unit circle. We prove that if the gas is not too dense, or the number of particles in the gas is not too large compared to the scale of the discretization, the absolute value of the characteristic polynomial can be described in terms of a Gaussian multiplicative chaos measure. This is done by analyzing discrete Toeplitz determinants with Fisher-Hartwig singularities. In particular, we prove that if the gas is not too dense, the classical Fisher-Hartwig conjecture holds for the discrete Toeplitz determinant as well. Our analysis suggests that if the gas is any denser than this, the formula needs to be modified. 
\end{abstract}

\maketitle

\section{Introduction}

It is a basic fact in random matrix theory that for a large class of random matrix models, the distribution of the eigenvalues can be described in terms of a log gas - namely a Gibbs measure of a system of particles experiencing an external confining potential and interacting with each other through a logarithmic repulsion. Notable examples of such models are the Gaussian ensembles, circular ensembles, and the Ginibre Ensemble - for more details, see e.g \cite{mehta,forrester}.

\vspace{0.3cm}

When considering the global (or macroscopic) scale, asymptotic properties of the distribution of eigenvalues are often described in terms of linear statistics: if $(\lambda_1,...,\lambda_N)$ are the eigenvalues, one considers $\sum_j g(\lambda_j)$ for some nice enough function $g$. The leading order behavior of such linear statistics are known to be described by a law of large numbers - the asymptotic distribution is described in terms of a deterministic equilibrium measure (see e.g. \cite{johansson1}). The fluctuations around this equilibrium measure are known in many cases to be Gaussian (see e.g. \cite{ds,johansson1,ridvir}). 

\vspace{0.3cm}

In fact, these fluctuations in the linear statistics imply that the fluctuations in the distribution of the eigenvalues, e.g. on the level of the logarithm of the characteristic polynomial, are asymptotically described by the Gaussian Free Field (perhaps restricted to an interval, the unit circle, or the unit disk and with a suitable convention for the zero mode) - see e.g. \cite{hko,fks,ridvir}. Such fluctuations have been seen even on the level of discrete log-gases - for these, the underlying space where the log-gas lives on is a discrete set. Such models arise in for example combinatorial questions - see e.g. \cite{johansson2,bgg}.

\vspace{0.3cm}

The Gaussian Free Field is a rough object (for an introduction to it, see e.g. \cite{sheffield}). In particular, its "geometric" properties are far from trivial. It has recently been discovered that the geometry of the Gaussian Free Field, and more generally log-correlated Gaussian fields, is best described by "exponentiating" them into random measures through the theory of Gaussian multiplicative chaos going back to Kahane \cite{kahane} (for a review see \cite{rv} and for a concise proof of existence and uniqueness, see \cite{berestycki}). This exponentiating has proven to be intimately related to the multifractal properties of the field (see e.g. \cite{rv}) and the extrema of it (see \cite{drz,mad}). These random measures also play a crucial role in the mathematical study of the random geometry of two-dimensional quantum gravity (\cite{dupshef}) and construction of conformally invariant random planar curves (\cite{shefweld,ajks}).

\vspace{0.3cm}

The way these measures are constructed is by regularizing the given log-correlated field into a function, exponentiating and normalizing this, and then passing to a limit where the regularization is removed. A natural question is then do all (reasonable) regularizations produce a multiplicative chaos measure in this way (and is the law of the limit the same for every regularization). Mostly the regularizations which are known to produce a chaos measure are trivial in the sense that the regularization imposes a  martingale structure or something else which makes proving convergence quite simple. To the author's knowledge, the only case, where a Gaussian multiplicative chaos measure has been proven to emerge from a "non-trivial" regularization of a log-correlated field, is the characteristic polynomial of the circular unitary ensemble (see \cite{webb} based on conjectures in \cite{fk} and results in \cite{dik1,dik2,ck}). One of the goals of this note is to offer an example from another class of models - the characteristic polynomial of a discrete log-gas which is in fact a discretization of the law of the eigenvalues of the CUE.

\vspace{0.3cm}

The main tool for analyzing our model is the realization that the relevant quantities can be described as discrete Toeplitz determinants with Fisher-Hartwig singularities. In our setting, a discrete Toeplitz determinant is one whose entries are given in terms of a discrete Fourier transform of a symbol instead of a continuum one. In \cite{bl}, the asymptotics of such discrete Toeplitz determinants are studied in the case where the symbol has no singularities. Comparing to the continuum setting, it is a natural question then to consider what happens when the symbol can have Fisher-Hartwig singularities. A further goal of this note is to describe (in terms of  the relationship between the number of particles in the gas compared to the number of lattice points in the discretization) when does a discrete Toeplitz determinant with Fisher-Hartwig singularities have the same asymptotics as a continuum one. Our result concerning this is that if there are $N$ particles in the gas and $M$ points in the discretization of the unit circle, and if $N/M\to 0$ as $N\to\infty$, then the discrete and continuum Toeplitz determinants with FH-singularities have the same asymptotics. Perhaps more interestingly, our analysis suggests that if $N/M\to q>0$ as $N\to\infty$, the asymptotics of the discrete Toeplitz determinant will deviate from the continuum one - see the last section of this note for a discussion about this. Another interesting question (which we don't discuss any further) is how our results relate to discrete Riemann-Hilbert problems (see e.g. \cite{bkmlm} for how these are used for studying asymptotics of discrete orthogonal polynomials).

\vspace{0.3cm}

The outline of this note is the following: we start by describing our model and quickly review a construction for the relevant chaos measure. We then introduce our main result and the smaller ingredients this will consist of. Next we go on to review continuum and discrete Toeplitz determinants. We then prove our main tool for analyzing discrete Toeplitz determinants - a generalization of a result in \cite{bl}, namely a description in terms a continuum Toeplitz determinant with the same symbol and a Fredholm determinant. Due to this representation, the relevant question then becomes analyzing the asymptotics of the Fredholm determinant which can be done by making use of the asymptotics of the (continuum) orthogonal polynomials with respect to a Fisher-Hartwig weight. The bulk of the note consists of straightforward, but tedious analysis of the asymptotics of the polynomials and the Fredholm determinant - this part relies heavily on results in \cite{mfmls1,mfmls2,ck}. Finally when we've proven the relevant asymptotics, we are able to prove our main result (from the point of view of Gaussian multiplicative chaos it is Theorem \ref{th:main}, but from the point of view of discrete Toeplitz determinants with FH-singularities it is Proposition \ref{prop:fh2}). After the proofs, we discuss briefly the possibility of deviating from the continuum FH-asymptotics.

\vspace{0.3cm}

\bf Acknowledgements: \rm The author wishes to thank Y.V. Fyodorov for discussions which have been extremely influential on the author's view of log-gases and seminal for the work in this note. The author also wishes to thank A. Kupiainen and E. Saksman for discussions related to trace class operators, Hilbert-Schmidt operators, and Fredholm determinants.

\section{Model, main result, and outline of approach}

We will now review our model, introduce the relevant objects and notation to state our main result, and finally discuss what are the relevant estimates we shall need. 

\subsection{The model - a discrete log-gas and its characteristic polynomial}

We begin by considering a discretization of the unit circle ($M$ evenly spaced points on the unit circle) and then consider a probability measure on the $N$-fold product of this discrete unit circle with itself. The resulting probability measure can be seen to be a discretization of the law of the eigenvalues of a Haar distributed $N\times N$ random unitary matrix.

\begin{definition}
Let us denote by $\T$, the unit circle in the complex plane. Fix $M\in\Z_+$, let $\gamma_M:\C\to \C$,

\begin{equation}\label{eq:gamma}
\gamma_M(z)=z^M-1,
\end{equation}

\noindent and let 

\begin{equation}
\mathcal{D}_M=\lbrace z\in \C: \gamma_M(z)=0\rbrace\subset \T.
\end{equation}

We then fix some $N\in \Z_+$ such that $N\leq M$ and consider the following probability measure on $\mathcal{D}_M^N$:

\begin{equation}\label{eq:measure}
\mathbb{P}_{N,M}(z_1,...,z_N)=\frac{1}{Z_{N,M}}\prod_{1\leq i<j\leq N}|z_i-z_j|^2,
\end{equation}

\noindent where 

\begin{equation}
Z_{N,M}=\sum_{z_1,...,z_N\in \mathcal{D}_M}\prod_{1\leq i<j\leq N}|z_i-z_j|^2.
\end{equation}

We'll denote by $\E_{N,M}$ or just $\E$ the expectation with respect to this measure (we'll also drop the subscripts $N,M$ from $\mathbb{P}_{N,M}$ when convenient).

\end{definition}

\begin{remark}
We'll prove shortly that 

\begin{equation}\label{eq:Z}
Z_{N,M}=N!M^N.
\end{equation}
\end{remark}

\begin{remark}\label{rem:nbound}
The restriction that $N\leq M$ is essential as $Z_{N,M}=0$ if $N>M$ (if $N>M$ we can't find $N$ distinct points from $\mathcal{D}_M$ which contains $M$ points). On the other hand, the other extreme would be when we let $M\to\infty$ as $N$ remains fixed. In this case we would simply get a log-gas on the unit circle and the law of the points $(z_1,...,z_N)$ would be that of the eigenvalues of the Circular Unitary Ensemble (or Haar distributed random unitary matrices). 
\end{remark}

The way we wish to describe the limiting properties of these probability measures as $N,M\to \infty$ is through the "characteristic polynomial" of the random points $(z_1,...,z_N)$. For simplicity, we focus on its absolute value (as the zeros give the locations of the points $z_i$, this describes the global geometry just as well as the full characteristic polynomial). Also with the same effort, we can study (small enough) positive real powers of this quantity, so we add a further parameter to the quantity. Moreover, as we wish to prove convergence to a multiplicative chaos measure, we want to treat it as a measure so we make the following definitions:

\begin{definition}
Let $\beta>0$ and $F_{N,M}^{\beta}:\C\to[0,\infty)$,

\begin{equation}
F_{N,M}^{\beta}(w)=\prod_{j=1}^N|w-z_j|^\beta.
\end{equation}

Moreover, consider the following Radon measure on the unit circle:

\begin{equation}
\mu_{N,M}^{\beta}(d\theta)=\frac{F_{N,M}^{\beta}(e^{i\theta})}{\E F_{N,M}^{\beta}(e^{i\theta})}\frac{d\theta}{2\pi}.
\end{equation}

\end{definition}

We now describe what $\mu_{N,M}^{\beta}$ will converge to as $N,M\to\infty$ in a suitable way as well as the mode of convergence. 

\subsection{Gaussian multiplicative chaos and convergence in distribution with respect to the weak topology of Radon measures.}

The relevant limiting object is a Gaussian Multiplicative Chaos measure. We refer to \cite{rv} for a general review on the theory of such measures and to \cite{berestycki} for a concise proof for the existence (and uniqueness) of them.

\begin{definition}
Let $(Z_j)_{j=1}^\infty$ be i.i.d. standard complex Gaussians, $L\in \Z_+$ and $X_L:[0,2\pi)\to \R$,

\begin{equation}
X_L(\theta)=\mathrm{Re}\sum_{j=1}^L\frac{1}{\sqrt{j}}Z_j e^{ij\theta}.
\end{equation} 

For $\beta>0$, consider the Radon measures

\begin{equation}
\mu_L^\beta(d\theta)=\frac{e^{\beta X_L(\theta)}}{\E(e^{\beta X_L(\theta)})}\frac{d\theta}{2\pi}
\end{equation}

\noindent and 

\begin{equation}
\mu^\beta(d\theta)=\lim_{L\to\infty}\mu^\beta_L(d\theta),
\end{equation}

\noindent where the limit is in the almost sure sense and with respect to the topology of weak convergence on the space of Radon measures on the unit circle, i.e. for every continuous $f:\T\to\R$

\begin{equation}
\int_0^{2\pi}f(e^{i\theta})\mu_L^\beta(d\theta)\stackrel{L\to\infty}{\to}\int_0^{2\pi}f(e^{i\theta})\mu^\beta(d\theta)
\end{equation}

\noindent almost surely.
\end{definition}

\begin{remark}
One can check that 

\begin{equation}
\E(X_L(\theta)X_L(\theta'))\stackrel{L\to\infty}{\to}-\frac{1}{2}\log |e^{i\theta}-e^{i\theta'}|,
\end{equation}

\noindent so one can view $X_L$ as a (constant multiple of a) regularization of the two-dimensional Gaussian Free Field restricted to the unit circle. If we write $X$ for this field, then formally $\mu_\beta(d\theta)=e^{\beta X(\theta)-\frac{\beta^2}{2}\E(X(\theta)^2)}d\theta/2\pi$. This exponential is purely formal as $X$ is a generalized function and $\E(X(\theta)^2)=\infty$. The proper definition of $\mu^\beta$ is precisely in terms of this limiting procedure. In \cite{berestycki,js} it is proved that constructing $\mu^\beta$ through almost any other reasonable regularization of the field $X$ yields a measure with the same law as $\mu^\beta$.
\end{remark}

\begin{remark}
A basic fact in the theory of Gaussian Multiplicative Chaos is that $\mu_L^\beta$ indeed converges almost surely to a non-trivial random measure (which we call $\mu^\beta$) for $0<\beta<2$. For $\beta\geq 2$, $\mu^\beta=0$.
\end{remark}

When proving convergence of $\mu_{N,M}^\beta$, the topology we'll consider convergence in is still that of weak convergence of measures, but the convergence is no longer almost sure, but that of convergence in distribution. We refer to section 4 of \cite{kallenberg} for a general version of the following characterization of convergence in distribution to  with respect to the weak topology of measures on the unit circle.

\begin{proposition}\label{prop:dconv}
A sequence of random measures $\nu_n$ on the unit circle converges in distribution with respect to the topology of weak convergence of measures, to the measure $\nu$ if and only if, as $n\to\infty$,

\begin{equation}
\int_0^{2\pi}f(e^{i\theta})\nu_n(d\theta)\stackrel{d}{\to}\int_0^{2\pi}f(e^{i\theta})\nu(d\theta),
\end{equation}
 
\noindent for every continuous $f:\T\to [0,\infty)$.
\end{proposition}

\subsection{Main theorem}

We are now in a position to state our main theorem.

\begin{theorem}\label{th:main}
Let $N/M\to 0$ as $N\to\infty$, and $\beta\in(0,\sqrt{2})$. Then as $N\to\infty$,

\begin{equation}
\mu_{N,M}^\beta(d\theta)\to \mu^\beta(d\theta)
\end{equation}

\noindent in the sense of convergence in distribution with respect to the topology of weak convergence of measures.

\end{theorem}

\begin{remark}
Note that we do not prove convergence for all values of $\beta$. This is due to the fact that our approach relies crucially on calculating variances which are finite only for $\beta\in(0,\sqrt{2})$. It is an interesting open problem to extend such results to $\beta\in[\sqrt{2},2)$ (or even further).
\end{remark}

\begin{remark}
As noted in Remark \ref{rem:nbound}, one must have $N\leq M$ to have a meaningful model, but this result still leaves open for example the cases where $N/M\to q\in(0,1]$ as $N\to\infty$. In fact, there are some hints in our approach that perhaps something else happens and the discreteness of the model (deviation from the $M\to\infty$ for fixed $N$ - situation) becomes visible in the characteristic polynomial. For further discussion, see the last section of this note.
\end{remark}

\subsection{Structure of the proof} We will now provide an outline for our proof of Theorem \ref{th:main}. Our argument is a rather standard probabilistic argument - we add a further level of approximation to $\mu_{N,M}^\beta$, by truncating the Fourier series of $\log F_{N,M}^\beta(e^{i\theta})$, and we show that as $N\to\infty$, this converges in law to something that converges to $\mu^\beta$ as we remove the truncation. In addition to this, we prove that the variance of the error of this approximation tends to zero as we first let $N\to\infty$ and then remove the truncation. More precisely, let us make the following definitions:

\begin{definition}
For $L\in \Z_+$, let 

\begin{equation}
F_{N,M,L}^\beta(e^{i\theta})=e^{-\beta\mathrm{Re}\sum_{j=1}^L\frac{1}{j}e^{-ij\theta}(\sum_{k=1}^N z_k^j)}
\end{equation}

\noindent and 

\begin{equation}
\mu_{N,M,L}^\beta(d\theta)=\frac{F_{N,M,L}^\beta(e^{i\theta})}{\E(F_{N,M,L}^\beta(e^{i\theta}))}\frac{d\theta}{2\pi}.
\end{equation}

\end{definition}

The first (and easier) ingredient for the proof will be to show that $\mu_{N,M,L}^\beta$ converges to $\mu_L^\beta$ as $N\to\infty$. More precisely

\begin{proposition}\label{prop:approxconv}
For any $\beta>0$, $L\in \Z_+$, and $N,M$ such that $M-N\to\infty$ as $N\to\infty$,

\begin{equation}
\mu_{N,M,L}^\beta(d\theta)\stackrel{d}{\to}\mu_L^\beta(d\theta)
\end{equation}

\noindent as $N\to\infty$ (with respect to the topology of weak convergence of measures on the unit circle).
\end{proposition}

This will actually follow from convergence of certain linear statistics.

\begin{proposition}\label{prop:linstat}
Let $(Z_j)_{j=1}^\infty$ be i.i.d. standard complex Gaussians and for $j\in \Z_+$, let 

\begin{equation}
\widetilde{Z}_j=\sum_{k=1}^N z_k^j.
\end{equation}

Then for any fixed $l\in \Z_+$, and any $M, N$ such that $M-N\to\infty$ as $N\to\infty$,

\begin{equation}
(\widetilde{Z}_1,...,\widetilde{Z}_l)\stackrel{d}{\to}(Z_1,\sqrt{2}Z_2,...,\sqrt{l}Z_l)
\end{equation}

\noindent as $N\to\infty$.
\end{proposition}

\begin{remark}
Note that the conditions on $\beta,N,$ and $M$ are much weaker here than in our main theorem. The fact that we don't have any restrictions on $\beta$ is due to the regularized field $X_L$ being a non-singular object so exponentiation is trivial (the convergence of linear statistics is essentially a statement about convergence of the Fourier coefficients of $\log F_{N,M,L}^\beta$ to those of the field $\beta X_L$). The fact that we have less restrictive conditions on $N$ and $M$ is due to the fact that the characteristic polynomial seems to be more sensitive to the discreteness of the model than the linear statistics.
\end{remark}

The next ingredient for our proof will be estimating the variance of the approximation due to introducing $L$.

\begin{proposition}\label{prop:var}
For any $\beta\in(0,\sqrt{2})$, any continuous $f:\T\to[0,\infty)$, and $N,M$ such that $N/M\to 0$ as $N\to\infty$,

\begin{equation}
\lim_{L\to\infty}\lim_{N\to\infty}\E\left(\left(\int_0^{2\pi}f(e^{i\theta})(\mu_{N,M,L}^\beta(d\theta)-\mu_{N,M}^\beta(d\theta))\right)^2\right)=0.
\end{equation}
\end{proposition}

Expanding the square in the proposition above, we see that this boils down to precise, uniform estimates on (possibly mixed) moments of $F_{N,M,L}^\beta$ and $F_{N,M}^\beta$. The relevant estimates are given in the following result.

\begin{proposition}\label{prop:fh}

Let $\beta\in(0,\sqrt{2})$ and $L\in \Z_+$. 

\vspace{0.3cm}

1) If $M-N\to\infty$ as $N\to\infty$,

\begin{equation}
\lim_{N\to\infty}\E(F_{N,M,L}^\beta(e^{i\theta}))=e^{\frac{\beta^2}{4}\sum_{j=1}^L\frac{1}{j}}
\end{equation}

\noindent uniformly in $\theta$.

\vspace{0.3cm}

2) If $N/M\to 0$ as $N\to \infty$, 

\begin{equation}
\E(F_{N,M}^\beta(e^{i\theta}))=N^{\frac{\beta^2}{4}}\frac{G\left(1+\frac{\beta}{2}\right)^2}{G(1+\beta)}\times(1+\mathit{o}(1)),
\end{equation}

\noindent where $G$ is the Barnes $G$ function and $\mathit{o}(1)$ is uniform in $\theta$.

\vspace{0.3cm}

3) If $M-N\to\infty$ as $N\to \infty$, 

\begin{equation}
\E\left(F_{N,M,L}^\beta(e^{i\theta})F_{N,M,L}^\beta(e^{i\theta'})\right)=E_{N,L}(\theta,\theta')\times(1+\mathit{o}(1)),
\end{equation}

\noindent where $\mathit{o}(1)$ is uniform in $\theta,\theta'$ and as $N\to\infty$, $E_{N,L}(\theta,\theta')$ increases to $\E(e^{\beta X_L(\theta)}e^{\beta X_L(\theta')})$.

\vspace{0.3cm}

4) If $N/M\to 0$ as $N\to\infty$, then 

\begin{align}
\notag \E(F_{N,M,L}^\beta(e^{i\theta})F_{N,M}^\beta(e^{i\theta'}))&=e^{\frac{\beta^2}{4}\sum_{j=1}^L \frac{1}{j}}e^{\frac{\beta^2}{2}\sum_{j=1}^L \frac{1}{j}\cos(j(\theta-\theta'))}\\
&\times N^{\frac{\beta^2}{4}}\frac{G\left(1+\frac{\beta}{2}\right)^2}{G(1+\beta)}\times(1+\mathit{o}(1)),
\end{align}

\noindent where $\mathit{o}(1)$ is uniform in $\theta,\theta'$.

\vspace{0.3cm}

5) If $N/M\to 0$ as $N\to \infty$ and there is a fixed $\delta>0$ such that $|\theta-\theta'|>\delta$, then 

\begin{equation}
\E\left(F_{N,M}(e^{i\theta})F_{N,M}(e^{i\theta'})\right)=N^{\frac{\beta^2}{2}}|e^{i\theta}-e^{i\theta'}|^{-\frac{\beta^2}{2}}\left(\frac{G\left(1+\frac{\beta}{2}\right)^2}{G(1+\beta)}\right)^2\times(1+\mathit{o}(1)),
\end{equation}

\noindent where $\mathit{o}(1)$ is uniform on $|\theta-\theta'|>\delta$.

\vspace{0.3cm}

6) There is a fixed $t_0$ such that if $N/M\to 0$ as $N\to \infty$, then for $|\theta-\theta'|<2t_0$

\begin{align}
\notag \E\left(F_{N,M}(e^{i\theta})F_{N,M}(e^{i\theta'})\right)&=N^{\beta^2}\frac{G(1+\beta)^2}{G(1+2\beta)}e^{\int_0^{-iN|\theta-\theta'|}\frac{1}{s}\left(\sigma(s)-\frac{\beta^2}{2}\right)ds}\\
&\times e^{-\frac{\beta^2}{2}\log \frac{2\sin\frac{|\theta-\theta'|}{2}}{|\theta-\theta'|}}\times(1+\mathit{o}(1)),
\end{align}

\noindent where $\mathit{o}(1)$ is uniform in $|\theta-\theta'|<2t_0$ and $\sigma$ is a continuous function (depending only on $\beta$ - not on $N,M,\theta,\theta',$ or $t_0$ ) on $-i\R_+$ with the following asymptotics: there exists a $\delta>0$ such that 

\begin{equation}
\sigma(s)=\frac{1}{2}\beta^2+\mathcal{O}(|s|^\delta)
\end{equation} 

\noindent as $s\to 0$ along the negative imaginary axis and 

\begin{equation}
\sigma(s)=\mathcal{O}(|s|^{-\delta})
\end{equation}

\noindent as $s\to\infty$ along the negative imaginary axis.

\end{proposition}

The way these results are obtained is through a discrete Toeplitz determinant representation of moments of $F_{N,M,L}^\beta$ and $F_{N,M}^\beta$. We then show, mimicking an argument for non-singular Toeplitz determinants given in \cite{bl}, that this discrete Toeplitz determinant can be represented in terms of a standard, or continuum, Toeplitz determinant with Fisher-Hartwig singularities as well as a certain Fredholm determinant. The asymptotics of the continuum Toeplitz determinants are well known (see \cite{dik1,dik2,ck} and references therein), so the problem then boils down to the analysis of the asymptotics of the Fredholm determinant and this is done using asymptotic properties of orthogonal polynomials with respect to the Fisher-Hartwig weight on the unit circle (which have been worked out in or follow from results in \cite{mfmls1,mfmls2,dik1,dik2,ck}). One of the main results of this note is describing the regime for certain discrete Toeplitz determinants with FH-singularities where the classical Fisher-Hartwig formula still holds. We thus write the statement down explicitly (statements $2), 4)$, and $5)$ are essentially special cases of this result).

\begin{proposition}\label{prop:fh2}
Let $V$ be a Laurent polynomial of the form 

\begin{equation}
V(z)=\sum_{j=0}^K \frac{1}{2}(\alpha_j z^j+\overline{\alpha_j}z^{-j}),
\end{equation}

\noindent let $(w_j)_{j=1}^m$ be fixed distinct points on the unit circle, and let $\beta_j>0$ for $j=1,...,m$. If we write for $|z|=1$,

\begin{equation}
f(z)=e^{V(z)}\prod_{j=1}^m|z-w_j|^{\beta_j}
\end{equation}

\noindent and

\begin{equation}
T_{N-1}(f)=\det\left(\frac{1}{M}\sum_{j=0}^{M-1}f\left(e^{\frac{2\pi i j}{M}}\right)e^{-(p-q)\frac{2\pi i j}{M}}\right)_{p,q=0}^{N-1}
\end{equation}

\noindent then as $N\to\infty$ and $M\geq N$

\begin{align}
\notag T_{N-1}(f)&=e^{N\frac{\alpha_0+\overline{\alpha_0}}{2}+\frac{1}{4}\sum_{j=1}^Kj|\alpha_j|^2}\prod_{j=1}^m e^{-\frac{\beta_j}{2}\sum_{l=0}^K\left(\frac{\alpha_l}{2}w_j^l+\frac{\overline{\alpha_l}}{2}w_j^{-l}\right)} N^{\sum_{j=1}^m \frac{\beta_j^2}{4}}\\
&\times \prod_{1\leq p<q\leq m}|w_p-w_q|^{-\frac{\beta_p\beta_q}{2}}\prod_{j=1}^m\frac{G\left(1+\frac{\beta_j}{2}\right)^2}{G(1+\beta_j)}\left(1+\mathcal{O}\left(\frac{N}{M}\right)\right).
\end{align}

\end{proposition}

\begin{remark}
So in particular, if $N/M\to 0$ as $N\to\infty$, we recover the classical Fisher-Hartwig formula. If on the other hand $N/M\to q\in(0,1]$, it is possible that this formula gets modified (for $q$ small enough, our analysis indicates that it shouldn't get modified too much - see our discussion at the end of the note).
\end{remark}

\section{Toeplitz determinants}

We will now review some facts about Toeplitz determinants (discrete and continuum) as well as prove a connection between discrete and continuum Toeplitz determinants with Fisher-Hartwig singularities.

\subsection{Toeplitz determinants and orthogonal polynomials}

We will now recall some basic facts about orthogonal polynomials on the unit circle and their relationship to Toeplitz determinants.

\begin{definition}\label{def:opuc}
Let $f:\T\to [0,\infty)$ and let $(\phi_j)_{j=0}^\infty$ be the orthonormal polynomials with respect to the weight $f(e^{i\theta})\frac{d\theta}{2\pi}$, i.e. $\phi_j(z)=\chi_j z^j+\mathcal{O}(z^{j-1})$ with $\chi_j\neq 0$ and 

\begin{equation}
\int_0^{2\pi}\phi_j(e^{i\theta})\overline{\phi_k(e^{i\theta})}f(e^{i\theta})\frac{d\theta}{2\pi}=\delta_{j,k}.
\end{equation}
\end{definition}

\begin{remark}\label{rem:polys}
It is a well known fact that such orthonormal polynomials have a determinantal representation and from this determinantal representation, one can check (e.g. using the Heine-Szeg\"o identity) that for a non-negative weight which is not Lebesgue almost everywhere zero on the unit circle, the orthonormal polynomials exist and are unique. Moreover, one can check with this argument that for a real weight, $\chi_j$ is real. For more information see e.g. \cite{dik1}.
\end{remark}

Let us now recall the definition of a continuum Toeplitz determinant.

\begin{definition}
For a $L^1$ function $f:\T\to \C$, we write 

\begin{align}
\notag \mathcal{T}_{N-1}(f)&=\det\left(\int_{\T}z^{-(j-k)}f(z)\frac{dz}{2\pi i z}\right)_{j,k=0}^{N-1}\\
&=\det\left(\int_0^{2\pi}e^{-i(j-k)\theta}f(e^{i\theta})\frac{d\theta}{2\pi}\right)_{j,k=0}^{N-1}.
\end{align}
\end{definition}

Its connection to the orthonormal polynomials is given by the following result.

\begin{proposition}\label{prop:toepprod}
Let $f:\T\to[0,\infty)$ be an $L^1$ function and let $\chi_j$ be as in Definition \ref{def:opuc}. Then
\begin{equation}
\mathcal{T}_{N-1}(f)=\prod_{j=0}^{N-1}\chi_j^{-2}.
\end{equation}
\end{proposition}

We will also make use of the Christoffel-Darboux identity.

\begin{proposition}\label{prop:cd}
Let $f$ be a weight on the unit circle such that the orthonormal polynomials exist and let $\overline{\phi}_j$ denote the polynomial obtained by complex conjugating the coefficients of $\phi_j$ (i.e. $\overline{\phi}_j(z)=\overline{\phi_j(\overline{z})}$). Then for $z,w\in \C$, $z,w\neq 0$, and $z\neq w$

\begin{equation}
\sum_{j=0}^{N-1}\phi_j(w)\overline{\phi}_j(z^{-1})=\frac{\frac{w^N}{z^N}\phi_N(z)\overline{\phi}_N(w^{-1})-\overline{\phi}_N(z^{-1})\phi_N(w)}{1-\frac{w}{z}}.
\end{equation}
\end{proposition}

Let us also recall the classical strong Szeg\"o theorem (see e.g. \cite{simonsz} for an extensive discussion related to the theorem). We consider it only in the generality we make use of.

\begin{theorem}[Szeg\"o]\label{th:szego}
Let $V$ be a Laurent polynomial of the form

\begin{equation}
V(z)=\frac{1}{2}\sum_{j=0}^K(\alpha_j z^j+\overline{\alpha_j}z^{-j}).
\end{equation}

Then 

\begin{equation}
\mathcal{T}_{N-1}(e^V)=e^{N\frac{\alpha_0+\overline{\alpha_0}}{2}+\frac{1}{4}\sum_{j=1}^K j|\alpha_j|^2}(1+\mathit{o}(1))
\end{equation}

\noindent as $N\to\infty$.

\end{theorem}

\vspace{0.3cm}

Finally let us recall the asymptotics of Toeplitz determinants with Fisher-Hartwig singularities. The results in the form we state them can be found in \cite{widom,dik2,ck}, though for a more comprehensive picture of the problem, see also e.g. \cite{dik1} and references therein. For the uniformity in the first statement, see e.g. \cite{dik2} where it is mentioned (it is a consequence of the uniform estimates of the jump matrices in the "final transformation" and the uniform expansions of the relevant local parametrices). 

\begin{theorem}[Widom; Deift, Its, and Krasovsky; Claeys and Krasovsky,...]\label{th:fhcont}
1) Let $V$ be a Laurent polynomial of the form 

\begin{equation}
V(z)=\sum_{j=1}^K \frac{1}{2}(\alpha_j z^j+\overline{\alpha_j}z^{-j}).
\end{equation}

Moreover, let $(w_j)_{j=1}^m$ be distinct points on the unit circle such that for some fixed $\delta>0$, $|w_i-w_j|>\delta$ for $i\neq j$, and $\beta_j>0$ for $j=1,...,m$. Then for 

\begin{equation}
f(z)=e^{V(z)}\prod_{j=1}^m |z-w_j|^{\beta_j},
\end{equation}

\begin{align}
\notag \mathcal{T}_{N-1}(f)&=e^{N\frac{\alpha_0+\overline{\alpha_0}}{2}+\frac{1}{4}\sum_{j=1}^K j|\alpha_j|^2}\prod_{j=1}^m e^{-\frac{\beta_j}{2}\sum_{l=1}^K\frac{1}{2}(\alpha_l w_j^l+\overline{\alpha_l}w_j^{-l})} N^{\sum_{j=1}^m\frac{\beta_j^2}{4}}\\
&\times \prod_{1\leq p<q\leq m}|w_p-w_q|^{-\frac{\beta_p\beta_q}{2}}\prod_{j=1}^m\frac{G\left(1+\frac{\beta_j}{2}\right)^2}{G(1+\beta_j)}(1+\mathit{o}(1)),
\end{align}

\noindent where $\mathit{o}(1)$ is uniform on $|w_i-w_j|>\delta$ and $G$ is the Barnes $G$-function.

\vspace{0.3cm}

2) For $t>0$, $|z|=1$, and $\beta>0$, let 

\begin{equation}
f_t(z)=|z-e^{it}|^\beta|z-e^{-it}|^\beta.
\end{equation}

There is a fixed $t_0$ such that for $t<t_0$

\begin{align}
\notag \mathcal{T}_{N-1}(f_t)&=N^{\beta^2}\frac{G(1+\beta)^2}{G(1+2\beta)}e^{\int_0^{-2iNt}\frac{1}{s}\left(\sigma(s)-\frac{\beta^2}{2}\right)ds}\\
&\times e^{-\frac{\beta^2}{2}\log \frac{\sin t}{t}}\times(1+\mathit{o}(1)),
\end{align}

\noindent where $\mathit{o}(1)$ is uniform in $t<t_0$ and $\sigma$ is a continuous function (depending only on $\beta$ - not on $N,t,$ or $t_0$) on $-i\R_+$ with the following asymptotics: there exists a $\delta>0$ such that 

\begin{equation}
\sigma(s)=\frac{1}{2}\beta^2+\mathcal{O}(|s|^\delta)
\end{equation} 

\noindent as $s\to 0$ along the negative imaginary axis and 

\begin{equation}
\sigma(s)=\mathcal{O}(|s|^{-\delta})
\end{equation}

\noindent as $s\to\infty$ along the negative imaginary axis.

\end{theorem}

\subsection{Discrete Toeplitz determinants}

A discrete Toeplitz determinant is one where the integrals appearing in the entries of $\mathcal{T}_{N-1}(f)$ are replaced by an approximating Riemann-sum. We now argue how these arise from expectations with respect to the measure $\mathbb{P}_{N,M}$.

\begin{proposition}\label{prop:dischs}
Let $f:\T\to \C$ be an arbitrary function. Then 

\begin{equation}
\E_{N,M}\prod_{j=1}^N f(z_j)=T_{N-1}(f):=\det\left(\frac{1}{M}\sum_{z\in \mathcal{D}_M}z^{-(j-k)}f(z)\right)_{j,k=0}^{N-1}.
\end{equation}
\end{proposition}

\begin{proof}
The proof is essentially the same as in the continuum case, but we write it down for the convenience of the reader. By definition

\begin{equation}
\E_{N,M}\prod_{j=1}^N f(z_j)=\frac{1}{Z_{N,M}}\sum_{z_1,..,z_N\in \mathcal{D}_M}\prod_{1\leq i<j\leq N}|z_i-z_j|^2\prod_{k=1}^Nf(z_j).
\end{equation}

We then make use of the fact that the first product is simply the square of the absolute value of the Vandermonde determinant: multiplying out the two determinants and using multilinearity in the first row we find

\begin{align}
\notag \prod_{1\leq i<j\leq N}|z_i-z_j|^2&=\left|\begin{array}{cccc}
1 & 1 & \cdots & 1\\
z_1 & z_2 & \cdots & z_N\\
\vdots & \vdots & \ddots & \vdots\\
z_1^{N-1} & z_2^{N-1} & \cdots & z_N^{N-1}
\end{array}\right|\left|\begin{array}{cccc}
1 & \overline{z_1} & \cdots & \overline{z_1}^{N-1}\\
1 & \overline{z_2} & \cdots & \overline{z_2}^{N-1}\\
\vdots & \vdots & \ddots & \vdots\\
1 & \overline{z_N} & \cdots & \overline{z_N}^{N-1}
\end{array}\right|\\
&=\left|\begin{array}{cccc}
N & \sum_{j=1}^N \overline{z_j} & \cdots & \sum_{j=1}^N\overline{z_j}^{N-1}\\
\sum_{j=1}^Nz_j & N & \cdots & \sum_{j=1}^N\overline{z_j}^{N-2}\\
\vdots & \vdots & \ddots & \vdots\\
\sum_{j=1}^{N-1}z_j^{N-1} & \sum_{j=1}^{N-1}z_j^{N-2} & \cdots & N
\end{array}\right|\\
&\notag =\sum_{l=1}^N\left|\begin{array}{cccc}
1 & \sum_{j=1}^N \overline{z_j} & \cdots & \sum_{j=1}^N\overline{z_j}^{N-1}\\
z_l & N & \cdots & \sum_{j=1}^N\overline{z_j}^{N-2}\\
\vdots & \vdots & \ddots & \vdots\\
z_l^{N-1} & \sum_{j=1}^{N-1}z_j^{N-2} & \cdots & N
\end{array}\right|.
\end{align}

Now in the sum in the second row, the $l$th term in the sum is proportional to the first column (the factor of proportionality being $\overline{z_l}$) thus - making use of multilinearity - this term in the sum yields a determinant equal to zero and we see that 

\begin{equation}
\prod_{1\leq i<j\leq N}|z_i-z_j|^2=\sum_{l_1=1}^N\sum_{1\leq l_2\leq N, l_2\neq l_1}\left|\begin{array}{cccc}
1 & \overline{z_{l_2}} & \cdots & \sum_{j=1}^N\overline{z_j}^{N-1}\\
z_{l_1} & 1 & \cdots & \sum_{j=1}^N\overline{z_j}^{N-2}\\
\vdots & \vdots & \ddots & \vdots\\
z_{l_1}^{N-1} & z_{l_2}^{N-2} & \cdots & N
\end{array}\right|.
\end{equation}

Continuing in this manner, we find

\begin{equation}
\prod_{1\leq i<j\leq N}|z_i-z_j|^2=\sum_{\sigma\in S_N}\left|\begin{array}{cccc}
1 & \overline{z_{\sigma(2)}} & \cdots & \overline{z_{\sigma(N)}}^{N-1}\\
z_{\sigma(1)} & 1 & \cdots & \overline{z_{\sigma(N)}}^{N-2}\\
\vdots & \vdots & \ddots & \vdots \\
z_{\sigma(1)}^{N-1} & z_{\sigma(2)}^{N-2} &  \cdots & 1
\end{array}\right|.
\end{equation}

Now as $\prod_{j=1}^N f(z_j)$ is invariant under permutations of the indexes $j$, we see by multilinearity that 

\begin{align}
\notag \E \prod_{j=1}^N f(z_j)&=\frac{N!}{Z_{N,M}}\sum_{z_1,...,z_N\in \mathcal{D}_M}\left|\begin{array}{cccc}
1 & \overline{z_{2}} & \cdots & \overline{z_{N}}^{N-1}\\
z_{1} & 1 & \cdots & \overline{z_{N}}^{N-2}\\
\vdots & \vdots & \ddots & \vdots \\
z_{1}^{N-1} & z_{2}^{N-2} &  \cdots & 1
\end{array}\right|\prod_{j=1}^N f(z_j)\\
&=\frac{N!}{Z_{N,M}}\det\left( \sum_{z\in \mathcal{D}_M} z^{-(j-k)}f(z)\right)_{j,k=0}^{N-1}.
\end{align}

Then setting $f(z)=1$ for all $z\in \C$ we see that 

\begin{equation}
1=\frac{N!}{Z_{N,M}}\det\left(\sum_{l=0}^{M-1} e^{-(j-k)2\pi i \frac{l}{M}}\right)_{j,k=0}^{N-1}.
\end{equation}

For $j\neq k$, the sum in the above determinant is zero and for $j=k$ it is $M$ so we find

\begin{equation}
Z_{N,M}=N! M^N,
\end{equation}

\noindent which implies that 

\begin{equation}
\E_{N,M}\prod_{j=1}^N f(z_j)=T_{N-1}(f).
\end{equation}

\end{proof}

\begin{remark}
Note that as we are interested in moments of $F_{N,M,L}^\beta$ and $F_{N,M}^\beta$, the symbols $f$ that we are interested in are special cases of the Fisher-Hartwig symbol: $f:\T\to \C$,

\begin{equation}\label{eq:fh}
f(z)=e^{V(z)}\prod_{j=1}^k |z-w_j|^{\beta_j},
\end{equation}

\noindent where $V$ is a smooth enough function, $(w_j)_{j=1}^k$ are fixed distinct points on the unit circle, and $\beta_j\geq 0$. In our case, $V$ is a Laurent polynomial, either $k=1$ or $k=2$, and either $\beta_j=0$ or $\beta_j=\beta$. We will also need the situation where $\beta_j=0$ for all $j$. When it is not restrictive, we will consider the general symbol \eqref{eq:fh}. More generally, one could also consider the phase of the characteristic polynomial (causing the symbol to have a jump), one could consider more general potentials $V$, or complex $\beta_j$.

\end{remark}

\subsection{Connection between discrete and continuum Toeplitz determinants with Fisher-Hartwig symbols}

We now present our main tool for analyzing discrete Toeplitz determinants with Fisher-Hartwig singularities - namely a generalization of an argument of \cite{bl} where it was noted that the discrete Toeplitz determinants with no singularities can be written as a product of a continuum one with essentially the same symbol, and a Fredholm determinant related to the continuum orthogonal polynomials and a function related to the discretization. Compared to their approach, our symbol does not have an analytic continuation into a neighborhood of the unit circle, but it does have a continuation with branch cuts along the rays going through the singularities. Our approach is to modify the argument of \cite{bl} to this case. We begin by recalling this continuation (for more information, see \cite{mfmls2,dik1}). 

\begin{lemma}\label{le:cont}
Let $V$ be a Laurent polynomial of the form 

\begin{equation}
V(z)=\sum_{j=-p}^p a_j z^j,
\end{equation}

\noindent let $(w_j)_{j=1}^k$ be distinct points on the unit circle, and for $|z|=1$, let

\begin{equation}
f(z)=e^{V(z)}\prod_{j=1}^k |z-w_j|^{\beta_j},
\end{equation}

\noindent with $\beta_j>0$. There is an analytic continuation of $f$ into $\C\setminus (\cup_{j=1}^k w_j \times (0,\infty))$ and it can be written as 

\begin{equation}
f(z)=\frac{D_i(f,z)}{D_e(f,z)},
\end{equation}

\noindent where $D_i(f,z)$ is the analytic continuation of 

\begin{equation}
z\mapsto e^{\frac{a_0}{2}+\sum_{j=1}^{p}a_j z^{j}}\prod_{l=1}^{k}(1-zw_l^{-1})^{\frac{\beta_l}{2}}
\end{equation}

\noindent from $\lbrace |z|<1\rbrace$ into the domain $\C\setminus (\cup_{j=1}^k w_j\times [1,\infty))$ (for $|z|<1$, the root is according to the principal branch).

Similarly, $D_e(f,z)$ is the analytic continuation of 

\begin{equation}
z\mapsto e^{-\frac{a_0}{2}-\sum_{j=1}^{p}a_{-j}z^{-j}}\prod_{l=1}^{k}(1-z^{-1}w_l)^{-\frac{\beta_l}{2}}
\end{equation}

\noindent from $\lbrace |z|>1\rbrace$ into the domain $\C\setminus (\cup_{j=1}^k w_j\times [0,1])$ (again in $|z|>1$ one takes the principal branch).
\end{lemma}

\begin{proof}
For $|z|\neq 1$, consider the function

\begin{equation}
D(f,z)=\exp\left(\frac{1}{4\pi}\int_0^{2\pi}\log f(e^{i\theta})\frac{e^{i\theta}+z}{e^{i\theta}-z}d\theta\right).
\end{equation}

For $|z|<1$ we have

\begin{align}
\notag \log D(f,z)&=\frac{1}{4\pi}\int_0^{2\pi}\left(\sum_{j=-p}^p a_j e^{ij\theta}+\sum_{l=1}^k \beta_l \log |e^{i\theta}-w_l|\right)\\
&\times(1+e^{-i\theta}z)\sum_{q=0}^\infty e^{-iq\theta}z^qd\theta.
\end{align}

Noting that for $j\in \Z\setminus \lbrace 0\rbrace$

\begin{equation}
\frac{1}{2\pi}\int_0^{2\pi}\log |e^{i\theta}-w| e^{ij\theta}d\theta=-\frac{1}{2|j|}w^j
\end{equation}

\noindent and $\int_0^{2\pi} \log |e^{i\theta}-w|d\theta=0$ so that we get 

\begin{equation}
\log D(f,z)=\frac{a_0}{2}+\sum_{j=1}^p a_j z^j-\sum_{l=1}^k \frac{\beta_l}{2} \sum_{j=1}^\infty \frac{1}{j}z^j w_l^{-j},
\end{equation}

\noindent or 

\begin{equation}
D(f,z)=e^{\frac{a_0}{2}+\sum_{j=1}^p a_j z^j+\sum_{l=1}^k \frac{\beta_l}{2} \log(1-z w_l^{-1})}=e^{\frac{a_0}{2}+\sum_{j=1}^p a_j z^j}\prod_{l=1}^k (1-z w_l^{-1})^{\frac{\beta_l}{2}},
\end{equation}

\noindent where the branch of $\log(1-z w_l^{-1})$ is chosen by the series expansion. This function is analytic in $|z|<1$ and extends analytically into $\C\setminus (\cup_{j=1}^k w_j\times [1,\infty))$. We will denote this continuation by $D_i(f,z)$.

\vspace{0.3cm}

Similarly for $|z|>1$ we have 

\begin{align}
\notag\log D(f,z)&=-\frac{1}{4\pi}\int_0^{2\pi}\left(\sum_{j=-p}^p a_j e^{ij\theta}+\sum_{l=1}^k \beta_l \log |e^{i\theta}-w_l|\right)\\
&\times(1+z^{-1}e^{i\theta})\sum_{q=0}^\infty e^{iq\theta}z^{-q}d\theta\\
\notag &=-\frac{a_0}{2}-\sum_{j=1}^p a_{-j}z^{-j}+\sum_{l=1}^k \frac{\beta_l}{2}\sum_{j=1}^\infty \frac{1}{j}z^{-j}w_l^j
\end{align}

\noindent or 

\begin{align}
\notag D(f,z)&=e^{-\frac{a_0}{2}-\sum_{j=1}^p a_{-j}z^{-j}-\sum_{l=1}^k\frac{\beta_l}{2} \log (1-z^{-1}w_l)}\\
&=e^{-\frac{a_0}{2}-\sum_{j=1}^p a_{-j}z^{-j}}\prod_{l=1}^k (1-z^{-1}w_l)^{-\frac{\beta_l}{2}},
\end{align}

\noindent where the branch is given by the series expansion. This is again analytic in $|z|>1$ and continues analytically into $\C\setminus(\cup_{j=1}^k w_j\times[0,1])$. We call the analytic continuation $D_e(f,z)$. Then

\begin{equation}
f(z)=\frac{D_i(f,z)}{D_e(f,z)}
\end{equation}

\noindent is analytic in $\C\setminus (\cup_{j=1}^k w_j\times[0,\infty))$ and agrees with the original symbol on the unit circle.

\end{proof}

We are now ready to adapt the proof of \cite{bl} to prove the connection between the discrete and continuum Toeplitz determinants with Fisher-Hartwig singularities. We restrict to the case where $V$ is real on the unit circle simply for the reason that this is the case we need.

\begin{proposition}\label{prop:detsing}
Let $V:\C\setminus \lbrace 0\rbrace\to \C$ be a Laurent polynomial of the form

\begin{equation}
V(z)=\sum_{j=0}^p (a_j z^j+\overline{a_j}z^{-j}),
\end{equation}

\noindent $(w_1,...,w_k)$ be distinct points on the unit circle such that $0\leq \arg(w_1)<\arg(w_2)<...<\arg(w_k)<2\pi$, and for $|z|=1$,

\begin{equation}
f(z)=e^{V(z)}\prod_{j=1}^k|z-w_j|^{\beta_j},
\end{equation}

\noindent where $\beta_j>0$. We also write $f(z)$ for the continuation of this function described in Lemma \ref{le:cont}.

\vspace{0.3cm}

Let $\epsilon\in(0,1)$, and let us write $\Gamma_j$ for the contour consisting of the segments $(1-\epsilon,1+\epsilon)\times w_j$, $(1-\epsilon,1+\epsilon)\times w_{j+1}$ and the (shorter) circular arcs of radius $1\pm \epsilon$ joining $(1\pm\epsilon)w_j$ and $(1\pm \epsilon)w_{j+1}$. For $j=k$, we understand $w_{k+1}$ as $w_1$. We also understand the contour to be oriented in the counter-clockwise direction.

\vspace{0.3cm}

Moreover, let 

\begin{equation}
v(z)=\begin{cases}
\frac{z^M}{1-z^M}, & |z|<1\\
\frac{z^{-M}}{1-z^{-M}}, & |z|>1
\end{cases}.
\end{equation}

\vspace{0.3cm}

Finally we denote by $\sigma$ the measure on $\cup_j \Gamma_j$ defined by

\begin{equation}
\sigma(dz)=\begin{cases}
\frac{dz}{2\pi iz}, & |z|=1\pm \epsilon\\
\frac{dz}{2\pi z}, & |z|\in(1-\epsilon,1+\epsilon)
\end{cases},
\end{equation}

\noindent and by $\oplus_{j=1}^k L^2(\Gamma_j,\sigma)$ the space of functions $h=(h_1,...,h_k)$, where $h_j:\Gamma_j\to \C$ are square integrable with respect to $\sigma$. 

\vspace{0.3cm}

Then we have

\begin{equation}
T_{N-1}(f)=\mathcal{T}_{N-1}(f)\det(1+K)_{\oplus_{j=1}^kL^2(\Gamma_j,\sigma)},
\end{equation}

\noindent where $K$ is a trace class integral operator on $\oplus_j L^2(\Gamma_j,\sigma)$ with kernel (acting on each component of the vector)

\begin{equation}
K(z,w)=I(z)\sqrt{v(z)f(z)}\sqrt{v(w)f(w)}K_{CD}(z,w),
\end{equation}

\noindent where the roots are according to the principal branch, when acting on the $j$th component the values of $f(z)$ and $f(w)$ are calculated as limits from the inside of $\Gamma_j$,

\begin{equation}
I(z)=\begin{cases}
1, & |z|=1\pm \epsilon\\
-i, & |z|\in(1-\epsilon,1+\epsilon)
\end{cases},
\end{equation}

\noindent and 

\begin{equation}
K_{CD}(z,w)=\frac{\frac{w^N}{z^N}\phi_N(z)\overline{\phi}_N(w^{-1})-\overline{\phi}_N(z^{-1})\phi_N(w)}{1-\frac{w}{z}}.
\end{equation}

\end{proposition}

\begin{remark}
Note that we can think of an element $h\in \oplus_jL^2(\Gamma_j,\sigma)$ as defining a two-valued function on $\cup_j \Gamma_j$ - or more precisely, it's single valued on $(1\pm \epsilon)\T$ and two-valued on $(1-\epsilon,1+\epsilon)w_j$.
\end{remark}

\begin{remark}\label{re:contour}
Keeping in mind the previous remark, an important difference between our case and the non-singular case considered in \cite{bl} is the different form of the integration contour. Instead of just circles of radius $1\pm \epsilon$, we also have the segments joining these and as there is a branch cut along each such segment, the oppositely oriented integrals don't cancel and we will need slightly more complicated asymptotics of the orthogonal polynomials to control our Fredholm determinant. If $\beta_j=0$ for all $j$, there are no branch cuts, the oppositely oriented integrals along the rays through $w_j$ cancel, and we recover the result of \cite{bl}.
\end{remark}

\begin{proof}
The proof is essentially a simple modification of that in \cite{bl} to take into account the branch cuts of $f$. To do this,  we being by noting that as $f(w_j)=0$, we have

\begin{equation}
T_{N-1}(f)=\det\left(\frac{1}{M}\sum_{z\in \mathcal{D}_M\setminus \lbrace w_1,...,w_k\rbrace}z^{-(l-m)}f(z)\right)_{l,m=0}^{N-1}.
\end{equation}

We then note that by Cauchy's integral formula, 

\begin{equation}
T_{N-1}(f)=\det\left(\sum_{j=1}^k\int_{\Gamma_j} z^{-(l-m)}f(z)\frac{\gamma_M'(z)}{M\gamma_M(z)}\frac{dz}{2\pi i}\right)_{l,m=0}^{N-1},
\end{equation}

\noindent where on $\Gamma_j$, $f$ is defined by a limit from the inside of the contour. 

\vspace{0.3cm}

One might need to take some care here as the contour $\Gamma_j$ goes through the points $w_j,w_{j+1}$ so if $w_j\in \mathcal{D}_M$, the integrand might have a singularity as $\gamma_M(w_j)=0$ in this case (so at most, a simple pole). But as also $f(w_j)=0$ (or more precisely, $f(z)=\mathcal{O}(|z-w_j|^{\beta_j})$ as $z\to w_j$) we see that this is an integrable singularity so one can indeed apply Cauchy's integral formula.

\vspace{0.3cm}

As determinants are unchanged by elementary row and column operations, this implies that 

\begin{align}
\notag T_{N-1}(f)&=\frac{1}{\prod_{j=0}^{N-1}\chi_j^2}\det\left(\sum_j\int_{ \Gamma_j} \overline{\phi}_l(z^{-1})\phi_m(z)f(z)\frac{\gamma_M'(z)}{M\gamma_M(z)}\frac{dz}{2\pi i}\right)_{l,m=0}^{N-1}\\
&=\mathcal{T}_{N-1}(f)\det\left(\sum_j\int_{\Gamma_j} \overline{\phi}_l(z^{-1})\phi_m(z)f(z)\frac{\gamma_M'(z)}{M\gamma_M(z)}\frac{dz}{2\pi i}\right)_{l,m=0}^{N-1},
\end{align}

\noindent where we made use of Proposition \ref{prop:toepprod}. Now using the orthonormality conditions, we write

\begin{align}
\notag\sum_j\int_{\Gamma_j} &\overline{\phi}_l(z^{-1})\phi_m(z)f(z)\frac{\gamma_M'(z)}{M\gamma_M(z)}\frac{dz}{2\pi i}\\
&=\delta_{l,m}-\int_{|z|=1}\overline{\phi}_l(z^{-1})\phi_m(z)f(z)\frac{dz}{2\pi iz}\\
\notag &+\sum_j\int_{\Gamma_j} \overline{\phi}_l(z^{-1})\phi_m(z)f(z)\frac{\gamma_M'(z)}{M\gamma_M(z)}\frac{dz}{2\pi i}.
\end{align}

As 

\begin{equation}
z\mapsto \overline{\phi}_l(z^{-1})\phi_m(z)f(z)\frac{1}{z}
\end{equation}

\noindent is analytic inside of $\Gamma_j$, we can deform the part of $\T$ that is inside of $\Gamma_j$ into the curve $\lbrace z\in \Gamma_j:|z|>1\rbrace$ with opposite orientation (by Cauchy's integral theorem). Thus we find

\begin{align}
\notag\sum_j\int_{\Gamma_j} &\overline{\phi}_l(z^{-1})\phi_m(z)f(z)\frac{\gamma_M'(z)}{M\gamma_M(z)}\frac{dz}{2\pi i}\\
&=\delta_{l,m}+\sum_j\int_{ \Gamma_j\cap \lbrace |z|<1\rbrace} \overline{\phi}_l(z^{-1})\phi_m(z)f(z)\frac{z\gamma_M'(z)}{M\gamma_M(z)}\frac{dz}{2\pi iz}\\
\notag &+\sum_j\int_{\Gamma_j\cap \lbrace |z|>1\rbrace} \overline{\phi}_l(z^{-1})\phi_m(z)f(z)\left(\frac{z\gamma_M'(z)}{M\gamma_M(z)}-1\right)\frac{dz}{2\pi i}\\
\notag &=\delta_{l,m}+\sum_j\int_{\Gamma_j} \overline{\phi}_l(z^{-1})\phi_m(z)f(z)v(z)\frac{dz}{2\pi iz }
\end{align}

Now we need to write the determinant as a Fredholm determinant. Let us introduce the following notation: let 

\begin{equation}
A:\oplus_{j=1}^kL^2\left(\Gamma_j,\sigma\right)\to \ell^2\left(\lbrace 0,...,N-1\rbrace\right),
\end{equation}

\begin{equation}
(A(h_1,...,h_k))_l=\sum_{j=1}^k\int_{\Gamma_j}\overline{\phi}_l(z^{-1})\sqrt{v(z)f(z)}h_j(z) \frac{dz}{2\pi i z}
\end{equation}

\noindent and 

\begin{equation}
B:\ell^2\left(\lbrace 0,...,N-1\rbrace\right)\to \oplus_jL^2\left(\Gamma_j,\sigma\right),
\end{equation}

\begin{equation}
B(a_0,...,a_{N-1})=\sum_{k=0}^{N-1}a_k \sqrt{vf}\phi_k,
\end{equation}

\noindent where both of the roots are according to the principal branch, and the $j$th component of this function is obtained by restricting to $\Gamma_j$ and defining $f$ as a limit from the inside of $\Gamma_j$.

Note that by Cauchy-Schwarz, since $\beta_j>0$ (or more precisely since $f(z)=\mathcal{O}(|z-w_j|^{\beta_j})$), $(Ah)_j$ is indeed finite for any $h\in \oplus_jL^2(\Gamma_j,\sigma)$ even if $v$ has a simple pole at some $w_j$ (which occurs when $w_j\in\mathcal{D}_M$). For the same reason, $B$ really maps into $\oplus_j L^2(\Gamma_j,\sigma)$.

\vspace{0.3cm}

Thus we have 

\begin{equation}
(AB)_{l,k}=\sum_j\int_{\Gamma_j}\overline{\phi}_l(z^{-1})\phi_k(z)v(z)f(z) \frac{dz}{2\pi i z}
\end{equation}

\noindent and 

\begin{equation}
T_{N-1}(f)=\mathcal{T}_{N-1}(f)\det(I+AB).
\end{equation}

\vspace{0.3cm}
We note that as $K=BA$ is an operator on $\oplus_jL^2(\Gamma_j,\sigma)$ and

\begin{equation}
Kh=\sum_{k=0}^{N-1}\sum_j\int_{\Gamma_j}\overline{\phi}_k(z^{-1}) \sqrt{v(z)f(z)} h_j(z)\frac{dz}{2\pi i z}\sqrt{vf}\phi_k, 
\end{equation}

\noindent (where the $j$th component is obtained by a limit from the inside of $\Gamma_j$) so we see that $K$ a finite rank operator, and in particular, trace class.

\vspace{0.3cm}

By Sylvester's determinant identity, $\det(I+AB)=\det(I+BA)$ ($I$ denotes the identity in the relevant space and the determinant is a Fredholm determinant).

\vspace{0.3cm}

Making use of the Christoffel-Darboux identity (Proposition \ref{prop:cd}), we find that the kernel of $K$ ($(Kh)(w)=\sum_j\int_{\Gamma_j} K(z,w)h_j(z)\sigma(dz))$ is 

\begin{equation}
K(z,w)=I(z)\sqrt{v(z)f(z)}\sqrt{v(w)f(w)}\frac{\frac{w^N}{z^N}\phi_N(z)\overline{\phi}_N(w^{-1})-\overline{\phi}_N(z^{-1})\phi_N(w)}{1-\frac{w}{z}}.
\end{equation}

\end{proof}

\begin{remark}
A similar result and argument also holds for a log-gas on a discrete subset of the real line, where one ends up with a discrete Hankel determinant, which can be related in a similar way then to a continuum Hankel determinant.
\end{remark}

\section{Asymptotics of orthogonal polynomials}

We will now review the basic results of the asymptotics of the orthogonal polynomials with respect to the Fisher-Hartwig weight on the unit circle (as well as the non-singular case) from \cite{mfmls1,mfmls2,ck}. We will consider the asymptotics on the contour relevant to us. We will start with the non-singular case studied in \cite{mfmls1}. Next we will consider the singular case with the restriction that the distance between singularities is bounded away from zero (studied e.g. in \cite{mfmls2}). Finally we will consider the situation where the singularities can merge, i.e. their distance can go to zero as $N\to\infty$. This case has been studied in \cite{ck}. As in some cases the exact estimates are not written down explicitly in these articles, we will review a small part of the relevant proof, to argue why the estimates we require hold.

\subsection{The non-singular case}

We will first review results from \cite{mfmls1} which we will use to analyze the asymptotics of the Fredholm determinant in the non-singular case. This has already been essentially done in \cite{bl}, but as in the proof of Proposition \ref{prop:var}, we wish to integrate the resulting (discrete) Toeplitz determinant, we'll need uniform asymptotics, so we will take some care.

\begin{theorem}[Mart\'inez-Finkelshtein, McLaughlin, and Saff]\label{th:polyasynosing}
Consider a Laurent polynomial $V$ of the form 

\begin{equation}
V(z)=\sum_{j=0}^p \frac{1}{2}(a_j z^j+\overline{a_j}z^{-j}),
\end{equation}

\noindent let $(\phi_j)_{j=0}^\infty$, with $\phi_j(z)=\chi_j z^j+\mathcal{O}(z^{j-1})$, be the orthonormal polynomials with respect to the weight $e^{V(z)}$ on the unit circle (as $e^{V(z)}>0$ on $\T$, the polynomials exist). Also fix any $R>0$ and $\epsilon\in(0,1)$. Then there exists a $c>0$ such that for $|z|=1-\epsilon$, 

\begin{equation}
\phi_N(z),\phi_N'(z)=\mathcal{O}(1)e^{-cN}
\end{equation}

\noindent as $N\to\infty$. Moreover, the $\mathcal{O}(1)$ term is uniform on $(1-\epsilon)\T$ and on the set 

\begin{equation}\label{eq:paramset}
A_R=\left\lbrace\max_{1\leq j\leq p}|a_j|\leq R\right\rbrace.
\end{equation}

For $|z|=1+\epsilon$, one has 

\begin{equation}
\begin{array}{ccc}
\phi_N(z)=\mathcal{O}(1)(1+\epsilon)^N & \mathrm{and} & \phi_N'(z)=\mathcal{O}(1) N(1+\epsilon)^N,
\end{array}
\end{equation}

\noindent where again the $\mathcal{O}(1)$ terms are uniform on $(1+\epsilon)\T$ and $A_R$.
\end{theorem}

\begin{remark}
The same results hold for the conjugate polynomials $\overline{\phi}_j(z)=\overline{\phi_j(\overline{z})}$.
\end{remark}

\begin{proof}
This is essentially an application of Theorem 1 and Corollary 2 in \cite{mfmls1}. We will not go into great detail here, but try to point out where in \cite{mfmls1} the reader can convince themselves about the validity of the required estimates. The only things that aren't immediately clear from the results in \cite{mfmls1} are the asymptotics of the derivatives and the uniformity on $A_R$.

\vspace{0.3cm}

For the derivatives, we note that (in the notation of \cite{mfmls1}) the derivative of $f_n^{(k)}$ has similar bounds as those in equation $(24)$ of \cite{mfmls1}. Then by uniform convergence, one can differentiate $\mathcal{E}_n$ and $\mathcal{I}_n$ term-wise, which gives the desired estimate for the derivatives. 

\vspace{0.3cm}

To have uniformity on $A_R$, the bounds $(24)$ in \cite{mfmls1} (and corresponding ones for the derivatives) imply that it is enough to have uniform bounds on the term $\Lambda$ defined in $(20)$ of \cite{mfmls1}. This in turn can be checked from the relevant definitions. Checking that the convergence of $\chi_{N}$ is uniform on $A_N$ is similar.

\end{proof}

\subsection{The singular case, when the distance between singularities is bounded away from zero} Again the relevant asymptotics have been studied in \cite{mfmls2}, but the uniformity in $V$ or the estimates for the derivatives are not stressed. Due to this, we'll again sketch some of the arguments for checking these facts.

\begin{theorem}[Mart\'inez-Finkelshtein, McLaughlin, and Saff]\label{th:singpolyasy1}
Let $V:\C\setminus \lbrace 0\rbrace\to \C$ be a Laurent polynomial of the form

\begin{equation}
V(z)=\sum_{j=0}^p\frac{1}{2} (a_j z^j+\overline{a_j}z^{-j}),
\end{equation}

\noindent $(w_1,...,w_k)$ be distinct points on the unit circle such that $0\leq \arg(w_1)<\arg(w_2)<...<\arg(w_k)<2\pi$, and assume that there exists some (smalle enough) fixed $\delta>0$ such that $|w_i-w_j|>\delta$ for $i\neq j$. For $|z|=1$, let

\begin{equation}
f(z)=e^{V(z)}\prod_{j=1}^k|z-w_j|^{\beta_j},
\end{equation}

\noindent where $\beta_j>0$. We also write $f(z)$ for the continuation of this function described in Lemma \ref{le:cont}. Finally let $(\phi_j)_j$ be the orthonormal polynomials with respect to the weight $f$ on the unit circle (as $f$ is non-negative and almost everywhere non-zero these exist) and write $\phi_j(z)=\chi_j z^{j}+\mathcal{O}(z^{j-1})$.

\vspace{0.3cm}

We then have

\begin{itemize}
\item[1)] For any fixed compact subset of the unit disk,

\begin{equation}
\begin{array}{ccc}
\phi_N(z)=\mathcal{O}(N^{-1}) & \mathrm{and} & \phi_N'(z)=\mathcal{O}(N^{-1})
\end{array}
\end{equation}

\noindent uniformly on the given compact set, as $N\to\infty$.

\item[2)] For any fixed compact subset of $\lbrace |z|>1\rbrace$,

\begin{equation}
\begin{array}{ccc}
\phi_N(z)=z^{N}\mathcal{O}(1) & \mathrm{and} & \phi_N'(z)=Nz^{N}\mathcal{O}(1)
\end{array}
\end{equation}

\noindent uniformly on the given compact set, as $N\to\infty$.

\item[3)] On the interval $[1-\delta/6,1+\delta/6]\times w_j$,

\begin{align}
\phi_N(z)&=\begin{cases}
\mathcal{O}(1)\frac{1}{\sqrt{f(z)}} (N|1-|z||)^{\frac{\beta_j}{2}}, & |1-|z||<\frac{1}{N}\\
\mathcal{O}(1) \frac{1}{\sqrt{f(z)}}, & |1-|z||>\frac{1}{N}
\end{cases}\\
&\notag =\begin{cases}
\mathcal{O}(1)N^{\frac{\beta_j}{2}}, & |1-|z||<\frac{1}{N}\\
\mathcal{O}(1) |1-|z||^{-\frac{\beta_j}{2}}, & |1-|z||>\frac{1}{N}
\end{cases}
\end{align}

\noindent and

\begin{align}
\phi_N'(z)&=\begin{cases}
\mathcal{O}(1) N \frac{1}{\sqrt{f(z)}} (N|1-|z||)^{\frac{\beta_j}{2}}, & |1-|z||<\frac{1}{N}\\
\mathcal{O}(1) \frac{1}{\sqrt{f(z)}}\frac{1}{|1-|z||}, & |1-|z||>\frac{1}{N}
\end{cases}\\
\notag & =\begin{cases}
\mathcal{O}(1)N^{\frac{\beta_j}{2}+1}, & |1-|z||<\frac{1}{N}\\
\mathcal{O}(1) |1-|z||^{-\frac{\beta_j}{2}-1}, & |1-|z||>\frac{1}{N}
\end{cases}
\end{align}

\noindent where the $\mathcal{O}(1)$ terms are uniform on the interval and $1/\sqrt{f(z)}$ is understood to be the limit of $1/\sqrt{f(ze^{i\epsilon})}$ as $\epsilon\to 0$ (and the root is according to the principal branch).

\end{itemize}

Moreover, if $A_R$ is defined as in \eqref{eq:paramset}, the above estimates are uniform on $A_R$ and they are uniform on $B_\delta=\lbrace|w_i-w_j|>\delta\ \mathrm{for} \ i\neq j\rbrace$.

\end{theorem}

\begin{remark}
Again corresponding results hold also for the conjugate polynomials $\overline{\phi}_j$.
\end{remark}

\begin{remark}
We have two different representation for the asymptotics of the polynomials on the rays  as in some cases we'll make use of the fact that there is the factor $\sqrt{f(z)}$ appearing in the kernel of $K$ so this will cancel with the factor of $1/\sqrt{f(z)}$ when analyzing the asymptotics of the kernel.
\end{remark}

\begin{proof}
While the asymptotics of the polynomials follow from Theorem 1, Theorem 2, and Theorem 3 in \cite{mfmls2}, the asymptotics of the derivatives are not written out explicitly in \cite{mfmls2} and the uniformity in $V$ is not stressed. So let us review parts of their argument.

\vspace{0.3cm}

Using (by now standard) Riemann-Hilbert arguments, it is proven in \cite{mfmls2} that in a compact subset of $\lbrace |z|<1\rbrace$, if we write $\phi_N=\chi_N \Phi_N$, then $\Phi_N(z)=Y_{11}(z)$ and $\chi_{N-1}^{2}=-Y_{21}(0)$ (our normalization of the weight differs from theirs by a factor of $2\pi$), where $Y$ is a $2\times 2$ matrix valued function which is analytic in $\C\setminus \T$ and in the compact subset it can be written as 

\begin{equation}
Y(z)=S(z)N(z),
\end{equation}

\noindent where $S(z)$ is a $2\times 2$ matrix valued analytic function in the given compact set, and satisfies 

\begin{equation}
\begin{array}{ccc}
S(z)=I+\mathcal{O}(N^{-1}) & \mathrm{and} & S'(z)=\mathcal{O}(N^{-1})
\end{array}
\end{equation}

\noindent uniformly on that compact set. Here $I$ is the $2\times 2$ identity matrix and the estimates are entry-wise (a convention we will follow through this note unless otherwise stated). Moreover, from the Neumann-series representation of $S$ (see Proposition 4 in \cite{mfmls2}) it follows that these terms are uniform on $A_R$ and $B_\delta$. More precisely, this involves obtaining uniform estimates for the jump matrix $J_S$ which in turn boils down to uniform estimates on the scattering function $\mathcal{S}$. These can be checked from its definition.

\vspace{0.3cm}

Inside the unit disk one has

\begin{equation}
N(z)=\left(D_i(f,z)D_i(f,0)\right)^{\sigma_3}\begin{pmatrix}
0 & 1\\
-1& 0
\end{pmatrix},
\end{equation}

\noindent where 

\begin{equation}
\sigma_3=\begin{pmatrix}
1 & 0\\
0 & -1
\end{pmatrix}.
\end{equation}

From this, we find

\begin{equation}
\Phi_N(z)=-(D_i(f,z)D_i(f,0))^{-1}S_{12}(z)=\mathcal{O}(N^{-1})
\end{equation}

\noindent and 

\begin{equation}
\Phi_N'(z)=\mathcal{O}(N^{-1})
\end{equation}

\noindent uniformly in the compact set, uniformly on $A_R$, and uniformly on $B_\delta$.

\vspace{0.3cm}

Moreover, 

\begin{equation}
\chi_{N-1}^{2}=S_{22}(0)D_i(f,0)^{-2}\to D_i(f,0)^{-2}
\end{equation}

\noindent uniformly in $A_R$ and $B_\delta$ as $N\to\infty$ (so in particular, it is uniformly bounded).

\vspace{0.3cm}

In a compact subset of $\lbrace |z|>1\rbrace$, one has the representation

\begin{equation}
Y(z)=S(z)N(z)z^{N\sigma_3},
\end{equation}

\noindent where again

\begin{equation}
\begin{array}{ccc}
S(z)=I+\mathcal{O}(N^{-1}) & \mathrm{and} & S'(z)=\mathcal{O}(N^{-1})
\end{array}
\end{equation}

\noindent uniformly in the compact set and uniformly on $A_R$ and $B_\delta$. For $|z|>1$,

\begin{equation}
N(z)=\left(D_i(f,0)D_e(f,z)\right)^{\sigma_3}.
\end{equation}
 
Thus 

\begin{equation}
\Phi_N(z)=S_{11}(z) D_i(f,0)D_e(f,z)z^{N}=\mathcal{O}(1)z^{N}
\end{equation} 

\noindent and 

\begin{equation}
\Phi_N'(z)=\mathcal{O}(1)Nz^{N}
\end{equation}

\noindent uniformly on the compact set and uniformly on $A_R$ and $B_\delta$.

\vspace{0.3cm}

In neighborhoods of the points $w_j$, the expression for $Y(z)$ is slightly more complicated. In particular, the functions in terms of which $Y$ is written have branch cuts (some of which cancel when taking the product). As $\Phi_N$ nevertheless is analytic, we can calculate $\Phi_N(z)$ on the interval $z\in[1-\delta/6,1+\delta/6]\times w_j$ as a limit of $\Phi_N(z')$ as $z'\to z$ in any way we wish. The way we shall calculate it is by writing $z'=|z| w_je^{i\epsilon}$ and then let $\epsilon\to 0^{+}$. For simplicity, we'll focus on the $|z|<1$ case. The $|z|>1$ case can be analyzed in a similar manner.

\vspace{0.3cm}

One still has $\Phi_N(z)=Y_{11}(z)$, but now, for $z'=|z|w_j e^{i\epsilon}$ for $\epsilon>0$ small enough, $Y$ is written as 

\begin{equation}
Y(z')=S(z')E(z')\Psi\left(\frac{\beta_j}{2},-i\frac{N}{2}\log \frac{z'}{w_j}\right)\left(e^{i \frac{\beta_j}{2}\pi}\frac{1}{\sqrt{f(z')}}(z')^{\frac{N}{2}}\right)^{\sigma_3}.
\end{equation} 

For such $z'$, $\Psi$ is of the form 

\begin{align}
\notag \Psi\left(\frac{\beta_j}{2},-i\frac{N}{2}\log \frac{z'}{w_j}\right)&=\begin{pmatrix}
\sqrt{\pi}\sqrt{-i\frac{N}{2}\log \frac{z'}{w_j}}I_{\frac{\beta_j}{2}+\frac{1}{2}}\left(-\frac{N}{2}\log \frac{z'}{w_j}\right) & *\\
-i\sqrt{\pi}\sqrt{-i\frac{N}{2}\log \frac{z'}{w_j}}I_{\frac{\beta_j}{2}-\frac{1}{2}}\left(-\frac{N}{2}\log \frac{z'}{w_j}\right) & *
\end{pmatrix}\\
&\times e^{-\frac{\beta_j}{4}\pi i \sigma_3.}
\end{align}

Here $I_\nu$ is a modified Bessel function of the first kind and we denote by $*$ the second column as it will be unimportant to us.

\vspace{0.3cm}

Moreover $E$ is analytic in the given neighborhood of $w_j$ and 

\begin{equation}
E(z)=\left(Q(z)\right)^{\sigma_3}\frac{1}{\sqrt{2}}\begin{pmatrix}
1 & i\\
-1 & -i
\end{pmatrix}
\end{equation}

\noindent where in the given neighborhood of $w_j$, $Q$ is a non-vanishing analytic function which is independent of $N$ and can be bounded uniformly on $A_R$ and $B_\delta$ (as can its derivatives).

\vspace{0.3cm}

Finally $S$ is as before, i.e. it is analytic in this neighborhood of $w_j$, and it satisfies

\begin{equation}
\begin{array}{ccc}
S(z)=I+\mathcal{O}(N^{-1}) & \mathrm{and} & S'(z)=\mathcal{O}(N^{-1})
\end{array}
\end{equation}

\noindent uniformly in the neighborhood and uniformly on $A_R$ and $B_\delta$.

Writing out the definitions, we find that for $|z|<1$

\begin{align}
\Phi_N(z')&=\sqrt{\frac{\pi}{2}}\left(Q(z')S_{11}(z')-Q(z')^{-1}S_{12}(z')\right)e^{i\frac{\beta_j}{4}\pi}\frac{1}{\sqrt{f(z')}}(z')^{\frac{N}{2}}\\
&\times\sqrt{-i\frac{N}{2}\log \frac{z'}{w_j}}\left(I_{\frac{\beta_j}{2}+\frac{1}{2}}\left(-\frac{N}{2}\log \frac{z'}{w_j}\right)+I_{\frac{\beta_j}{2}-\frac{1}{2}}\left(-\frac{N}{2}\log \frac{z'}{w_j}\right)\right).
\end{align}

To simplify notation a bit, we note that 

\begin{equation}
z'\mapsto \sqrt{\frac{\pi}{2}}\left(Q(z')S_{11}(z')-Q(z')^{-1}S_{12}(z')\right)e^{i\frac{\beta_j}{4}\pi}
\end{equation}

\noindent is $\mathcal{O}(1)$ (uniformly on $A_R$ and $B_\delta$) as is its derivative. We thus see that if we write 

\begin{equation}
\zeta_\epsilon=-\frac{N}{2}\log \frac{z'}{w_j},
\end{equation}

\noindent we have 

\begin{equation}
\Phi_N(z')=\mathcal{O}(1)\frac{1}{\sqrt{f(z')}}e^{-\zeta_\epsilon}\sqrt{\zeta_\epsilon}\left(I_{\frac{\beta_j}{2}+\frac{1}{2}}(\zeta_\epsilon)+I_{\frac{\beta_j}{2}-\frac{1}{2}}(\zeta_\epsilon)\right),
\end{equation}

\noindent where the $\mathcal{O}(1)$ term and its derivative are uniform on $A_R$ and $B_\delta$. We then take the $\epsilon\to 0$ limit of this and find (for $\zeta=-N(\log |z|)/2>0$)

\begin{equation}
\Phi_N(z)=\mathcal{O}(1)\frac{1}{\sqrt{f(z)}}e^{-\zeta}\sqrt{\zeta}\left(I_{\frac{\beta_j}{2}+\frac{1}{2}}(\zeta)+I_{\frac{\beta_j}{2}-\frac{1}{2}}(\zeta)\right),
\end{equation}

\noindent where the expression $1/\sqrt{f(z)}$ is defined precisely as this limit.

\vspace{0.3cm}

We now make use of the asymptotics of $I_\nu$: for $0<x<1$,

\begin{equation}
I_\nu(x)=\frac{1}{\Gamma(\nu+1)}\left(\frac{x}{2}\right)^\nu+\mathcal{O}(x^{\nu+1})
\end{equation}

\noindent and for $x>1$,

\begin{equation}\label{eq:iasybig}
\sqrt{x}e^{-x}I_\nu(x)=1+\mathcal{O}(x^{-1}).
\end{equation}

Thus if $-N\log |z|<2$, (i.e. $\zeta\in(0,1)$)

\begin{equation}
\Phi_N(z)=\mathcal{O}(1)\frac{1}{\sqrt{f(z)}}(-N\log |z|)^{\frac{\beta_j}{2}}=\mathcal{O}(1)\frac{1}{\sqrt{f(z)}}(N|1-|z||)^{\frac{\beta_j}{2}}.
\end{equation}

Whereas for $-N\log |z|>2$, we have 

\begin{equation}
\Phi_N(z)=\mathcal{O}(1)\frac{1}{\sqrt{f(z)}}.
\end{equation}

As both of the expressions agree (up to a uniform $\mathcal{O}(1)$ factor) when $|1-|z||=\mathcal{O}(N^{-1})$, the difference between the conditions $-N\log |z|<2$ and say $N|1-|z||<1$ is immaterial so we have proven the claim about the asymptotics of the polynomials. Let us now consider the derivatives.

\vspace{0.3cm}

\noindent For this, we use the fact that the $\mathcal{O}(1)$ terms had uniformly $\mathcal{O}(1)$ derivatives and that $I_\nu'(x)=\frac{1}{2}(I_{\nu+1}(x)+I_{\nu-1}(x))$. We also note that for $\zeta=-\frac{N}{2}\log \frac{z}{w_j}$, 

\begin{equation}
\frac{d}{dz}=-\frac{N}{2}\frac{1}{z}\frac{d}{d\zeta}.
\end{equation}

If we write $\mathcal{I}(\zeta)=I_{\frac{\beta_j}{2}+\frac{1}{2}}(\zeta)+I_{\frac{\beta_j}{2}-\frac{1}{2}}(\zeta)$, we conclude that (with a similar limiting argument) for $|z|<1$ on our interval

\begin{align}\label{eq:phider}
\notag \Phi_N'(z)&=\mathcal{O}(1)\frac{1}{\sqrt{f(z)}}\sqrt{\zeta}e^{-\zeta}\\
&\times\left(\mathcal{O}(1)\mathcal{I}(\zeta)-\frac{1}{2}\frac{f'(z)}{f(z)}\mathcal{I}(\zeta)-\frac{N}{4z\zeta}\mathcal{I}(\zeta)+\frac{N}{2z}\mathcal{I}(\zeta)-\frac{N}{2z}\mathcal{I}'(\zeta)\right).
\end{align}

We then note that (uniformly)

\begin{equation}
\frac{f'(z)}{f(z)}=\mathcal{O}(1)+\frac{\beta_j}{z-w_j}.
\end{equation}

Putting things together (and using the small $x$ asymptotics of $I_\nu$) we see that for $\zeta\in(0,1)$, (i.e. roughly $N|1-|z||<1$)

\begin{align}
\notag \Phi_N'(z)&=\mathcal{O}(1)\frac{1}{\sqrt{f(z)}}\sqrt{\zeta}e^{-\zeta}\\
&\times \Bigg[\mathcal{O}\left(\zeta^{\frac{\beta_j}{2}-\frac{1}{2}}\right)N-\frac{\beta_j}{2}\frac{1}{z-w_j}\frac{1}{\Gamma(\frac{\beta_j}{2}+\frac{1}{2})}\left(\frac{\zeta}{2}\right)^{\frac{\beta_j}{2}-\frac{1}{2}}\\
\notag &-\frac{N}{8z}\frac{1}{\Gamma(\frac{\beta_j}{2}+\frac{1}{2})}\left(\frac{\zeta}{2}\right)^{\frac{\beta_j}{2}-\frac{3}{2}}-\frac{N}{4z}\frac{1}{\Gamma(\frac{\beta_j}{2}-\frac{1}{2})}\left(\frac{\zeta}{2}\right)^{\frac{\beta_j}{2}-\frac{3}{2}}\Bigg].
\end{align}

\noindent where all of the estimates are uniform. We then note that 

\begin{align}
\notag\frac{1}{\Gamma(\frac{\beta_j}{2}+\frac{1}{2})}&\left(\frac{\zeta}{2}\right)^{\frac{\beta_j}{2}-\frac{3}{2}}\left(-\frac{\beta_j}{4}\frac{\zeta}{z-w_j}-\frac{N}{8z}-\frac{N}{4z}\left(\frac{\beta_j}{2}-\frac{1}{2}\right)\right)\\
&=N\frac{\beta_j}{8\Gamma(\frac{\beta_j}{2}+\frac{1}{2})}\left(\frac{\zeta}{2}\right)^{\frac{\beta_j}{2}-\frac{3}{2}}w_j^{-1}\left(\frac{\log \frac{z}{w_j}}{\frac{z}{w_j}-1}-\frac{w_j}{z}\right)
\end{align}

\noindent and that for $x\in(1-\delta/6,1)$, 

\begin{equation}
\frac{\log x}{x-1}-\frac{1}{x}=\mathcal{O}(x-1).
\end{equation}

\vspace{0.3cm}

Thus

\begin{equation}
\frac{1}{\Gamma(\frac{\beta_j}{2}+\frac{1}{2})}\left(\frac{\zeta}{2}\right)^{\frac{\beta_j}{2}-\frac{3}{2}}\left(-\frac{\beta_j}{4}\frac{\zeta}{z-w_j}-\frac{N}{8z}-\frac{N}{4z}\left(\frac{\beta_j}{2}-\frac{1}{2}\right)\right)=\mathcal{O}\left(\zeta^{\frac{\beta_j}{2}-\frac{1}{2}}\right)
\end{equation}

\noindent (uniformly) and we conclude that for $N|1-|z||<1$, 

\begin{equation}
\Phi_N'(z)=\mathcal{O}(1)\frac{1}{\sqrt{f(z)}}N \zeta^{\frac{\beta_j}{2}}=\mathcal{O}(1)\frac{1}{\sqrt{f(z)}}N (N|1-|z||)^{\frac{\beta_j}{2}.}
\end{equation}

\vspace{0.3cm}

For $N|1-|z||>1$, we have by \eqref{eq:phider} and \eqref{eq:iasybig}

\begin{equation}
\Phi_N'(z)=\mathcal{O}(1)\frac{1}{\sqrt{f(z)}}\left(\frac{1}{||z|-1|}+N\mathcal{O}(\zeta^{-1})\right)=\mathcal{O}(1)\frac{1}{\sqrt{f(z)}}|1-|z||^{-1}
\end{equation}

\noindent uniformly on $A_R$ and $B_\delta$.

\end{proof}

\subsection{Merging singularities} The analysis in the case of the merging singularities becomes very involved quite rapidly, so we will focus on the simplest case which is sufficient for us - namely we consider $V=0$, $k=2$, and $\beta_1=\beta_2=\beta>0$. In \cite{ck}, a more general case (still with two singularities) is considered, but we restrict to this case. Let us now review the implications of the analysis in \cite{ck} for the asymptotics of the orthonormal polynomials with respect to such a weight. 

\begin{theorem}[Claeys and Krasovsky]\label{th:singpolyasy2}
Fix a small $t_0>0$. Then for $t\in(0,t_0)$ and $\beta>0$, let $f_t:\T\to \R$,

\begin{equation}
f_t(z)=|z-e^{it}|^\beta|z-e^{-it}|^\beta.
\end{equation}

Let $(\phi_j)_{j=0}^\infty$, $\phi_j(z)=\chi_j z^j+\mathcal{O}(z^{j-1})$ be the orthonormal polynomials with respect to the weight $f_t$ on the unit circle. We then have the following asymptotic behavior.

\vspace{0.3cm}

For any fixed compact subset of the unit disk, as $N\to\infty$

\begin{equation}
\phi_N(z),\phi_N'(z)=\mathcal{O}(1)N^{-1}
\end{equation}

\noindent uniformly in $z$ and $0<t<t_0$.

\vspace{0.3cm}

For any fixed compact subset of $\lbrace |z|>1\rbrace$,

\begin{equation}
\begin{array}{ccc}
\phi_N(z)=z^{N}\mathcal{O}(1), & \mathrm{and} & \phi_N'(z)=Nz^{N}\mathcal{O}(1)
\end{array}
\end{equation}

\noindent uniformly in $z$ and $0<t<t_0$. 

\vspace{0.3cm}

For any fixed $\gamma>0$ and  $z\in[1-\gamma,1+\gamma]\times e^{\pm it}$, the asymptotics depend on the behavior of $t$. For certain fixed $c>0$ small enough and $C>0$ large enough, we have the following behavior:

\vspace{0.3cm}

\underline{For $Nt<c$}

\begin{align}
\phi_N(z)&=\begin{cases}
\mathcal{O}(1)\frac{1}{\sqrt{f_t(z)}}N^\beta|1-|z||^{\frac{\beta}{2}}\max(|1-|z||^{\frac{\beta}{2}},t^{\frac{\beta}{2}}), & N|1-|z||<1\\
\mathcal{O}(1) \frac{1}{\sqrt{f_t(z)}}, & N|1-|z||>1
\end{cases}\\
&\notag =\begin{cases}
\mathcal{O}(1)N^\beta, & N|1-|z||<1\\
\mathcal{O}(1) |1-|z||^{-\beta}, & N|1-|z||>1
\end{cases}
\end{align}

\noindent and 

\begin{align}
\phi_N'(z)&=\begin{cases}
\mathcal{O}(1)\frac{1}{\sqrt{f_t(z)}}N^{\beta+1}|1-|z||^{\frac{\beta}{2}}\max(|1-|z||^{\frac{\beta}{2}},t^{\frac{\beta}{2}}), & N|1-|z||<1\\
\mathcal{O}(1) N\frac{1}{\sqrt{f_t(z)}}, & N|1-|z||>1
\end{cases}\\
&\notag =\begin{cases}
\mathcal{O}(1)N^{\beta+1}, & N|1-|z||<1\\
\mathcal{O}(1) N|1-|z||^{-\beta}, & N|1-|z||>1
\end{cases}
\end{align}

\noindent uniformly in $z$ and $0<t<cN^{-1}$.

\vspace{0.3cm}

\underline{For $cN^{-1}<t<CN^{-1}$}

\begin{align}
\phi_N(z)&=\begin{cases}
\mathcal{O}(1)\frac{1}{\sqrt{f_t(z)}}(N|1-|z||)^{\frac{\beta}{2}}, & N|1-|z||<1\\
\mathcal{O}(1)\frac{1}{\sqrt{f_t(z)}}, & N|1-|z||>1
\end{cases}\\
\notag & =\begin{cases}
\mathcal{O}(1)N^\beta, & N|1-|z||<1\\
\mathcal{O}(1)|1-|z||^{-\beta}, & N|1-|z||>1
\end{cases}
\end{align}

\noindent and 

\begin{align}
\phi_N'(z)&=\begin{cases}
\mathcal{O}(1)\frac{1}{\sqrt{f_t(z)}}N(N|1-|z||)^{\frac{\beta}{2}}, & N|1-|z||<1\\
\mathcal{O}(1)\frac{1}{\sqrt{f_t(z)}}N, & N|1-|z||>1
\end{cases}\\
\notag &=\begin{cases}
\mathcal{O}(1)N^{\beta+1}, & N|1-|z||<1\\
\mathcal{O}(1)N|1-|z||^{-\beta}, & N|1-|z||>1
\end{cases}
\end{align}

\noindent uniformly in $z$ and $cN^{-1}<t<CN^{-1}$.

\vspace{0.3cm}

\underline{$t>CN^{-1}$}

\begin{align}
\phi_N(z)&=\begin{cases}
\mathcal{O}(1)\frac{1}{\sqrt{f_t(z)}}(N|1-|z||)^{\frac{\beta}{2}}, & N|1-|z||<1\\
\mathcal{O}(1)\frac{1}{\sqrt{f_t(z)}}, & N|1-|z||>1
\end{cases}\\
&\notag =\begin{cases}
\mathcal{O}(1)(Nt^{-1})^{\frac{\beta}{2}}, & N|1-|z||<1\\
\mathcal{O}(1)|1-|z||^{-\frac{\beta}{2}}\min(|1-|z||^{-\frac{\beta}{2}},t^{-\frac{\beta}{2}}), & N|1-|z||>1
\end{cases}
\end{align}

\noindent and 

\begin{align}
\phi_N'(z)&=\begin{cases}
\mathcal{O}(1)\frac{1}{\sqrt{f_t(z)}}N(N|1-|z||)^{\frac{\beta}{2}}, & N|1-|z||<1\\
\mathcal{O}(1)\frac{1}{\sqrt{f_t(z)}}|1-|z||^{-1}, & |1-|z||\in(N^{-1},t)\\
\mathcal{O}(1)\frac{1}{\sqrt{f_t(z)}} t^{-1}, & |1-|z||>t 
\end{cases}\\
\notag &= \begin{cases}
\mathcal{O}(1) N(Nt^{-1})^{\frac{\beta}{2}}, & N|1-|z||<1\\
\mathcal{O}(1) |1-|z||^{-\frac{\beta}{2}-1}t^{-\frac{\beta}{2}} , & |1-|z||\in(N^{-1},t)\\
\mathcal{O}(1) t^{-1}|1-|z||^{-\beta}, & |1-|z||>t
\end{cases}
\end{align}

\noindent uniformly in $z$ and $CN^{-1}<t<t_0$.

\end{theorem}

\begin{remark}
The orthogonal polynomials with respect to the weight $|z-w_1|^{\beta}|z-w_2|^\beta$ with $|\mathrm{arg}w_1-\mathrm{arg}w_2|=2t$ are obtained from the ones with weight $f_t$ by simply rotating the argument - in particular, the same asymptotics hold near $w_i$.
\end{remark}

\begin{proof}
In \cite{ck}, it is not the asymptotics of the orthogonal polynomials but the Toeplitz determinant that was the main focus, so to check the required estimates requires a fair amount of work from us.

\vspace{0.3cm}

The proof for the compact subset of the unit disk is roughly the same as in the case where the distance of the singularities is bounded away from zero. In \cite{ck}, it is proven that there exists a function $\Upsilon$ (corresponding to $S$ in the previous case) which is analytic in $\C$ apart from a jump contour which consists of the boundary of a fixed disk $\mathcal{U}$ containing the points $e^{\pm it}$ for $t<t_0$ (but not intersecting the given compact subset) and part of a lens on the unit circle going from $e^{it}$ to $e^{-it}$ in the counter-clockwise direction - also not intersecting the compact subset (similar to the case when the distance of the singularities is bounded, except here we only have a single lens instead of one going also from $e^{-it}$ to $e^{it}$ in the counter-clockwise direction. This function satisfies $\Upsilon(z)=I+\mathcal{O}(N^{-1})$ and $\Upsilon'(z)=\mathcal{O}(N^{-1})$ uniformly off of the jump contour. 

\vspace{0.3cm}

It follows from the analysis in \cite{ck} that one then has for $\lbrace |z|<1:z\notin\overline{\mathcal{U}}\rbrace$ (here we write $\tau=D_i(f,0)^{-1}$)

\begin{equation}
\Phi_N(z)=\frac{\phi_N(z)}{\chi_N}=\left(\Upsilon(z)\left(\frac{D_i(f_t,z)}{\tau}\right)^{\sigma_3}\begin{pmatrix}
0 & 1\\
-1 & 0
\end{pmatrix}\right)_{11}
\end{equation}

\noindent and $\chi_{N-1}$ is given by in terms of the $(2,1)$ entry evaluated at zero. From this one finds the claim for the compact subset of the unit disk.

\vspace{0.3cm}

In the case of the compact subset of $\lbrace |z|>1\rbrace$, the argument is again analogous to the situation where the distance between the singularities is bounded.
\vspace{0.3cm}

The part for $z\in[1-\gamma,1+\gamma]\times e^{\pm i t}$ is more complicated. We now choose $\mathcal{U}$ so that it contains these segments. For simplicity, let us consider the situation where $z=(1-r)\times e^{it}$ for $r>0$. The other cases can be treated in a similar manner. Moreover, we calculate things for such a $z$ by taking the $\epsilon\to 0^{+}$ limit of $(1-r)\times e^{i(t-\epsilon)}$. In this region, we can write 

\begin{equation}
\Phi_N(z)=\left(\Upsilon(z)P(z)\right)_{11},
\end{equation}

\noindent where $\Upsilon$ again satisfies the same conditions and 

\begin{align}
\notag P(z)&=\begin{pmatrix}
0 & 1\\
1 & 0
\end{pmatrix}\left(\frac{D_i(f_t,z)D_e(f_t,z)}{\tau^{2}}\right)^{-\frac{1}{2}\sigma_3}\Psi\left(\frac{1}{t}\log z, -2iN t\right)\\
&\times\left(-z^{\frac{N}{2}\sigma_3}f_t(z)^{-\sigma_3/2}\sigma_3\right),
\end{align}

\noindent and $\Psi$ is defined by the following Riemann-Hilbert problem (in \cite{ck}, it is shown that there is a unique solution to the problem in our setting). 

\begin{remark}
This $\Psi$ is no longer the same as in the previous case, but we keep the notation to be  consistent with the original articles as our arguments rely heavily on the results in the  original articles.
\end{remark}

\begin{definition}\label{def:psi}
For each $s\in(-i \R_+)$ (i.e. purely imaginary complex number with strictly negative imaginary part), let $\zeta\mapsto \Psi(\zeta,s)$ be the unique solution to the following Riemann-Hilbert problem:

\vspace{0.3cm}

Find a function $\Psi=\Psi(\zeta,s)$ such that

\begin{itemize}

\item $\Psi:\C\setminus \Gamma\to \C^{2\times 2}$ is analytic, where 

\begin{equation}
\begin{array}{lll}
\Gamma=\cup_{k=1}^{5}\Gamma_k, & \Gamma_1=i+e^{i\frac{\pi}{4}}\R_+, & \Gamma_2=i +e^{i\frac{3\pi}{4}}\R_+,\\
\Gamma_3=-i+e^{i\frac{5\pi}{4}}\R_+, & \Gamma_4=-i+e^{i\frac{7\pi}{4}}\R_+, & \Gamma_5=[-i,i]. \\
\end{array}
\end{equation}

\item $\Psi$ satisfies the jump conditions: for $\zeta\in \Gamma_k$,

\begin{equation}
\Psi_+(\zeta,s)=\Psi_-(\zeta,s)J_k,
\end{equation}

\noindent where $\Psi_+$ ($\Psi_-$) denotes the limit of $\Psi$ from the left (right) of the contour (the arrows in Figure \ref{figure: Gamma} determine the orientation of the curves), and 

\begin{equation}
\begin{array}{ll}
 J_1=\begin{pmatrix}1&e^{\pi i\beta}\\0&1\end{pmatrix}, & J_2=\begin{pmatrix}1&0\\-e^{-\pi i\beta}&1\end{pmatrix}\\
J_3=\begin{pmatrix}1&0\\-e^{\pi i\beta}&1\end{pmatrix},&  J_4=\begin{pmatrix}1&e^{-\pi i\beta}\\0&1\end{pmatrix}\\
J_5=\begin{pmatrix}0&1\\-1&1\end{pmatrix}. &
\end{array}
\end{equation}
\item In all regions, 

\begin{equation}\label{eq:psiasyexp}
\Psi(\zeta,s)=\left(I+\Psi_1(s)\zeta^{-1}+\Psi_2(s)\zeta^{-2}+\mathcal{O}(|\zeta|^{-3})\right)e^{-\frac{is}{4}\zeta \sigma_3}
\end{equation}

\noindent as $\zeta\to \infty$.

\item For $\beta\notin \Z_+$, let 

\begin{equation}
g=-\frac{1}{2i \sin(\pi \beta)}(e^{\pi i \beta}-1), \quad h=-\frac{1}{2i \sin(\pi \beta)}(1-e^{-i \pi \beta})=g e^{-i\pi\beta},
\end{equation}

\begin{equation}
G_{III}=\begin{pmatrix}
1 & g\\
0 & 1
\end{pmatrix}, \quad G_I=G_{III}J_5^{-1}, \quad G_{II}=G_1 J_1,
\end{equation}

\noindent and

\begin{equation}
H_{III}=\begin{pmatrix}
1 & h\\
0 & 1
\end{pmatrix}, \quad H_{IV}=H_{III}J_3^{-1},\quad H_I=H_{IV}J_4^{-1}.
\end{equation}

Then define $F_1=F_1(\zeta,s)$ in a neighborhood of $i$ by 

\begin{equation}
\Psi(\zeta,s)=F_1(\zeta,s)(\zeta-i)^{\frac{\beta}{2}\sigma_3} G_j
\end{equation}

\noindent in regions $j=I,II,III$, where the branch cut of $(\zeta-i)^{\frac{\beta}{2}\sigma_3}$ is along $i+e^{\frac{3\pi i }{4}}(0,\infty)$ and $\mathrm{arg}(\zeta-i)\in(-5\pi/4, 3\pi /4)$, and $F_2=F_2(\zeta,s)$ in a neighborhood of $-i$ by 

\begin{equation}
\Psi(\zeta,s)=F_2(\zeta,s)(\zeta+i)^{\frac{\beta}{2}\sigma_3} H_j
\end{equation}

\noindent in regions $j=I,III,IV$, where the branch cut now is along $-i+e^{\frac{5\pi i}{4}}(0,\infty)$ and $\mathrm{arg}(\zeta+i)\in(-3\pi/4,5\pi/4)$.

\vspace{0.3cm}

If $\beta\in\Z_+$, define $G_{III}=H_{III}=I$ and the other matrices through the jump matrices as before, and define in the region $j$, $F_1$ and $F_2$ (with similar branch cuts) by 

\begin{equation}
\Psi(\zeta,s)=F_1(\zeta,s)(\zeta-i)^{\frac{\beta}{2}\sigma_3}\begin{pmatrix}
1 & \frac{1-e^{\pi i \beta}}{2\pi i e^{\pi i \beta}}\log(\zeta-i)\\
0 & 1
\end{pmatrix}G_j
\end{equation}

\noindent and

\begin{equation}
\Psi(\zeta,s)=F_2(\zeta,s)(\zeta+i)^{\frac{\beta}{2}\sigma_3}\begin{pmatrix}
1 & \frac{e^{-\pi i \beta}-1}{2\pi i e^{-\pi i \beta}}\log(\zeta+i)\\
0 & 1
\end{pmatrix}H_j.
\end{equation}

\vspace{0.3cm}

Then these functions $F_1$ and $F_2$  must be analytic functions of $\zeta$ in some neighborhoods of $\pm i$.

\end{itemize}
\end{definition}

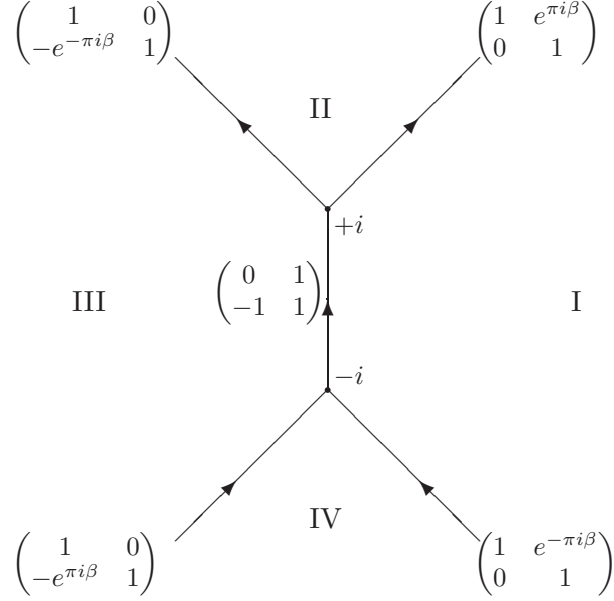
\begin{figure}[t]\label{fig:psijumps}
\begin{center}
    \setlength{\unitlength}{0.8truemm}
    \begin{picture}(110,100)(-10,-7.5)
    \put(50,60){\thicklines\circle*{.8}}
    \put(50,30){\thicklines\circle*{.8}}
    \put(51,56){\small $+i$}
    \put(51,31){\small $-i$}
    \put(50,60){\thicklines\circle*{.8}}
    \put(50,30){\thicklines\circle*{.8}}
    \put(50,60){\line(1,1){25}}
    \put(50,30){\line(1,-1){25}}
    \put(50,60){\line(-1,1){25}}
    \put(50,30){\line(-1,-1){25}}
    \put(50,30){\line(0,1){30}}
    \put(65,75){\thicklines\vector(1,1){.0001}}
    \put(65,15){\thicklines\vector(-1,1){.0001}}
    \put(35,75){\thicklines\vector(-1,1){.0001}}
    \put(50,45){\thicklines\vector(0,1){.0001}}
    \put(35,15){\thicklines\vector(1,1){.0001}}

    \put(74,88){\small $\begin{pmatrix}1&e^{\pi i\beta}\\0&1\end{pmatrix}$}
    \put(-2,88){\small $\begin{pmatrix}1&0\\-e^{-\pi i\beta}&1\end{pmatrix}$}
    \put(31,45){\small $\begin{pmatrix}0&1\\-1&1\end{pmatrix}$}
    \put(-2,0){\small $\begin{pmatrix}1&0\\-e^{\pi i\beta}&1\end{pmatrix}$}
    \put(74,0){\small $\begin{pmatrix}1&e^{-\pi i\beta}\\0&1\end{pmatrix}$}
    \put(8,43){III}
    \put(90,43){I}
    \put(47,7){IV}
    \put(47,75){II}
    \end{picture}
    \caption{The jump contour and jump matrices for $\Psi$ (a modification of Figure 1 in \cite{ck}).}
    \label{figure: Gamma}
\end{center}
\end{figure}

We thus find

\begin{align}
\notag \Phi_N(z)&=-\Upsilon_{11}(z)\left(\frac{D_i(f_t,z)D_e(f_t,z)}{\tau^{2} }\right)^{\frac{1}{2}}\frac{z^{\frac{N}{2}}}{\sqrt{f_t(z)}}\Psi_{21}\left(\frac{1}{t}\log z, -2iN t\right)\\
&+\Upsilon_{12}(z)\left(\frac{\tau^{2}}{D_i(f_t,z)D_e(f_t,z)}\right)^{\frac{1}{2}}\frac{z^{\frac{N}{2}}}{\sqrt{f_t(z)}}\Psi_{11}\left(\frac{1}{t}\log z, -2iN t\right).
\end{align}

We note that if we write $s=-2iNt$ and $\zeta=\frac{1}{t}\log z$, then  

\begin{equation}
z^{\frac{N}{2}}=e^{\frac{i}{4}s\zeta}=e^{\frac{1}{4}|s|\zeta}
\end{equation}

\noindent so the relevant question is the asymptotics of the first column of

\begin{equation}
\Psi(\zeta,s)e^{\frac{1}{4}|s|\zeta\sigma_3},
\end{equation}

\noindent where $s$ is purely imaginary with a negative imaginary part and as we consider $z=(1-r)e^{it}$, $\zeta=i+\frac{1}{t}\log |z|$. We need to study these asymptotics in the different regimes of $\zeta$ and $s$ ($\zeta$ close to $i$ or far away from $i$; $s$ small, bounded, or large). This can be done by making use of results in \cite{ck}. 

\vspace{0.3cm}

\begin{remark}
For notational simplicity, we will only go over things for non-integer $\beta$. Integer $\beta$ would complicate some formulas slightly, but can be done in a similar manner. 
\end{remark}

\emph{Bounded $s$:} Let us first focus on the case where $s$ is bounded, i.e. we assume that there exist some constants $0<c<C<\infty$ (we think of $c$ being small and $C$ large) and consider $c<Nt<C$ so that $s=-2iNt$ is bounded. Let us now assume that $\zeta$ is in a small neighborhood of $i$ so we have 

\begin{equation}
\Psi(\zeta,s)e^{\frac{|s|\zeta}{4}\sigma_3}=F_1(\zeta,s)(\zeta-i)^{\frac{\beta}{2}\sigma_3}G_{III}e^{\frac{|s|\zeta}{4}\sigma_3}.
\end{equation}

As $G_{III}$ is upper triangular with bounded entries and $F_1(\zeta,s)$ is bounded in the neighborhood of $i$ (as $s$ is bounded), we see that for $\zeta$ in the fixed neighborhood of $i$,

\begin{equation}
\Psi(\zeta,s)e^{\frac{|s|\zeta}{4}\sigma_3}=\begin{pmatrix}
\mathcal{O}(1)(\zeta-i)^{\frac{\beta}{2}} & *\\
\mathcal{O}(1)(\zeta-i)^{\frac{\beta}{2}} & *
\end{pmatrix}
\end{equation}

\noindent where the $\mathcal{O}(1)$ terms are uniform throughout the neighborhood and uniform in $c<Nt<C$. 
\vspace{0.3cm}

If on the other hand $\zeta$ is outside of the given neighborhood (and as we are outside the fixed neighborhood of $-i$), we have by the asymptotic expansion of $\Psi$ - \eqref{eq:psiasyexp} - that 

\begin{equation}
\Psi(\zeta,s)e^{\frac{|s|}{4}\zeta\sigma_3}=\mathcal{O}(1)
\end{equation}

\noindent uniformly in the relevant parameters.

\vspace{0.3cm}

Let us now note that for $z\in[1-\gamma,1+\gamma]e^{it}$, $\zeta-i=(\log|z|)/t$. Thus $\zeta$ in a small neighborhood of $i$ means roughly (recall $t\sim N^{-1}$) $N|1-|z||<1$. We conclude that 

\begin{equation}
\Phi_N(z)=\begin{cases}
\mathcal{O}(1)\frac{1}{\sqrt{f_t(z)}}(N|1-|z||)^{\frac{\beta}{2}}, & N|1-|z||<1\\
\mathcal{O}(1)\frac{1}{\sqrt{f_t(z)}}, & N|1-|z||>1
\end{cases}
\end{equation} 

\noindent uniformly in $z$ and $Nt\in(c,C)$.

\vspace{0.3cm}

Again the relevant question is to analyze the first column of 

\begin{equation}
\frac{d}{dz}\frac{1}{\sqrt{f_t(z)}}\Psi(\zeta,s)e^{\frac{|s|}{4}\zeta\sigma_3}.
\end{equation}

Let us start with the case where $\zeta$ is close to $i$ (so $|1-|z||<N^{-1}$). Let us make use of the representation $\Psi(\zeta,s)=F_1(\zeta,s)(\zeta-i)^{\frac{\beta}{2}\sigma_3}G_{III}$. We have (denoting by $F_1'(\zeta,s)$ the derivative of $F_1(\zeta,s)$ with respect to $\zeta$)

\begin{align}\label{eq:psider}
\notag \frac{d}{dz}\frac{1}{\sqrt{f_t(z)}}\Psi(\zeta,s)e^{\frac{|s|}{4}\zeta\sigma_3}&=\frac{1}{\sqrt{f_t(z)}}\bigg(-\frac{1}{2}\frac{f_t'(z)}{f_t(z)}F_1(\zeta,s)(\zeta-i)^{\frac{\beta}{2}\sigma_3}G_{III}e^{\frac{|s|}{4}\zeta\sigma_3}\\
&+\frac{1}{zt}F_1'(\zeta,s)(\zeta-i)^{\frac{\beta}{2}\sigma_3}G_{III}e^{\frac{|s|}{4}\zeta\sigma_3}\\
\notag &+ \frac{1}{zt}\frac{1}{\zeta-i}\frac{\beta}{2}F_1(\zeta,s)\sigma_3(\zeta-i)^{\frac{\beta}{2}\sigma_3}G_{III}e^{\frac{|s|}{4}\zeta\sigma_3}\\
\notag &+\frac{1}{zt}\frac{|s|}{4}F_1(\zeta,s)(\zeta-i)^{\frac{\beta}{2}\sigma_3}G_{III}e^{\frac{|s|}{4}\zeta\sigma_3}\sigma_3\bigg).
\end{align}

For the first term, we note that 

\begin{equation}
f_t(z)=(1-ze^{it})^{\frac{\beta}{2}}(1-ze^{-it})^{\frac{\beta}{2}}(1-z^{-1}e^{it})^{\frac{\beta}{2}}(1-z^{-1}e^{-it})^{\frac{\beta}{2}}
\end{equation}

\noindent so that 

\begin{align}\label{eq:fder}
\frac{f_t'(z)}{f_t(z)}&=\frac{\beta }{z-e^{it}}+\frac{\beta }{z-e^{-it}}+\mathcal{O}(1).
\end{align}

We then point out that 

\begin{equation}
\left(F_1(\zeta,s)(\zeta-i)^{\frac{\beta}{2}\sigma_3}G_{III}e^{\frac{|s|}{4}\zeta\sigma_3}\right)_{j1}=\left(F_1(\zeta,s)\sigma_3(\zeta-i)^{\frac{\beta}{2}\sigma_3}G_{III}e^{\frac{|s|}{4}\zeta\sigma_3}\right)_{j1},
\end{equation}

\noindent so we can combine the first and third term of \eqref{eq:psider} into 

\begin{equation}
\mathcal{O}(1)\frac{1}{\sqrt{f_t(z)}}(\zeta-i)^{\frac{\beta}{2}}\left(-\frac{\beta}{2}\frac{1}{z-e^{it}}+\frac{\beta}{z-e^{-it}}+\mathcal{O}(1)+\frac{\beta}{2}\frac{1}{zt}\frac{1}{\zeta-i}\right).
\end{equation}

Writing $1/(\zeta-i)=t/(\log (z/e^{it}))$ we see that this can be written as 

\begin{align}
\notag \mathcal{O}(1)&\frac{1}{\sqrt{f_t(z)}}(\zeta-i)^{\frac{\beta}{2}}\frac{\beta}{2}e^{-it}\left(-\frac{1}{|z|-1}+\mathcal{O}(1)t^{-1}+\frac{1}{|z|\log |z|}\right)\\
&=\mathcal{O}(1)\frac{1}{\sqrt{f_t(z)}}N(N|1-|z||)^{\frac{\beta}{2}}
\end{align} 

\noindent with $\mathcal{O}(1)$ being uniform in $z$ and $Nt\in(c,C)$. The remaining two terms in \eqref{eq:psider} can be estimated to be of the same order so we find that for $N|1-|z||<1$, 

\begin{equation}
\Psi_N'(z)=\mathcal{O}(1)\frac{1}{\sqrt{f_t(z)}}N(N|1-|z||)^{\frac{\beta}{2}}.
\end{equation}

For $|1-|z||>N^{-1}$, we note that it is proven in \cite{ck} that 

\begin{equation}
\frac{d}{d\zeta}\Psi(\zeta,s)=A(\zeta,s)\Psi(\zeta,s),
\end{equation}

\noindent where for bounded $s$, $A$ is (uniformly) $\mathcal{O}(1)$ outside the fixed neighborhoods of $\pm i$. Thus 

\begin{align}
\notag\frac{d}{d\zeta}(\Psi(\zeta,s)e^{\frac{|s|}{4}\zeta\sigma_3})&=A(\zeta,s)\Psi(\zeta,s)e^{\frac{|s|}{4}\zeta\sigma_3}+\Psi(\zeta,s)e^{\frac{|s|}{4}\zeta\sigma_3}\frac{|s|}{4}\sigma_3\\
&=\mathcal{O}(1),
\end{align}

\noindent and

\begin{align}
\notag\frac{d}{dz}\frac{1}{\sqrt{f_t(z)}}\Psi(\zeta,s)e^{\frac{|s|}{4}\zeta\sigma_3}&=\frac{1}{|1-|z||}\frac{1}{\sqrt{f_t(z)}}\mathcal{O}(1)+\frac{1}{t}\frac{1}{\sqrt{f_t(z)}}\mathcal{O}(1)\\
&=\mathcal{O}(1)N\frac{1}{\sqrt{f_t(z)}}.
\end{align}

\vspace{0.3cm}

\emph{Large $s$:} Consider now the situation where $Nt>C$. For this, we make use of the $s\to-i\infty$ asymptotics of $\Psi(\zeta,s)$ which have been studied in section 5 of \cite{ck}. For us, the central result of this section is the following: in the region we are interested in (i.e. for $z=|z|e^{it}$ with $|z|<1$ and small perturbations of the angle),

\begin{equation}
\Psi(\zeta,s)e^{\frac{|s|}{4}\zeta\sigma_3}=\begin{cases}
R(\zeta,s)P_1(\zeta,s), & |\zeta-i|<1\\
R(\zeta,s), & |\zeta-i|>1
\end{cases}
\end{equation}

\noindent where $R(\zeta,s)=\mathcal{O}(1)$ and $\frac{d}{d\zeta}R(\zeta,s)=\mathcal{O}(1)$ uniformly in $\zeta$ off of a certain jump contour which we don't need to worry about and uniformly in the relevant domain of $s$. 

\begin{remark}
Here in the condition $|\zeta-i|<1$, the upper bound should be replaced by some small enough $\mathcal{O}(1)$ number, but this only changes our results by a uniform $\mathcal{O}(1)$ factor, so for notational simplicity, we stick to this condition.
\end{remark}

Here the function $P_1$ is defined in the following manner:

\begin{equation}\label{eq:mref}
P_1(\zeta,s)=e^{-i\frac{|s|}{4}\sigma_3}e^{\frac{\pi i }{4}\beta\sigma_3}M\left(\frac{|s|}{2}(\zeta-i)\right)e^{-\frac{\pi i}{4}\beta\sigma_3}e^{\frac{|s|}{4}\zeta\sigma_3},
\end{equation}

\noindent where in the region relevant to us, 

\begin{equation}
M(\lambda)=L(\lambda) \lambda^{\frac{\beta}{2}\sigma_3}\widetilde{G}_3,
\end{equation}

\noindent where $L$ is a function which is analytic in a neighborhood of zero and depends only on $\beta$. The branch of the power is chosen so that $\mathrm{arg}\lambda\in(0,2\pi)$. $\widetilde{G}_3$ is an upper triangular matrix with ones on the diagonal. Another fact we need about $M$ is that as outside of a neighborhood of the origin,

\begin{equation}\label{eq:masy}
M(\lambda)=(1+\mathcal{O}(\lambda^{-1}))e^{-\frac{1}{2}\lambda\sigma_3}.
\end{equation} 

Let us consider first the situation $N|1-|z||<1$. This translates into (up to a uniform $\mathcal{O}(1)$ factor) $|s||(\zeta-i)|<1$ so as $|s|$ is large we have 

\begin{align}
\notag \Psi(\zeta,s)e^{\frac{|s|}{4}\zeta\sigma_3}&=\mathcal{O}(1)(|s|(\zeta-i))^{\frac{\beta}{2}\sigma_3}\widetilde{G}_3e^{\frac{|s|}{4}\zeta\sigma_3}e^{-\frac{\pi i}{4}\beta\sigma_3}\\
&=\mathcal{O}(1)(N|1-|z||)^{\frac{\beta}{2}\sigma_3}\widetilde{G}_3e^{\frac{|s|}{4}\zeta\sigma_3}e^{-\frac{\pi i}{4}\beta\sigma_3}\\
\notag &=\begin{pmatrix}
\mathcal{O}(1) (N|1-|z||)^{\frac{\beta}{2}}e^{\frac{|s|}{4}\zeta} & *\\
\mathcal{O}(1) (N|1-|z||)^{\frac{\beta}{2}}e^{\frac{|s|}{4}\zeta} & *
\end{pmatrix}.
\end{align}

We then note that 

\begin{equation}
e^{\frac{|s|}{4}\zeta}=e^{\frac{|s|}{4}(\zeta-i)}e^{\frac{|s|}{4}i}=\mathcal{O}(1)
\end{equation}

\noindent for the $\zeta$ we are considering now. We conclude that for $N|1-|z||<1$, 

\begin{align}
\Phi_N(z)&=\mathcal{O}(1)\frac{1}{\sqrt{f_t(z)}}(N|1-|z||)^{\frac{\beta}{2}}
\end{align}

\noindent uniformly in the relevant variables.

\vspace{0.3cm}

Consider next the situation where $N|1-|z||>1$, but $|\zeta-i|<1$ (these conditions are roughly equivalent to $N^{-1}<|1-|z||<t$). In this region we use \eqref{eq:masy} for estimating $M$. We find

\begin{equation}
\Psi(\zeta,s)e^{\frac{|s|}{4}\zeta\sigma_3}=\mathcal{O}(1)e^{-\frac{|s|}{4}(\zeta-i)}e^{-\frac{\pi i }{4}\beta\sigma_3}e^{\frac{|s|}{4}\zeta\sigma_3}=\mathcal{O}(1)
\end{equation}

\noindent and conclude that 

\begin{equation}
\Phi_N(z)=\mathcal{O}(1)\frac{1}{\sqrt{f_t(z)}}
\end{equation}

\noindent uniformly in the relevant parameters.

\vspace{0.3cm}

Consider finally the case where $|\zeta-i|>1$ (or $|1-|z||>t$). Here we have also $\Psi(\zeta,s)e^{\frac{|s|}{4}\zeta\sigma_3}=\mathcal{O}(1)$ so again

\begin{equation}
\Phi_N(z)=\mathcal{O}(1)\frac{1}{\sqrt{f_t(z)}}
\end{equation}

\noindent uniformly in the relevant parameters. 

\vspace{0.3cm}

For the derivative, we note that in the $||z|-1|<N^{-1}$ regime, the argument of the bounded $s$ case goes over in an analogous manner (so differentiating picks up a factor of $|s|/t=2N$) and we have in this regime (uniformly)

\begin{equation}
\Phi_N'(z)=\mathcal{O}(1)\frac{1}{\sqrt{f_t(z)}}N(N|1-|z||)^{\frac{\beta}{2}}.
\end{equation}

\vspace{0.3cm}

In the regime $N^{-1}<|1-|z||<t$, $\frac{d}{d\zeta}(\Psi(\zeta,s)e^{\frac{|s|}{4}\zeta\sigma_3})=\mathcal{O}(1)$ so one finds with similar arguments

\begin{align}
\notag \Phi_N'(z)&=\mathcal{O}(1)\frac{1}{\sqrt{f_t(z)}}\left(-\frac{1}{2}\frac{f_t'(z)}{f_t(z)}+\mathcal{O}(1)t^{-1}\right)\\
&=\mathcal{O}(1)\frac{1}{\sqrt{f_t(z)}}|1-|z||^{-1}.
\end{align}

\vspace{0.3cm}

Finally for $|1-|z||>t$ we have again with similar arguments for estimating $f_t(z)$ and its derivatives

\begin{equation}
\Phi_N'(z)=\mathcal{O}(1)\frac{1}{\sqrt{f_t(z)}}t^{-1}.
\end{equation}

Again all of these estimates were uniform.

\vspace{0.3cm}

\emph{Small $s$:} Finally we consider the case where $Nt<c$. For this case, it is proven in \cite{ck} (section 6) that if we write $\lambda=|s|(\zeta-i)/2=N\log |z|$, then

\begin{equation}
\Psi(\zeta,s)=\begin{cases}
H(\lambda,s)P_0(\lambda,s), & N|1-|z||<1\\
H(\lambda,s)M(\lambda), & N|1-|z||>1
\end{cases}
\end{equation}

\noindent where $M$ is as in \eqref{eq:mref} except that $\beta/2$ is replaced by $\beta$. In the region relevant to us, $P_0$ can be written as 

\begin{equation}
P_0(\lambda,s)=L(\lambda)\begin{pmatrix}
1 & c_0 J(\lambda;s)\\
0 & 1
\end{pmatrix}\lambda^{\frac{\beta}{2}\sigma_3}(\lambda-s)^{\frac{\beta}{2}\sigma_3} e^{\beta\pi i \sigma_3}\widetilde{G}_3,
\end{equation}

\noindent with a suitable convention for the branch cuts (these aren't important for us as we are estimating only magnitude). The function $L$ is analytic in the relevant neighborhood of the origin (in fact it's essentially the same function as before). The function $J$ can be defined as 

\begin{equation}
J(\lambda;s)=\frac{1}{\pi i}\int_s^0\frac{|\xi|^{\beta}|\xi-s|^\beta}{\xi-\lambda}d\xi.
\end{equation}

Also $H$ and its derivative are $\mathcal{O}(1)$ uniformly off of a certain jump contour (whose details are not relevant for us).

\begin{remark}
Again the neighborhood of the origin is not necessarily exactly $N|1-|z||<1$, but up to a $\mathcal{O}(1)$ factor, we can pretend it is.
\end{remark}

Consider then the $N|1-|z||<1$ situation. Arguing as before (e.g. the $(1,2)$ entries of the upper triangular matrices don't affect the first column of $\Psi$), we find

\begin{equation}
\Psi(\zeta,s)e^{\frac{|s|}{4}\zeta\sigma_3}=\begin{pmatrix}
\mathcal{O}(1)(N\log |z|)^{\frac{\beta}{2}}(N\log |z|-s)^{\frac{\beta}{2}} & *\\
\mathcal{O}(1)(N\log |z|)^{\frac{\beta}{2}}(N\log |z|-s)^{\frac{\beta}{2}} & *
\end{pmatrix}e^{\frac{|s|}{4}(\zeta-i)\sigma_3}e^{i\frac{|s|}{4}\sigma_3}.
\end{equation}

We first note that $N|1-|z||<1$ is roughly equivalent to $|s||\zeta-i|<1$ so the last two terms don't affect anything. We then note that 

\begin{equation}
(N\log |z|-s)^{\frac{\beta}{2}}=\mathcal{O}(1)N^{\frac{\beta}{2}}(\log |z|-2it)^{\frac{\beta}{2}}
\end{equation}

\noindent so we see that 
\begin{equation}
\Phi_N(z)=\mathcal{O}(1)\frac{1}{\sqrt{f_t(z)}}N^\beta|1-|z||^{\frac{\beta}{2}}\max(|1-|z||^{\frac{\beta}{2}},t^{\frac{\beta}{2}}).
\end{equation}

For $N|1-|z||>1$ we recall \eqref{eq:masy} and find (uniformly)

\begin{align}
\notag \Psi(\zeta,s)e^{\frac{|s|}{4}\zeta\sigma_3}&=\mathcal{O}(1)e^{-\frac{1}{2}N\log |z|\sigma_3}e^{\frac{|s|}{4}(i+\frac{1}{t}\log |z|)\sigma_3}\\
&=\mathcal{O}(1).
\end{align}

Thus for $N|1-|z||>1$

\begin{equation}
\Phi_N(z)=\mathcal{O}(1)\frac{1}{\sqrt{f_t(z)}}
\end{equation}

\noindent uniformly.

\vspace{0.3cm}

Let us now consider the derivative. Returning to the notation $\lambda=|s|(\zeta-i)/2=N\log |z|$, we recall that we can write (for $N|1-|z||<1$ or roughly equivalently $|\lambda|<1$)

\begin{equation}
\Psi(\zeta,s)e^{\frac{|s|}{4}\zeta\sigma_3}=\begin{pmatrix}
\mathcal{O}(1)\lambda^{\frac{\beta}{2}}(\lambda-s)^{\frac{\beta}{2}} & *\\
\mathcal{O}(1)\lambda^{\frac{\beta}{2}}(\lambda-s)^{\frac{\beta}{2}}  & *
\end{pmatrix},
\end{equation}

\noindent where the $\mathcal{O}(1)$ also have uniformly $\mathcal{O}(1)$ $\lambda$-derivatives. We then note that 

\begin{equation}
\frac{d}{dz}=\frac{1}{tz}\frac{d}{d\zeta}=\frac{1}{tz}\frac{|s|}{2}\frac{d}{d\lambda}=\frac{N}{z}\frac{d}{d\lambda}.
\end{equation}

We thus have 

\begin{align}
\notag& \frac{d}{dz}\frac{1}{\sqrt{f_t(z)}}\Psi(\zeta,s)e^{\frac{|s|}{4}\zeta\sigma_3}\\
&=\frac{1}{\sqrt{f_t(z)}}\begin{pmatrix}
\mathcal{O}(1)\lambda^{\frac{\beta}{2}}(\lambda-s)^{\frac{\beta}{2}}\left(-\frac{1}{2}\frac{f_t'(z)}{f_t(z)}+\frac{N}{z}\frac{\beta}{2}(\lambda^{-1}+(\lambda-s)^{-1})\right) & *\\
\mathcal{O}(1)\lambda^{\frac{\beta}{2}}(\lambda-s)^{\frac{\beta}{2}}\left(-\frac{1}{2}\frac{f_t'(z)}{f_t(z)}+\frac{N}{z}\frac{\beta}{2}(\lambda^{-1}+(\lambda-s)^{-1})\right) &*
\end{pmatrix}\\
&\notag +\frac{N}{\sqrt{f_t(z)}}\begin{pmatrix}
\mathcal{O}(1)\lambda^{\frac{\beta}{2}}(\lambda-s)^{\frac{\beta}{2}} & *\\
\mathcal{O}(1)\lambda^{\frac{\beta}{2}}(\lambda-s)^{\frac{\beta}{2}} & *
\end{pmatrix}.
\end{align}

Recalling \eqref{eq:fder}, we have

\vspace{0.3cm}

\begin{align}
&\notag -\frac{1}{2}\frac{f_t'(z)}{f_t(z)}+\frac{N}{z}\frac{\beta}{2}(\lambda^{-1}+(\lambda-s)^{-1})\\
\notag &=-\frac{\beta}{2}\frac{1}{z-e^{it}}-\frac{\beta}{2}\frac{1}{z-e^{-it}}+\frac{\beta N}{2z}\frac{1}{N\log |z|}+\frac{\beta N}{2 z}\frac{1}{N\log |z|+2iN t}+\mathcal{O}(1)\\
&=\frac{\beta e^{-it}}{2}\left(\frac{1}{1-|z|}+\frac{1}{|z|\log |z|}\right)+\frac{\beta}{2}e^{it}\left(-\frac{1}{ze^{it}-1}+\frac{1}{ze^{it}}\frac{1}{\log(ze^{it})}\right)+\mathcal{O}(1)\\
\notag &=\mathcal{O}(1)
\end{align}

\noindent uniformly in the relevant parameters (here we used the fact that $((1-w)^{-1}+1/(w\log w)$ is bounded for $w$ close enough to $1$. We conclude that the relevant asymptotics of $\Phi_N'(z)$ for $N|1-|z||<1$ are 

\begin{align}
\notag \Phi_N'(z)&=\mathcal{O}(1)N \frac{1}{\sqrt{f_t(z)}}\lambda^{\frac{\beta}{2}}(\lambda-s)^{\frac{\beta}{2}}\\
&=\mathcal{O}(1)\frac{1}{\sqrt{f_t(z)}}N^{\beta+1}|1-|z||^{\frac{\beta}{2}}(\log |z|+2it)^{\frac{\beta}{2}}\\
&\notag =\mathcal{O}(1)\frac{1}{\sqrt{f_t(z)}}N^{\beta+1}|1-|z||^{\frac{\beta}{2}}\max(|1-|z||^{\frac{\beta}{2}},t^{\frac{\beta}{2}})
\end{align}

\noindent uniformly.

\vspace{0.3cm}

For $N||z|-1|>1$ we had

\begin{equation}
\Psi(\zeta,s)e^{\frac{|s|}{4}\zeta\sigma_3}=\mathcal{O}(1)
\end{equation}

\noindent uniformly as is its $\lambda$-derivative so we find 
\begin{align}
\notag \frac{d}{dz}\frac{1}{\sqrt{f_t(z)}}\Psi(\zeta,s)e^{\frac{|s|}{4}\zeta\sigma_3}&=\frac{1}{\sqrt{f_t(z)}}\left(\mathcal{O}(1)N+\mathcal{O}(1)\frac{f_t'(z)}{f_t(z)}\right)\\
&=\mathcal{O}(1)N\frac{1}{\sqrt{f_t(z)}}.
\end{align}

We conclude that for $N|1-|z||>1$,

\begin{equation}
\Phi_N'(z)=\mathcal{O}(1)N\frac{1}{\sqrt{f_t(z)}}.
\end{equation}

\end{proof}

We will also need a result on the asymptotics of $\phi_N(z)$ and $\phi_N'(z)$ in the case where $Nt<c$, but $z$ is not on the contour, but in the sector between $(1-\epsilon,1+\epsilon)e^{\pm it}$. These asymptotics also follow directly from the analysis in \cite{ck}. More precisely,

\begin{lemma}\label{le:offcont}
For $Nt<c$, where $c$ is as in the previous theorem, and $z=re^{i\theta}$, where $r\in(1-\gamma,1+\gamma)$ and $\theta\in(-t,t)$, we have: for $N|1-|z||<1$

\begin{align}
\notag \Phi_N(z)&=\mathcal{O}(1)\frac{1}{\sqrt{f_t(z)}}N^{\beta}\max(|1-|z||^{\frac{\beta}{2}},(t-\theta)^{\frac{\beta}{2}})\max(|1-|z||^{\frac{\beta}{2}},(t+\theta)^{\frac{\beta}{2}})\\
&=\mathcal{O}(1) N^\beta
\end{align}

\noindent and 

\begin{align}
\notag \Phi_N'(z)&=\mathcal{O}(1)\frac{N^{\beta+1}}{\sqrt{f_t(z)}}\max(|1-|z||^{\frac{\beta}{2}},(t-\theta)^{\frac{\beta}{2}})\max(|1-|z||^{\frac{\beta}{2}},(t+\theta)^{\frac{\beta}{2}})\\
&=\mathcal{O}(1)N^{\beta+1}
\end{align}

For $N||z|-1|>1$,

\begin{align}
\notag \Phi_N(z)&=\mathcal{O}(1)\frac{1}{\sqrt{f_t(z)}}\\
&=\mathcal{O}(1)|1-|z||^{-\beta}
\end{align}

\noindent and 

\begin{align}
\notag \Phi_N'(z)&=\mathcal{O}(1)N\frac{1}{\sqrt{f_t(z)}}\\
&=\mathcal{O}(1) N|1-|z||^{-\beta}.
\end{align}

\noindent uniformly in $z$ and $t<cN^{-1}$.

\end{lemma}

\begin{proof}
As in the previous proof, let us write $\zeta=\log z/t$ and $s=-2iNt$. We also consider for simplicity the $|z|<1$ case. Again, the relevant quantity is the first column of 

\begin{equation}
\frac{1}{\sqrt{f_t(z)}}\Psi(\zeta,s)e^{\frac{|s|}{4}\zeta\sigma_3}
\end{equation}

\noindent and the $z$-derivative of this. Let us consider first the case where $z$ is close to $e^{\pm it}$, or $N|1-|z||<1$. We find as before (with $\lambda=|s|(\zeta-i)/2$) that 

\begin{equation}
\Psi(\zeta,s)e^{\frac{|s|}{4}\zeta\sigma_3}=\begin{pmatrix}
\mathcal{O}(1)\lambda^{\frac{\beta}{2}}(\lambda-s)^{\frac{\beta}{2}} & *\\
\mathcal{O}(1)\lambda^{\frac{\beta}{2}}(\lambda-s)^{\frac{\beta}{2}} & *
\end{pmatrix}.
\end{equation}

We note that in our case, 

\begin{align}
\lambda&=\frac{|s|}{2}(\zeta-i)=N\log z-iNt
\end{align}

\noindent and 

\begin{equation}
\lambda-s=N\log z+iNt.
\end{equation}

We conclude that 

\begin{align}
\notag \lambda^{\frac{\beta}{2}}(\lambda-s)^{\frac{\beta}{2}}&=\mathcal{O}(1)N^{\beta}\max(|1-|z||^{\frac{\beta}{2}},(t-\theta)^{\frac{\beta}{2}})\max(|1-|z||^{\frac{\beta}{2}},(t+\theta)^{\frac{\beta}{2}}).
\end{align}

Noting that in our case

\begin{align}
\notag \frac{1}{\sqrt{f_t(z)}}&=\mathcal{O}(1)|z-e^{it}|^{-\frac{\beta}{2}}|z-e^{-it}|^{-\frac{\beta}{2}}\\
&=\mathcal{O}(1)\min(|1-|z||^{-\frac{\beta}{2}},(t-\theta)^{-\frac{\beta}{2}})\min(|1-|z||^{-\frac{\beta}{2}},(t+\theta)^{-\frac{\beta}{2}})
\end{align}

\noindent so our claim about the asymptotics of $\Phi_N$ is immediate.

\vspace{0.3cm}

For $N|1-|z||>1$, we have 

\begin{equation}
\Phi_N(z)=\mathcal{O}(1)\frac{1}{\sqrt{f_t(z)}}
\end{equation}

\vspace{0.3cm}

For the derivative in the $N|1-|z||<1$ case, we saw that 

\begin{align}
\notag& \frac{d}{dz}\frac{1}{\sqrt{f_t(z)}}\Psi(\zeta,s)e^{\frac{|s|}{4}\zeta\sigma_3}\\
&=\frac{1}{\sqrt{f_t(z)}}\begin{pmatrix}
\mathcal{O}(1)\lambda^{\frac{\beta}{2}}(\lambda-s)^{\frac{\beta}{2}}\left(-\frac{1}{2}\frac{f_t'(z)}{f_t(z)}+\frac{N}{z}\frac{\beta}{2}(\lambda^{-1}+(\lambda-s)^{-1})\right) & *\\
\mathcal{O}(1)\lambda^{\frac{\beta}{2}}(\lambda-s)^{\frac{\beta}{2}}\left(-\frac{1}{2}\frac{f_t'(z)}{f_t(z)}+\frac{N}{z}\frac{\beta}{2}(\lambda^{-1}+(\lambda-s)^{-1})\right) &*
\end{pmatrix}\\
&\notag +\frac{N}{\sqrt{f_t(z)}}\begin{pmatrix}
\mathcal{O}(1)\lambda^{\frac{\beta}{2}}(\lambda-s)^{\frac{\beta}{2}} & *\\
\mathcal{O}(1)\lambda^{\frac{\beta}{2}}(\lambda-s)^{\frac{\beta}{2}} & *
\end{pmatrix}.
\end{align}

As earlier, one sees that the second term is dominant and differentiating picks up a factor of $N$.

\vspace{0.3cm}

Consider now the $N||z|-1|>1$ situation. In this case we have as before 

\begin{equation}
\Phi_N'(z)=\mathcal{O}(1)N \frac{1}{\sqrt{f_t(z)}}.
\end{equation}

\end{proof}

\section{Asymptotics of the Fredholm determinant}

Our approach for analyzing the asymptotics of the Fredholm determinant is based on the following estimate (see \cite{simon} Theorem 6.5 - set $B=0$ and $A=K$).

\begin{theorem}[Simon]\label{th:simon}
For a trace class operator $K$,
\begin{equation}
|\det(I+K)e^{-\mathrm{tr}K}-1|\leq \Vert K\Vert_2 e^{\Gamma_2(\Vert K\Vert_2+1)^2}
\end{equation}

\noindent for a suitable positive constant $\Gamma_2$, where the determinant is a Fredholm determinant and $\Vert K\Vert_2$ denotes the Hilbert-Schmidt norm of $K$.
\end{theorem}

Thus to estimate $\det(I+K)$, we must be able to estimate $||K||_2$ and $\mathrm{tr}K$ (in fact we'll conjugate $K$ by a certain multiplication operator - the trace and the Fredholm determinant remain invariant, but the HS-norm does not). Let us do this in the three different cases.

\subsection{No singularities} We now make use of the asymptotics of the orthogonal polynomials from Theorem \ref{th:polyasynosing} to estimate the asymptotics of $K$ in the case where the symbol has no singularities. This was essentially done already in \cite{bl}, but for completeness (and to make sure of the relevant uniformity) we give an argument here.

\begin{lemma}\label{le:nsdet}
Consider a Laurent polynomial $V$ of the form 

\begin{equation}
V(z)=\sum_{j=0}^p \frac{1}{2}(a_j z^j+\overline{a_j}z^{-j}).
\end{equation}

Also fix some $R>0$ and write $A_R=\lbrace \max_j|a_j|\leq R\rbrace$. Then there exists a $\alpha>0$ such that as $N,M\to\infty$ (such that $M\geq N$),

\begin{equation}
\det(I+K)=\mathcal{O}(e^{-\alpha(M-N)})
\end{equation}

\noindent uniformly on $A_R$.

\end{lemma}

\begin{proof}
As noted in Remark \ref{re:contour}, as the symbol is analytic in $\C\setminus \lbrace 0\rbrace$, the integrals across the rays will cancel and we only need to worry about the behavior of the kernel on the circles $(1\pm \epsilon)\times \T$. 

\vspace{0.3cm}

We also note that as the Fredholm determinant is invariant under conjugations (i.e. for invertible $X$, $\det(X(I+K)X^{-1})=\det(I+K)$), we can replace $K$ by $M_N K M_N^{-1}$, where $M_N:L^2(\cup_j\Gamma_j,dz/2\pi i z)\to L^2(\cup_j\Gamma_j,dz/2\pi i z)$, $(M_N h)(z)=z^{-N/2} h(z)$. This effectively symmetrizes the operator $K$. The kernel of this operator (which still is trace class) is $(z/w)^{N/2}K(z,w)$.

\vspace{0.3cm}

Consider first the case $|z|=1-\epsilon$, $|w|=1+\epsilon$ (or vice versa). We have by Theorem \ref{th:polyasynosing} (uniformly)

\begin{align}
\notag\left(\frac{z}{w}\right)^{\frac{N}{2}}K_{CD}(z,w)&=\frac{\left(\frac{w}{z}\right)^{\frac{N}{2}}\phi_N(z)\overline{\phi}_N(w^{-1})-\left(\frac{z}{w}\right)^{\frac{N}{2}}\phi_N(w)\overline{\phi}_N(z^{-1})}{1-\frac{w}{z}}\\
&=\mathcal{O}(1)(1+\epsilon)^{\frac{N}{2}}(1-\epsilon)^{-\frac{N}{2}}.
\end{align}

Thus 

\begin{equation}
K(z,w)=\mathcal{O}(1)(1+\epsilon)^{-\frac{M}{2}}(1-\epsilon)^{\frac{M}{2}}(1+\epsilon)^{\frac{N}{2}}(1-\epsilon)^{-\frac{N}{2}}=\mathcal{O}(e^{-\alpha(M-N)})
\end{equation}

\noindent for a suitable $\alpha>0$.

\vspace{0.3cm}

Consider then the case $|z|=|w|=1\pm\epsilon$. In this case, we write

\begin{align}
\notag\left(\frac{z}{w}\right)^{\frac{N}{2}}K_{CD}(z,w)&=\frac{\left(\frac{w}{z}\right)^{\frac{N}{2}}\phi_N(z)\overline{\phi}_N(w^{-1})-\left(\frac{z}{w}\right)^{\frac{N}{2}}\phi_N(w)\overline{\phi}_N(z^{-1})}{1-\frac{w}{z}}\\
&=z^{1-\frac{N}{2}}w^{-\frac{N}{2}}\phi_N(z)\frac{w^{N}\overline{\phi}_N(w^{-1})-z^{N}\overline{\phi}_N(z^{-1})}{z-w}\\
\notag &+z^{1+\frac{N}{2}}w^{-\frac{N}{2}}\overline{\phi}_N(z^{-1})\frac{\phi_N(z)-\phi_N(w)}{z-w}.
\end{align}

We then note that for a nice enough function $h$, 

\begin{equation}
\frac{h(z)-h(w)}{z-w}=h'(\zeta)
\end{equation}

\noindent for some point $\zeta$ with $|\zeta|=|z|(=|w|)$ so we see that

\begin{align}
\notag \left(\frac{z}{w}\right)^{\frac{N}{2}}K_{CD}(z,w)&=\mathcal{O}(1)N\phi_N(z)\overline{\phi}_N(\zeta^{-1})+\mathcal{O}(1)\phi_N(z)\overline{\phi}_N'(\zeta^{-1})\\
&+\mathcal{O}(1)\overline{\phi}_N(z^{-1})\phi_N'(\xi)
\end{align}

\noindent for some points $\zeta,\xi$ with $|\zeta|=|\xi|=|z|$. In each of the terms above, at least one of the two points is on the circle $(1-\epsilon)\T$ giving exponential smallness. This exponential smallness cancels the factor of $N$, so we see that all of these terms can be bounded (uniformly) by $\mathcal{O}(1)\max(|z|^N,|z|^{-N})$. We conclude that for some $\alpha>0$.

\begin{equation}
\left(\frac{z}{w}\right)^{\frac{N}{2}}K(z,w)=\mathcal{O}(1)\min(|z|^M,|z|^{-M})\max(|z|^N,|z|^{-N})=\mathcal{O}(e^{-\alpha(M-N)})
\end{equation}

\noindent uniformly in $z$ and $w$ as well as on $A_R$.

\vspace{0.3cm}

These estimates for the kernel imply that 

\begin{equation}
\mathrm{tr} K,||M_NKM_N^{-1}||_2=\mathcal{O}(e^{-\alpha(M-N)})
\end{equation}

\noindent uniformly on $A_R$. Making use of Theorem \ref{th:simon}, this gives the desired result.

\vspace{0.3cm}

\end{proof}

\subsection{Distance between singularities bounded away from zero}

Making use of Theorem \ref{th:singpolyasy1}, we try to mimic the proof of the non-singular case to analyze the asymptotics of $\det(I+K)$ in the case when the distance between the singularities is bounded from below.

\begin{lemma}\label{le:sdet}
Let $V:\C\setminus \lbrace 0\rbrace\to \C$ be a Laurent polynomial of the form

\begin{equation}
V(z)=\sum_{j=0}^p\frac{1}{2} (a_j z^j+\overline{a_j}z^{-j}),
\end{equation}

\noindent $(w_1,...,w_k)$ be distinct points on the unit circle such that $0\leq \arg(w_1)<\arg(w_2)<...<\arg(w_k)<2\pi$, and assume that there exists some fixed $\delta>0$ such that $|w_i-w_j|>\delta$ for $i\neq j$. For $|z|=1$, let

\begin{equation}
f(z)=e^{V(z)}\prod_{j=1}^k|z-w_j|^{\beta_j},
\end{equation}

\noindent where $\beta_j>0$, and for any fixed $R>0$, $A_R=\lbrace \max_j |a_j|<R\rbrace$. Assume further that there exists a $q\in(0,1)$ such that for large enough $N$, $N/M<q$.

\vspace{0.3cm}

In this case,

\begin{equation}
\det(I+K)=1+\mathcal{O}(1)\frac{N}{M}
\end{equation}

\noindent uniformly on $A_R$ and uniformly on $\lbrace \min_{i\neq j}|w_i-w_j|>\delta\rbrace$.

\end{lemma}

\begin{proof}
As in the non-singular case, let us conjugate $K$ by the multiplication operator $M_N$. We also start with the case $|z|< 1-\delta/6$ and $|w|>1+\delta/6$ (or vice versa). Making use of Theorem \ref{th:singpolyasy1}, we have (uniformly in everything relevant)

\begin{equation}
\left(\frac{z}{w}\right)^{\frac{N}{2}}K_{CD}(z,w)=\mathcal{O}(1)\max(|z|^{\frac{N}{2}},|z|^{-\frac{N}{2}})\times\max(|w|^{\frac{N}{2}},|w|^{-\frac{N}{2}})
\end{equation}

\noindent and for some $\alpha>0$

\begin{equation}\label{eq:hsest0}
\left(\frac{z}{w}\right)^{\frac{N}{2}}K(z,w)=\mathcal{O}(e^{-\alpha(M-N)}).
\end{equation}

The rest of our analysis we split into two cases: $|z-w|<N^{-1}/2$ and $|z-w|>N^{-1}/2$ (we call these the short and long range regimes). The point of this is that for $|z-w|<N^{-1}/2$, we replace difference quotients by derivatives and for any point $\zeta$ with $|z-\zeta|=\mathcal{O}(N^{-1})$, $\zeta^N=\mathcal{O}(1)z^N$. For $|z-w|>N^{-1}/2$, we simply replace the denominator $1-z/w$ in the CD-kernel by $N^{-1}$.

\vspace{0.3cm}

\bf The short range regime. \rm As in the non-singular case, we can write 

\begin{align}
\notag \left(\frac{z}{w}\right)^{\frac{N}{2}}K_{CD}(z,w)&=\mathcal{O}(1)N\phi_N(z)\overline{\phi}_N(\zeta^{-1})+\mathcal{O}(1)\phi_N(z)\overline{\phi}_N'(\zeta^{-1})\\
&+\mathcal{O}(1)\overline{\phi}_N(z^{-1})\phi_N'(\xi)
\end{align}

\noindent for some points $\xi,\zeta$ between $z$ and $w$ on the contour (here we made use of the fact that for $|z-w|=\mathcal{O}(N^{-1})$, $|\zeta|^N,|\xi|^N,|w|^N=\mathcal{O}(1)|z|^N$ uniformly in everything).

\vspace{0.3cm}

If $|z|,|w|<1-\delta/6$ or $|z|,|w|>1+\delta/6$, then by Theorem \ref{th:singpolyasy1}, we see that reasoning similarly to the non-singular case, we obtain the uniform bound 

\begin{align}
\notag \left(\frac{z}{w}\right)^{\frac{N}{2}}K_{CD}(z,w)&=\mathcal{O}(1)\max(|z|^N,|z|^{-N})\\
&=\mathcal{O}(1)\max(|z|^{\frac{N}{2}},|z|^{-\frac{N}{2}})\times\max(|w|^{\frac{N}{2}},|w|^{-\frac{N}{2}}).
\end{align}

This then implies that for some $\alpha>0$, we have uniformly in everything

\begin{equation}\label{eq:trest1}
K(z,w)=\mathcal{O}(e^{-\alpha(M-N)}).
\end{equation}

As the distance between the singularities is bounded away from zero, this only leaves the case $z,w\in[1-\delta/6,1+\delta/6]\times w_j$ for some $j$ (the case where one of the points is near the end of the interval and one right next to the interval can be treated as the case where neither is in the interval by retuning the parameters). Here we need to note that as $|z|\to 1$, we don't get exponential smallness from $v(z)$. In fact, in the worst case scenario, where $w_j\in\mathcal{D}_M$ (i.e. $w_j^M=1$), $v(w_j)=\infty$. We assume the worst case scenario is the generic one and we use the following estimate for $v$:

\begin{equation}
v(z)=\begin{cases}
\mathcal{O}(1)\min(|z|^M,|z|^{-M}), & |1-|z||>M^{-1}\\
\mathcal{O}(1)\frac{1}{M|1-|z||}, & |1-|z||<M^{-1}.
\end{cases}
\end{equation}

Due to this change of behavior at a scale of $M^{-1}$, we need to split our treatment further into several subcases. We also will need to make use of the fact that $f(z)=\mathcal{O}(1)|1-|z||^{\beta_j}$ is present in the kernel.

\vspace{0.3cm}

\underline{$|1-|z||,|1-|w||>N^{-1}$:} In this case, using Theorem \ref{th:singpolyasy1}, we have (note that in this case, we have for example $|1-|\zeta||^{-1}=|1-|w|+|w|-|\zeta||^{-1}=\mathcal{O}(1)|1-|w||^{-1}$) 

\begin{equation}
\left(\frac{z}{w}\right)^{\frac{N}{2}}K_{CD}(z,w)=\mathcal{O}(1)N|1-|z||^{-\frac{\beta_j}{2}}|1-|w||^{-\frac{\beta_j}{2}}
\end{equation}

\noindent so that 

\begin{equation}\label{eq:trest2}
\left(\frac{z}{w}\right)^{\frac{N}{2}}K(z,w)=\mathcal{O}(1)N \min (|z|^{\frac{M}{2}},|z|^{-\frac{M}{2}})\min (|w|^{\frac{M}{2}},|w|^{-\frac{M}{2}}).
\end{equation}

Again by retuning the parameters, we see that essentially the only other possibility for $|z-w|<N^{-1}/2$ is that $|1-|z||,|1-|w||<N^{-1}$. In this case, Theorem \ref{th:singpolyasy1} implies that 

\begin{equation}
\left(\frac{z}{w}\right)^{\frac{N}{2}}K_{CD}(z,w)=\mathcal{O}(1)N^{1+\beta_j}
\end{equation}

\noindent uniformly. One then needs to take into account the different behavior of $v(z)$ for $|1-|z||<M^{-1}$ and $|1-|z||>M^{-1}$. We have the following cases:

\vspace{0.3cm}

\underline{$M^{-1}<|1-|z||,|1-|w||<N^{-1}$}: Here we have 

\begin{align}\label{eq:trest3}
\notag \left(\frac{z}{w}\right)^{\frac{N}{2}}K(z,w)&=\mathcal{O}(1)\min(|z|^{\frac{M}{2}},|z|^{-\frac{M}{2}})\min(|w|^{\frac{M}{2}},|w|^{-\frac{M}{2}})\\
&\times |1-|z||^{\frac{\beta_j}{2}}|1-|w||^{\frac{\beta_j}{2}}N^{\beta_j+1}.
\end{align}

\underline{$|1-|z||<M^{-1}$ and $M^{-1}<|1-|w||<N^{-1}$ (or vice versa):} In this case, we have 

\begin{equation}\label{eq:hsest1}
\left(\frac{z}{w}\right)^{\frac{N}{2}}K(z,w)=\mathcal{O}(1)\frac{N^{1+\beta_j}}{\sqrt{M}}|1-|z||^{\frac{\beta_j}{2}-\frac{1}{2}}|1-|w||^{\frac{\beta_j}{2}}\min(|w|^{\frac{M}{2}},|w|^{-\frac{M}{2}})
\end{equation}

\noindent and the other case by interchanging the roles of $z$ and $w$.

\vspace{0.3cm}

\underline{$|1-|z||,|1-|w||<M^{-1}$:} Finally we have

\begin{equation}\label{eq:trest4}
\left(\frac{z}{w}\right)^{\frac{N}{2}}K(z,w)=\mathcal{O}(1)\frac{N^{1+\beta_j}}{M}|1-|z||^{\frac{\beta_j}{2}-\frac{1}{2}}|1-|w||^{\frac{\beta_j}{2}-\frac{1}{2}}
\end{equation}

\noindent uniformly.

\vspace{0.3cm}

\bf The long range regime: \rm In this case we estimate

\begin{align}
\notag \left(\frac{z}{w}\right)^{\frac{N}{2}}K_{CD}(z,w)&=\mathcal{O}(1)N\left(\frac{z}{w}\right)^{\frac{N}{2}}\phi_N(w)\overline{\phi}_N(z^{-1})\\
&+\mathcal{O}(1)N\left(\frac{w}{z}\right)^{\frac{N}{2}}\phi_N(z)\overline{\phi}_N(w^{-1}).
\end{align}

Again we consider the different parts of the contour $z$ and $w$ can  lie on.

\vspace{0.3cm}

\underline{$|z|,|w|<1-\delta/6$ or $|z|,|w|>1+\delta/6$:} Making use of Theorem \ref{th:singpolyasy1} we have 

\begin{equation}
\left(\frac{z}{w}\right)^{\frac{N}{2}}K_{CD}(z,w)=\mathcal{O}(1)\max(|z|^{\frac{N}{2}} |w|^{\frac{N}{2}},|z|^{-\frac{N}{2}}|w|^{-\frac{N}{2}})
\end{equation}

\noindent and we have as before, for some fixed $\alpha>0$

\begin{equation}\label{eq:hsest00}
\left(\frac{z}{w}\right)^{\frac{N}{2}}K(z,w)=\mathcal{O}(e^{-\alpha(M-N)})
\end{equation}

\noindent uniformly. 
\vspace{0.3cm}

Let us now focus on the situation where $z\in[1-\delta/6,1+\delta/6]\times w_j$ and $w\in[1-\delta/6,1+\delta/6]\times w_i$ for some $i$ and $j$ (note that possibly $i=j$).

\vspace{0.3cm}

\underline{$|1-|z||,|1-|w||>N^{-1}$:} In this case our worst case scenario is where $z$ and $w$ are on the opposite sides of the unit circle and we have 

\begin{equation}\label{eq:hsest2}
\left(\frac{z}{w}\right)^{\frac{N}{2}}K(z,w)=\mathcal{O}(1)N \min(|z|^{\frac{M-N}{2}},|z|^{-\frac{M-N}{2}})\min(|w|^{\frac{M-N}{2}},|w|^{-\frac{M-N}{2}}).
\end{equation}

\underline{$|1-|z||>N^{-1}$ and $M^{-1}<|1-|w||<N^{-1}$ (or vice versa):} In this case, $|w|^{\pm N}=\mathcal{O}(1)$ so we see that 

\begin{align}\label{eq:hsest3}
\notag \left(\frac{z}{w}\right)^{\frac{N}{2}}K(z,w)&=\mathcal{O}(1) N\min(|z|^{\frac{M-N}{2}},|z|^{-\frac{M-N}{2}})\min(|w|^{\frac{M}{2}},|w|^{-\frac{M}{2}})\\
&\times|1-|w||^{\frac{\beta_i}{2}} N^{\frac{\beta_i}{2}}.
\end{align}

The opposite case is obtained by interchanging $z$ and $w$.

\vspace{0.3cm}

\underline{$|1-|z||>N^{-1}$ and $|1-|w||<M^{-1}$ (or vice versa):} In this case, we need to take into account the behavior of $v$ near the unit circle. We find

\begin{align}\label{eq:hsest4}
\notag \left(\frac{z}{w}\right)^{\frac{N}{2}}K(z,w)&=\mathcal{O}(1) \frac{N}{\sqrt{M}}\min(|z|^{\frac{M-N}{2}},|z|^{-\frac{M-N}{2}})\\
&\times|1-|w||^{\frac{\beta_i}{2}-\frac{1}{2}} N^{\frac{\beta_i}{2}}.
\end{align}

The opposite case is obtained by interchanging $z$ and $w$.

\vspace{0.3cm}

\underline{$M^{-1}<|1-|z||,|1-|w||<N^{-1}$:} With similar arguments, 

\begin{align}\label{eq:hsest5}
\notag \left(\frac{z}{w}\right)^{\frac{N}{2}}K(z,w)&=\mathcal{O}(1) N\min(|z|^{\frac{M}{2}},|z|^{-\frac{M}{2}})\min(|w|^{\frac{M}{2}},|w|^{-\frac{M}{2}})\\
&\times |1-|w||^{\frac{\beta_i}{2}} N^{\frac{\beta_i}{2}}|1-|z||^{\frac{\beta_j}{2}} N^{\frac{\beta_j}{2}}.
\end{align}

\vspace{0.3cm}

\underline{$M^{-1}<|1-|z||<N^{-1}$ and $|1-|w||<M^{-1}$ (or vice versa):} Here we find

\begin{align}\label{eq:hsest6}
\notag \left(\frac{z}{w}\right)^{\frac{N}{2}}K(z,w)&=\mathcal{O}(1) \frac{N}{\sqrt{M}}\min(|z|^{\frac{M}{2}},|z|^{-\frac{M}{2}})\\
&\times|1-|w||^{\frac{\beta_i}{2}-\frac{1}{2}} N^{\frac{\beta_i}{2}}|1-|z||^{\frac{\beta_j}{2}}N^{\frac{\beta_j}{2}}.
\end{align}

The opposite case is obtained by interchanging $z$ and $w$.

\vspace{0.3cm}

\underline{$|1-|z||,|1-|w||<M^{-1}$:} Here we have 

\begin{equation}\label{eq:hsest7}
\left(\frac{z}{w}\right)^{\frac{N}{2}}K(z,w)=\mathcal{O}(1)\frac{N}{M}|1-|z||^{\frac{\beta_j}{2}-\frac{1}{2}}|1-|w||^{\frac{\beta_i}{2}-\frac{1}{2}}N^{\frac{\beta_j+\beta_i}{2}}.
\end{equation}

\vspace{0.3cm}

We are now in a position to estimate the trace and the HS-norm. Let us start with the trace.

\vspace{0.3cm}

\bf Estimating the trace: \rm As $K$ is trace class, its trace is well defined and given by $\sum_j\int_{\Gamma_j}K(z,z)\sigma(dz)$. Thus we only need the short range asymptotics ($z=w$). We then decompose our integration contour into the relevant parts:

\begin{align}
\notag\int_{\Gamma_j}&=\int_{\Gamma_j\cap\lbrace |1-|z||>\delta/6\rbrace}+\int_{\Gamma_j\cap\lbrace N^{-1}<|1-|z||<\delta/6\rbrace}\\
&+\int_{\Gamma_j\cap\lbrace M^{-1}<|1-|z||<N^{-1}\rbrace}+\int_{\Gamma_j\cap\lbrace |1-|z||<M^{-1}\rbrace}
\end{align}

By \eqref{eq:trest1},\eqref{eq:trest2},\eqref{eq:trest3}, and \eqref{eq:trest4} we have then for some fixed $\alpha>0$

\begin{align}
\notag \mathrm{tr}K&=\mathcal{O}(e^{-\alpha(M-N)})+\sum_j\int_{N^{-1}}^{\delta/6}\mathcal{O}(1)N (1+r)^{-M}dr\\
\notag &+\sum_j\int_{M^{-1}}^{N^{-1}}\mathcal{O}(1)N^{\beta_j+1}(1+r)^{-M}r^{\beta_j}dr\\
\notag &+\sum_j\int_0^{M^{-1}}\mathcal{O}(1)\frac{N^{1+\beta_j}}{M}r^{\beta_j-1}dr\\
&=\mathcal{O}(e^{-\alpha(M-N)})+\mathcal{O}(1)\frac{N}{M}+\mathcal{O}(1)\frac{N}{M}+\mathcal{O}(1)\sum_j\left(\frac{N}{M}\right)^{1+\beta_j}\\
\notag &=\mathcal{O}(1)\frac{N}{M}
\end{align}

\noindent uniformly on $A_R$.

\vspace{0.3cm}

\bf Estimating the square of the HS-norm: \rm To apply Theorem \ref{th:simon}, we wish to also estimate the HS-norm of $M_N KM_N^{-1}$. By definition, this is given by 

\begin{equation}
||M_N K M_N^{-1}||_2^2=\sum_j\sum_l\int_{\Gamma_j}\int_{\Gamma_l}\sigma(dz)\sigma(dw)\left|\frac{z}{w}\right|^N|K(z,w)|^2.
\end{equation}

We decompose $\Gamma_j\times \Gamma_l$ into the short and long range regimes and then split these into parts where we have estimated the different asymptotics of the kernel. Let us consider first the short range regime. We split $\Gamma_j\times\Gamma_l$ in the short range regime into the cases where i) $|1-|z||,|1-|w||>\delta/6$, ii) $|1-|z||,|1-|w||\in(N^{-1},\delta/6)$, iii) $|1-|z||,|1-|w||\in (M^{-1},N^{-1})$, iv) $|1-|z||<M^{-1}$ and $|1-|w||\in(M^{-1},N^{-1})$ (or vice versa), and v) $|1-|z||,|1-|w||<M^{-1}$ (recall that the remaining cases can be included into these by retuning the parameters). Making use of \eqref{eq:trest1}, \eqref{eq:trest2}, \eqref{eq:trest3}, \eqref{eq:hsest1}, and \eqref{eq:trest4} we see that the contribution of the short range case to the square of the HS norm is

\begin{align}
\notag &\mathcal{O}(e^{-\alpha(M-N)})+\mathcal{O}(1)N^2\left(\int_{N^{-1}}^{\delta/6}(1+r)^{-M}dr\right)^2\\
\notag &+\mathcal{O}(1)\sum_j N^{2+2\beta_j}\left(\int_{M^{-1}}^{N^{-1}}(1+r)^{-M}r^{\beta_j}dr\right)^2\\
&+\mathcal{O}(1)\sum_j \frac{N^{2+2\beta_j}}{M}\int_0^{M^{-1}}r^{\beta_j-1}dr\int_{M^{-1}}^{N^{-1}}(1+r)^{-M}r^{\beta_j}dr\\
\notag & +\mathcal{O}(1)\sum_j \frac{N^{2+2\beta_j}}{M^2}\left(\int_0^{M^{-1}}r^{\beta_j-1}dr\right)^2\\
\notag &=\mathcal{O}(1)\left(\frac{N}{M}\right)^2.
\end{align}

For the long range regime, we again split it into the relevant subregions: i) $|1-|z||,|1-|w||>\delta/6$, ii) $|1-|z||,|1-|w||\in(N^{-1},\delta/6)$, iii) $|1-|z||\in(N^{-1},\delta/6)$ and $|1-|w||\in(M^{-1},N^{-1})$ (and vice versa), iv) $|1-|z||\in(N^{-1},\delta/6)$ and $|1-|w||<M^{-1}$ (and vice versa), v) $|1-|z||,|1-|w||\in(M^{-1},N^{-1})$, vi) $|1-|z||\in(M^{-1},N^{-1})$ and $|1-|w||<M^{-1}$ (and vice versa), and vii) $|1-|z||,|1-|w||<M^{-1}$. Making use of \eqref{eq:hsest0}, \eqref{eq:hsest00}, \eqref{eq:hsest2}, \eqref{eq:hsest3}, \eqref{eq:hsest4}, \eqref{eq:hsest5}, \eqref{eq:hsest6}, and \eqref{eq:hsest7} we see that the contribution of the long range regime to the square of the HS-norm is (for some fixed $\alpha>0$)

\begin{align}
\notag &\mathcal{O}(e^{-\alpha(M-N)})+\mathcal{O}(1)N^2\left(\int_{N^{-1}}^{\delta/6}(1+r)^{-(M-N)}dr\right)^2\\
\notag &+\mathcal{O}(1)\sum_j N^{2+\beta_j}\int_{N^{-1}}^{\delta/6}(1+r)^{-(M-N)}dr\int_{M^{-1}}^{N^{-1}}(1+r)^{-M}r^{\beta_j}dr\\
\notag &+\mathcal{O}(1)\frac{N^{2+\beta_j}}{M}\sum_j \int_{N^{-1}}^{\delta/6}(1+r)^{-(M-N)}dr\int_0^{M^{-1}}r^{\beta_j-1}dr\displaybreak[0]\\
&+\mathcal{O}(1)\sum_{i,j} N^{2+\beta_i+\beta_j}\int_{M^{-1}}^{N^{-1}}(1+r)^{-M}r^{\beta_j}dr\int_{M^{-1}}^{N^{-1}}(1+r)^{-M}r^{\beta_i}dr\\
\notag & +\mathcal{O}(1)\sum_{i,j}\frac{N^{2+\beta_i+\beta_j}}{M}\int_{M^{-1}}^{N^{-1}}(1+r)^{-M}r^{\beta_j}dr\int_0^{M^{-1}} r^{\beta_i-1}dr\\
&\notag +\mathcal{O}(1)\sum_{i,j}\frac{N^{2+\beta_i+\beta_j}}{M^2}\int_0^{M^{-1}}r^{\beta_j-1}dr\int_0^{M^{-1}}r^{\beta_i-1}dr\\
\notag &=\mathcal{O}(1)\left(\frac{N}{M}\right)^2,
\end{align} 

\noindent where we made use of $N/M\leq q<1$ for large enough $n$.

\vspace{0.3cm}

Combining these, we see that $||M_N K M_N^{-1}||_2=\mathcal{O}(1)N/M$ so that by Theorem \ref{th:simon},

\begin{equation}
\det(I+K)=1+\mathcal{O}(1)\frac{N}{M}
\end{equation} 

\noindent uniformly on $A_R$ and $\lbrace \min_{i\neq j}|w_i-w_j|>\delta\rbrace$.

\end{proof}

\subsection{Merging singularities} The basic idea here is the same as before. We estimate the kernel with the help of Theorem \ref{th:singpolyasy2}. For simplicity, we use the symbol $|z-e^{it}|^{\beta}||z-e^{-it}|^{\beta}$. The estimates are the same in the case obtained by rotating this symbol.

\begin{lemma}\label{le:msdet}
For $|z|=1$, let $f_t(z)=|z-e^{it}|^\beta|z-e^{-it}|^\beta$. Assuming that there exists a $q\in(0,1)$ such that for large enough $N$, $N/M<q$, we have 

\begin{equation}
\det(I+K)=1+\mathcal{O}(1)\frac{N}{M},
\end{equation}

\noindent where $\mathcal{O}(1)$ is uniform in $0<t<t_0$.

\end{lemma}

\begin{proof}

Let us begin by considering what is different in the case of merging singularities. First of all, the asymptotics of the orthogonal polynomials are different. Next we note that if $t=\mathcal{O}(1)N^{-1}$, we might have the issue that in the short range regime, we need to consider the case where $z\in(1-\epsilon,1+\epsilon)e^{it}$ and $w\in(1-\epsilon,1+\epsilon)e^{-it}$ (or vice versa). In this case, we can't use the same estimate for the CD-kernel we used before - the points $\zeta$ and $\xi$ might not be on the contour any more (or then they might not be close to $z$ and $w$).  For this, we will use Lemma \ref{le:offcont}. Another issue is that the asymptotics of $f_t$ near $e^{\pm it}$ are different. Instead of $f(z)=\mathcal{O}(1)|1-|z||^{\beta}$, we have 

\begin{equation}
f_t(z)=\mathcal{O}(1)|1-|z||^{\beta}||z|-1+(1-e^{-2it})|^\beta=\mathcal{O}(1)|1-|z||^{\beta}\max(|1-|z||^{\beta},t^{\beta}). 
\end{equation}

\vspace{0.3cm}

Apart from these differences, the reasoning is similar to the previous case so we will not be as detailed as we were there. We will again consider the conjugated operator $M_N K M_N^{-1}$. Let us begin with considering the long range regime.

\vspace{0.3cm}

\bf Long range regime: \rm As in the case of non-merging singularities, we estimate the CD-kernel as 

\begin{align}
\notag \left(\frac{z}{w}\right)^{\frac{N}{2}}K_{CD}(z,w)&=\mathcal{O}(N)\left(\frac{w}{z}\right)^{\frac{N}{2}}\phi_N(z)\overline{\phi}_N(w^{-1})\\
&+\mathcal{O}(N)\left(\frac{z}{w}\right)^{\frac{N}{2}}\overline{\phi}_N(z^{-1})\phi_N(w).
\end{align}

As the polynomials have the same asymptotics as in the case of the non-merging singularities for $|1-|z||>\gamma$, we can focus on the case where at least one of the $z,w$ is at distance less than $\gamma$ from the unit circle. We then recall that $K(z,w)$ comes with a factor of $\sqrt{f_t(z)}\sqrt{f_t(w)}$ so what is relevant is the asymptotics of $\sqrt{f_t(z)}\phi_N(z)$. Comparing Theorem \ref{th:singpolyasy1} and \ref{th:singpolyasy2}, we see that these are the same in the two cases (the only case that looks different is the $N|1-|z||<1$ case for $Nt<c$, but here the difference is the term $\max((N|1-|z||)^{\frac{\beta}{2}},(Nt)^{\frac{\beta}{2}})<\max(1,c^{\frac{\beta}{2}})$). Thus in the contribution of the long range regime to the square of the HS-norm of $M_NKM_N^{-1}$ is the same as before (of order $(N/M)^2$).

\vspace{0.3cm}

\bf Short range regime: \rm We again estimate 

\begin{align}
\notag \left(\frac{z}{w}\right)^{\frac{N}{2}}K_{CD}(z,w)&=\mathcal{O}(1)N\phi_N(z)\overline{\phi}_N(\zeta^{-1})+\mathcal{O}(1)\phi_N(z)\overline{\phi}_N'(\zeta^{-1})\\
&+ \mathcal{O}(1)\overline{\phi}_N(z^{-1})\phi_N'(\xi)
\end{align}

\noindent for some points $\xi,\zeta$ with $|z-\zeta|=\mathcal{O}(N^{-1})$ and $|z-\xi|=\mathcal{O}(N^{-1})$. If $z$ and $w$ are on the different rays, we take $\xi$ and $\zeta$ such that they are on the line going through $z$ and $w$ (this can only happen if $t=\mathcal{O}(1)N^{-1}$ which up to a $\mathcal{O}(1)$ term is equivalent to the $Nt<c$ case). Otherwise, we take them on the contour and between the points. 

\vspace{0.3cm}

For $|1-|z||>\gamma$ and $|1-|w||>\gamma$ we can reason as before to get exponential smallness. Let us then simply argue as in the case of non-merging singularities to find the asymptotics of the kernel for different values of $z$ and $w$. 

\vspace{0.3cm}

\emph{$|1-|z||,|1-|w||\in(N^{-1},\gamma]$:} We have now e.g. $|1-|z||=\mathcal{O}(1)|1-|\zeta||$ implying also $f_t(z)=\mathcal{O}(1)f_t(\zeta)$. Let us first consider the case where $z$, $w$ are on the same ray (as the one where they are on different rays is only relevant for $Nt<c$). Using Theorem \ref{th:singpolyasy2}, we find

\begin{equation}
\sqrt{f_t(z)}\sqrt{f_t(w)}\left(\frac{z}{w}\right)^{\frac{N}{2}}K_{CD}(z,w)=\mathcal{O}(1)N
\end{equation}

\noindent and 

\begin{equation}\label{eq:kest1}
\left(\frac{z}{w}\right)^{\frac{N}{2}}K(z,w)=\mathcal{O}(1)N\min(|z|^{\frac{M}{2}},|z|^{-\frac{M}{2}})\min(|w|^{\frac{M}{2}},|w|^{-\frac{M}{2}})
\end{equation}

\noindent uniformly in $z,w$ and $0<t<t_0$. Again up to a uniform $\mathcal{O}(1)$ term, the only remaining situation is when $|1-|z||,|1-|w||<N^{-1}$ and in this case Theorem \ref{th:singpolyasy2} implies that 

\begin{equation}
\left(\frac{z}{2}\right)^{\frac{N}{2}}K_{CD}(z,w)=\mathcal{O}(1)N^{\beta+1}(\min(N,t^{-1}))^{\beta}.
\end{equation}

We thus have (uniformly)

\begin{align}\label{eq:kest2}
\notag \left(\frac{z}{2}\right)^{\frac{N}{2}}K(z,w)&=\mathcal{O}(1)\sqrt{v(z)}\sqrt{v(w)}|1-|z||^{\frac{\beta}{2}}|1-|w||^{\frac{\beta}{2}}\\
&\times \max(|1-|z||^{\frac{\beta}{2}},t^{\frac{\beta}{2}})\max(|1-|w||^{\frac{\beta}{2}},t^{\frac{\beta}{2}})N^{\beta+1}\min(N^\beta,t^{-\beta}).
\end{align}

\vspace{0.3cm}

Consider finally the case where $z$ and $w$ are not on the same ray (in the case where $Nt<c$). In this case, we estimate (using Lemma \ref{le:offcont})

\begin{align}\label{kest3}
\left(\frac{z}{w}\right)^{\frac{N}{2}}K_{CD}(z,w)&=\mathcal{O}(1)N\min(|1-|z||^{-\beta},N^\beta)\min(|1-|w||^{-\beta},N^\beta).
\end{align}

\vspace{0.3cm}

\bf Estimating the trace: \rm For the trace we only need the short range estimate. Using \eqref{eq:kest1} and \eqref{eq:kest2}, we have for some $\alpha>0$

\begin{align}
\notag \mathrm{tr}K &=\mathcal{O}(e^{-\alpha(M-N)})+\mathcal{O}(1)N\int_{N^{-1}}^\gamma (1+r)^{-M}dr\\
&+\mathcal{O}(1)\int_{|1-|z||<N^{-1}}|v(z)||1-|z||^{\beta}\max(|1-|z||^\beta,t^\beta)dz\\
\notag &\times N^{\beta+1}\min(N^\beta,t^{-\beta}).
\end{align}

The second term is (uniformly in $0<t<t_0$) of order $N/M$. So it remains to estimate the last term. If $t>N^{-1}$, it is 

\begin{equation}
\mathcal{O}(1)N^{\beta+1}\int_{|1-|z||<N^{-1}}|v(z)||1-|z||^{\beta}dz
\end{equation}

\noindent which is of the kind of integrals we've already estimated and found to be of order $N/M$. If $t<N^{-1}$, we the term is 

\begin{align}
\notag &\mathcal{O}(1)N^{2\beta+1}\int_{t<|1-|z||<N^{-1}}|v(z)||1-|z||^{2\beta}dz\\
&+\mathcal{O}(1)N^{2\beta+1}t^\beta\int_{|1-|z||<t}|v(z)||1-|z||^\beta dz
\end{align}

Here the first term can estimated upwards to 

\begin{equation}
\mathcal{O}(1)N^{2\beta+1}\int_{0<|1-|z||<N^{-1}}|v(z)||1-|z||^{2\beta}dz
\end{equation}

\noindent which again is something we have estimated to be of order $N/M$. The second one can be estimated upwards (recall $Nt<1$) to 

\begin{equation}
\mathcal{O}(1)N^{\beta+1}\int_{|1-|z||<N^{-1}}|v(z)||1-|z||^\beta dz
\end{equation}

\noindent so we conclude that the trace is of order $N/M$ (uniformly in $t$).

\vspace{0.3cm}

\bf Estimating the square of the HS-norm: \rm As noted, the long range estimate is the same and of order $(N/M)^2$. For the short range one, we note that if the points are on the same ray, our estimate for the kernel just factors into two copies of what we estimated in the case of the trace so from this we get a term which is uniformly if order $(N/M)^2$. We are left with an estimate in the case where the points are on different rays (in the case $Nt<c$). Again we have factorization and the remaining estimate is calculating 

\begin{equation}
N^2\left(\int_{|1-|z||<\gamma}|f_t(z)||v(z)|\min(|1-|z||^{-2\beta},N^{2\beta})dz\right)^2.
\end{equation}

Such integrals were estimated for the trace (and using those estimates we find this to be of order $(N/M)^2$ as well). We conclude that 

\begin{equation}
||M_NK M_N^{-1}||_2^2=\mathcal{O}(1)\left(\frac{N}{M}\right)^2
\end{equation}

\noindent uniformly in $t$.

\vspace{0.3cm}

Thus making use of Theorem \ref{th:simon} we have our claim.

\end{proof}

\section{Proof of main results}

Our estimates so far essentially state that if $N/M\to 0$, the discrete case does not differ from the CUE. Proving our main results is then essentially as in \cite{webb}. For completeness (and due to the fact that there are small differences) we provide proofs here. We will first prove some of the propositions presented when describing the structure of our proof and finally combine these into a proof of Theorem \ref{th:main}.

\subsection{Proof of Proposition \ref{prop:linstat} and Proposition \ref{prop:approxconv}}

\begin{proof}[Proof of Proposition \ref{prop:linstat}]
While some variant of the result certainly follows from more general results (see e.g. \cite{bgg}), we choose to give a proof here for the sake of completeness. Recall the notation

\begin{equation}
\widetilde{Z}_j=\sum_{k=1}^N z_k^j.
\end{equation}

We also write 

\begin{equation}
\begin{array}{ccc}
\widetilde{X}_j=\mathrm{Re}(\widetilde{Z}_j) & \mathrm{and}  & \widetilde{Y}_j=\mathrm{Im}(\widetilde{Z}_j).
\end{array}
\end{equation}

We wish to prove that for any fixed $l$, $(\widetilde{X}_j,\widetilde{Y}_j)_{j=1}^l$ converge in law to $(\sqrt{j/2}X_j,\sqrt{j/2}Y_j)_{j=1}^l$, where $(X_1,Y_1,...,X_l,Y_l)$ are i.i.d. standard Gaussians. Let us consider the moment generating function of $(\widetilde{X}_j,\widetilde{Y}_j)_{j=1}^l$: let $t_1,...,t_l,s_1,...,s_l\in \R$ and consider 

\begin{align}
\notag M(t_1,...,t_l,s_1,...,s_l)&=\E\left(e^{\sum_{j=1}^l (t_j \widetilde{X}_j+s_j\widetilde{Y}_j)}\right)\\
&=\E\left(\prod_{k=1}^N e^{\sum_{j=1}^l\frac{1}{2}\left((t_j-is_j)z_k^j+(t_j+is_j)z_k^{-j}\right)}\right).
\end{align}

If we then define 

\begin{equation}
V(z)=\sum_{j=1}^l\frac{1}{2}\left((t_j-is_j)z^j+(t_j+is_j)z^{-j}\right)
\end{equation}

\noindent we have by Proposition \ref{prop:dischs}

\begin{equation}
M(t_1,...,t_l,s_1,...,s_l)=T_{N-1}(e^V).
\end{equation}

Then by Proposition \ref{prop:detsing} (and Remark \ref{re:contour}) as well as Lemma \ref{le:nsdet},

\begin{equation}
M(t_1,...,t_l,s_1,...,s_l)=T_{N-1}(e^V)=\mathcal{T}_{N-1}(e^{V})\times(1+\mathcal{O}(e^{-\alpha(M-N)}))
\end{equation}

\noindent for some fixed $\alpha>0$. Now by the Strong Szeg\"o theorem (Theorem \ref{th:szego})

\begin{equation}
\lim_{N\to\infty}\mathcal{T}_{N-1}(e^{V})=e^{\sum_{j=1}^lj \frac{t_j^2+s_j^2}{4}}=\E\left(e^{\sum_{j=1}^l \sqrt{\frac{j}{2}}(t_j X_j+s_j Y_j)}\right)
\end{equation}

\noindent which finishes the proof.

\end{proof}

The proof of Proposition \ref{prop:approxconv} is then more or less immediate, and while a similar one is presented in \cite{webb}, we give one for the convenience of the reader. 

\begin{proof}[Proof of Proposition \ref{prop:approxconv}]
Noting that 

\begin{equation}
\E(F_{N,M,L}^\beta(e^{i\theta}))=\E\left(\prod_{k=1}^Ne^{-\beta\sum_{j=1}^L\frac{1}{j}(\cos (j\theta) \widetilde{X}_j+\sin(j\theta) \widetilde{Y}_j)}\right)
\end{equation}

\noindent we see from the proof of Proposition \ref{prop:linstat} that under our assumptions

\begin{equation}
\lim_{N\to\infty}\E(F_{N,M,L}^\beta(e^{i\theta}))=e^{\frac{\beta^2}{4}\sum_{j=1}^L\frac{1}{j}}.
\end{equation}

In fact, this convergence is uniform in $\theta$ as the relevant continuum Toeplitz determinant is independent of $\theta$ due to the fact that the law of the eigenvalues of the CUE is rotation invariant.

\vspace{0.3cm}

Due to the convergence in distribution of Proposition \ref{prop:linstat}, there exists a probability space where one can construct a sequence of random variables $(Z_1^{(N,M)},...,Z_L^{(N,M)})_{1\leq N\leq M}$ and $(Z_1,...,Z_L)$ such that for each $N,M$ 
\begin{equation}
(Z_1^{(N,M)},...,Z_L^{(N,M)})\stackrel{d}{=}\left(\sum_{k=1}^Nz_k,...,\sum_{k=1}^Nz_k^L\right)
\end{equation}

\noindent where $(z_1,...,z_N)$ is sampled from $\mathbb{P}_{N,M}$,  $(Z_j)_{j=1}^L$ are independent complex Gaussians with real and imaginary parts independent centered real Gaussians of variance $j/2$, and 

\begin{equation}
(Z_1^{(N,M)},...,Z_L^{(N,M)})\to (Z_j)_{j=1}^L
\end{equation}

\noindent almost surely as $N\to\infty$ and $M-N\to\infty$.

\vspace{0.3cm}

This implies that if we construct an object agreeing in law with $F_{N,M,L}^\beta$ from these quantities, it converges uniformly almost surely to 

\begin{equation}
\theta\to e^{-\beta \mathrm{Re}\sum_{j=1}^L \frac{1}{j}e^{-ij\theta}Z_j}.
\end{equation}

Due to this uniform convergence (as well as the fact that $(-Z_j)_j\stackrel{d}{=}(Z_j)_j$) and the convergence of $\E(F_{N,M,L}^\beta(e^{i\theta}))$  to $\E(e^{\beta X_L(\theta)})$ it follows that the corresponding object whose law agrees with that of $\mu_{N,M,L}^\beta$ converges almost surely to $\mu_L^\beta$ as $N\to\infty$ in such a way that $M-N\to\infty$ (the convergence is in the topology of weak convergence of measures). Thus we have the claim.

\end{proof}

\subsection{Proof of Proposition \ref{prop:fh2}, Proposition \ref{prop:fh}, and Proposition \ref{prop:var}}

We begin with Proposition \ref{prop:fh2} 

\begin{proof}[Proof of Proposition \ref{prop:fh2}]
This follows directly from combining Theorem \ref{th:fhcont}, Proposition \ref{prop:detsing}, and Lemma \ref{le:sdet}.
\end{proof}

We then move onto Proposition \ref{prop:fh}.

\begin{proof}[Proof of Proposition \ref{prop:fh}]
This is essentially just a combination of our estimates on the Fredholm determinant and the strong Szeg\"o theorem or results on the asymptotics of Toeplitz determinants with Fisher-Hartwig singularities - namely we combine Theorem \ref{th:szego}, Theorem \ref{th:fhcont}, Proposition \ref{prop:detsing}, and Lemmas \ref{le:nsdet}, \ref{le:sdet}, and \ref{le:msdet}. As we make statements about uniformity, we give a few more details now.

\vspace{0.3cm}

1) Convergence follows directly from Proposition \ref{prop:detsing}, Theorem \ref{th:szego}, and Lemma \ref{le:nsdet}. Uniformity follows from the fact that the continuum Toeplitz determinant is independent of $\theta$ (e.g. by the Heine-Szeg\"o identity and the fact that the law of the eigenvalues of a CUE matrix is rotation invariant), while our estimate in Lemma \ref{le:nsdet} is uniform.

\vspace{0.3cm}

2) This is as the previous case except we use statement $1)$ of Theorem \ref{th:fhcont} instead of Theorem \ref{th:szego} for the asymptotics of the continuum Toeplitz determinant.

\vspace{0.3cm}

3) The reasoning is the same again - Proposition \ref{prop:detsing} gives the factorization into a Fredholm determinant with uniform estimates and a continuum Toeplitz determinant, which we've denoted by $E_{N,L}(\theta,\theta')$. The fact that this increases to $\E(e^{\beta X_L(\theta)}e^{\beta X_L(\theta')})$ is proven for example in \cite{simonsz}.

\vspace{0.3cm}

4) Again we make use of Proposition \ref{prop:detsing}, Theorem \ref{th:fhcont}, and Lemma \ref{le:sdet}. The uniformity now follows from the fact that the continuum Toeplitz determinant is rotation invariant (so we can rotate $\theta$ from the non-singular part to $\theta'-\theta$ in the singular part) and make use of the uniformity of statement $1)$ of Theorem \ref{th:fhcont} (now there is only one singularity so the estimate is uniform in its location).

\vspace{0.3cm}
 
5) and 6) Follow directly from combining Proposition \ref{prop:detsing}, Theorem \ref{th:fhcont}, and Lemma \ref{le:sdet} and Lemma \ref{le:msdet} respectively. 
 
\end{proof}

Finally we prove our variance estimate. 

\begin{proof}[Proof of Proposition \ref{prop:var}]
Let us write down explicitly all the definitions. We first expand the square and then note in each term we can interchange the order of integration as everything is non-negative.

\begin{align}
\notag \E&\left(\left(\int_0^{2\pi}f(e^{i\theta})(\mu_{N,M,L}^\beta(d\theta)-\mu_{N,M}^\beta(d\theta))\right)^2\right)\\
\notag &=\int_0^{2\pi}\int_0^{2\pi}f(e^{i\theta})f(e^{i\theta'})\frac{\E\left( F_{N,M,L}^\beta(e^{i\theta})F_{N,M,L}^\beta(e^{i\theta'})\right)}{\E(F_{N,M,L}^\beta(e^{i\theta}))\E(F_{N,M,L}^\beta(e^{i\theta'}))}\frac{d\theta}{2\pi}\frac{d\theta'}{2\pi}\\
&-2\int_0^{2\pi}\int_0^{2\pi}f(e^{i\theta})f(e^{i\theta'})\frac{\E\left( F_{N,M,L}^\beta(e^{i\theta})F_{N,M}^\beta(e^{i\theta'})\right)}{\E(F_{N,M,L}^\beta(e^{i\theta}))\E(F_{N,M}^\beta(e^{i\theta'}))}\frac{d\theta}{2\pi}\frac{d\theta'}{2\pi}\\
\notag &+\int_0^{2\pi}\int_0^{2\pi}f(e^{i\theta})f(e^{i\theta'})\frac{\E\left( F_{N,M}^\beta(e^{i\theta})F_{N,M}^\beta(e^{i\theta'})\right)}{\E(F_{N,M}^\beta(e^{i\theta}))\E(F_{N,M}^\beta(e^{i\theta'}))}\frac{d\theta}{2\pi}\frac{d\theta'}{2\pi}\\
&\notag =:I_1-2I_2+I_3.
\end{align}

Combining parts 1) and 3) of Proposition \ref{prop:fh}, we see that (even if we just assume $M-N\to \infty$ as $N\to\infty$), 

\begin{align}
\notag \lim_{N\to\infty}I_1& =\int_0^{2\pi}\int_0^{2\pi}f(e^{i\theta})f(e^{i\theta'})\frac{\E(e^{\beta X_L(\theta)}e^{\beta X_L(\theta'))}}{\E(e^{\beta X_L(\theta)}) \E(e^{\beta X_L(\theta')})}\frac{d\theta}{2\pi}\frac{d\theta'}{2\pi}\\
&=\int_0^{2\pi}\int_0^{2\pi}f(e^{i\theta})f(e^{i\theta'})e^{\frac{\beta^2}{4}\sum_{j=1}^L\frac{1}{j}\cos(j(\theta-\theta'))}\frac{d\theta}{2\pi}\frac{d\theta'}{2\pi}
\end{align}

Combining parts 1), 2), and 4) of Proposition \ref{prop:fh}, we have (as $N/M\to 0$)

\begin{align}
\notag \lim_{N\to\infty}I_2&=\int_0^{2\pi}\int_0^{2\pi}f(e^{i\theta})f(e^{i\theta'}) e^{\frac{\beta^2}{2}\sum_{j=1}^L \frac{1}{j}\cos(j(\theta-\theta'))}\frac{d\theta}{2\pi}\frac{d\theta'}{2\pi}.
\end{align}

For $I_3$, we repeat an argument from \cite{ck}. We split the domain of integration into $|\theta-\theta'|>2t_0$ and $|\theta-\theta'|<2t_0$. For the first case, we use 2) and 5) of Proposition \ref{prop:fh} and find

\begin{align}
\notag \lim_{N\to\infty}&\int_{|\theta-\theta'|>2t_0}f(e^{i\theta})f(e^{i\theta'})\frac{\E\left(F_{N,M}^\beta(e^{i\theta}) F_{N,M}^\beta(e^{i\theta'})\right)}{\E(F_{N,M}^\beta(e^{i\theta}))\E(F_{N,M}^\beta(e^{i\theta'}))}\frac{d\theta}{2\pi}\frac{d\theta'}{2\pi}\\
&=\int_{|\theta-\theta'|>2t_0}f(e^{i\theta})f(e^{i\theta'})|e^{i\theta}-e^{i\theta'}|^{-\frac{\beta^2}{2}}\frac{d\theta}{2\pi}\frac{d\theta'}{2\pi}.
\end{align}

For the latter case, we note that if $N|\theta-\theta|<1$, we note that by the asymptotics of $\sigma$ and the fact that $\log (x^{-1}\sin(x/2))$ is bounded for small $x$,  

\begin{equation}
\E(F_{N,M}^\beta(e^{i\theta})F_{N,M}^\beta(e^{i\theta'}))=\mathcal{O}(1)N^{\beta^2},
\end{equation}

\noindent where the $\mathcal{O}(1)$ is uniform in $N|\theta-\theta'|<1$, so we see that 

\begin{align}
\notag\int_{|\theta-\theta'|<N^{-1}}&f(e^{i\theta})f(e^{i\theta'})\frac{\E\left(F_{N,M}^\beta(e^{i\theta}) F_{N,M}^\beta(e^{i\theta})\right)}{\E(F_{N,M}^\beta(e^{i\theta}))\E(F_{N,M}^\beta(e^{i\theta'}))}\frac{d\theta}{2\pi}\frac{d\theta'}{2\pi}\\
&=\mathcal{O}(N^{-1})\mathcal{O}(1)\frac{N^{\beta^2}}{N^{\frac{\beta^2}{2}}}\\
\notag &=\mathcal{O}(N^{\frac{\beta^2}{2}-1})
\end{align}

\noindent which tends to zero as $N\to\infty$ under our assumptions.

\vspace{0.3cm}

On the other hand, for the $N^{-1}<|\theta-\theta'|<2t_0$ case we write 

\begin{align}
\notag\frac{\E\left(F_{N,M}^\beta(e^{i\theta})F_{N,M}^\beta(e^{i\theta'})\right)}{\E(F_{N,M}^\beta(e^{i\theta}))\E(F_{N,M}^\beta(e^{i\theta'}))}&=\mathcal{O}(1)e^{\int_0^{-i}\frac{1}{s}(\sigma(s)-\frac{\beta^2}{2})ds} N^{\frac{\beta^2}{2}}|e^{i\theta}-e^{i\theta'}|^{-\frac{\beta^2}{2}}\\
&\times e^{\int_{-i}^{-iN|\theta-\theta'|}\frac{1}{s}\left(\sigma(s)-\frac{\beta^2}{2}\right)ds}e^{\frac{\beta^2}{2}\log |\theta-\theta'|} ,
\end{align}

\noindent where the $\mathcal{O}(1)$ can be taken to depend only on $\beta$ (so there are precise asymptotics for it, but we care only about it being uniformly bounded for fixed $\beta$). Here we made use of the fact that $|2\sin (x/2)|=|e^{ix}-1|$. Again making use of the $s\to 0$ asymptotics of $\sigma$ the first integral is finite (and its value only depends on $\beta$), while the second integral we write as 

\begin{equation}
\int_{-i}^{-iN |\theta-\theta'|}\frac{\sigma(s)}{s}ds-\frac{\beta^2}{2}\log N-\frac{\beta^2}{2}\log |\theta-\theta'|.
\end{equation}

Making use of the $s\to \infty$ asymptotics of $\sigma$, we see that the above integral is uniformly bounded and we conclude that 

\begin{align}
\lim_{N\to\infty}&\int_{|\theta-\theta'|<2t_0}f(e^{i\theta})f(e^{i\theta'})\frac{\E\left(F_{N,M}^\beta(e^{i\theta})F_{N,M}^\beta(e^{i\theta'})\right)}{\E(F_{N,M}^\beta(e^{i\theta}))\E(F_{N,M}^\beta(e^{i\theta'}))}\frac{d\theta}{2\pi}\frac{d\theta'}{2\pi}\\
\notag &\leq C\int_{|\theta-\theta'|<2t_0}f(e^{i\theta})f(e^{i\theta'})  |e^{i\theta}-e^{i\theta'}|^{-\frac{\beta^2}{2}}\frac{d\theta}{2\pi}\frac{d\theta'}{2\pi}.
\end{align}

As $\lim_{N\to\infty}I_3$ must be independent of $t_0$, and this tends to zero as $t_0\to 0$ (since $\beta^2<2$), putting everything together we see that 

\begin{equation}
\lim_{N\to\infty}I_3=\int_0^{2\pi}\int_0^{2\pi}f(e^{i\theta})f(e^{i\theta'})|e^{i\theta}-e^{i\theta'}|^{-\frac{\beta^2}{2}}\frac{d\theta}{2\pi}\frac{d\theta'}{2\pi}.
\end{equation}

We conclude that for $N/M\to\infty$ as $N\to\infty$, 

\begin{align}\label{eq:var}
&\lim_{N\to\infty}\notag \E\left(\left(\int_0^{2\pi}f(e^{i\theta})(\mu_{N,M,L}^\beta(d\theta)-\mu_{N,M}^\beta(d\theta))\right)^2\right)\\
 &=\int_0^{2\pi}\int_0^{2\pi}f(e^{i\theta})f(e^{i\theta'})\left(|e^{i\theta}-e^{i\theta'}|^{-\frac{\beta^2}{2}}-e^{\frac{\beta^2}{2}\sum_{j=1}^L\frac{1}{j}\cos(j(\theta-\theta'))}\right)\frac{d\theta}{2\pi}\frac{d\theta'}{2\pi}.
\end{align}

As

\begin{equation}
\lim_{L\to\infty}e^{\frac{\beta^2}{2}\sum_{j=1}^L\frac{1}{j}\cos(j(\theta-\theta'))}=|e^{i\theta}-e^{i\theta'}|^{-\frac{\beta^2}{2}},
\end{equation}

\noindent Fatou's Lemma implies that 

\begin{align}
\notag\int_0^{2\pi}\int_0^{2\pi}&f(e^{i\theta})f(e^{i\theta'})|e^{i\theta}-e^{i\theta'}|^{-\frac{\beta^2}{2}}\frac{d\theta}{2\pi}\frac{d\theta'}{2\pi}\\
&\leq \limsup_{L\to\infty} \int_0^{2\pi}\int_0^{2\pi}f(e^{i\theta})f(e^{i\theta'})e^{\frac{\beta^2}{2}\sum_{j=1}^L\frac{1}{j}\cos(j(\theta-\theta'))}\frac{d\theta}{2\pi}\frac{d\theta'}{2\pi},
\end{align}

\noindent but \eqref{eq:var} implies the opposite relation (the left side of \eqref{eq:var} is non-negative - it's a variance) so we see that 

\begin{equation}
\lim_{L\to\infty}\lim_{N\to\infty}\E\left(\left(\int_0^{2\pi}f(e^{i\theta})(\mu_{N,M,L}^\beta(d\theta)-\mu_{N,M}^\beta(d\theta))\right)^2\right)=0.
\end{equation}

\end{proof}

\subsection{Proof of Theorem \ref{th:main}}

We now prove our main result. 

\begin{proof}[Proof of Theorem \ref{th:main}]
We wish to prove that for any continuous $f:\T\to [0,\infty)$, as $N\to\infty$ (and $N/M\to 0$)

\begin{equation}
\int_{0}^{2\pi}f(e^{i\theta})\mu_{N,M}^\beta(d\theta)\stackrel{d}{\to}\int_0^{2\pi} f(e^{i\theta})\mu^\beta(d\theta).
\end{equation}

We write 

\begin{align}
\notag \int_{0}^{2\pi}f(e^{i\theta})\mu_{N,M}^\beta(d\theta)&=\int_0^{2\pi}f(e^{i\theta}) \mu_{N,M,L}^\beta(d\theta)\\
&+\int_0^{2\pi}f(e^{i\theta})(\mu_{N,M}^\beta(d\theta)-\mu_{N,M,L}^\beta(d\theta)).
\end{align}

By Proposition \ref{prop:approxconv} and the definition of $\mu^\beta$, 

\begin{equation}
\int_0^{2\pi}f(e^{i\theta}) \mu_{N,M,L}^\beta(d\theta)\stackrel{d}{\to} \int_0^{2\pi}f(e^{i\theta})\mu^\beta(d\theta)
\end{equation}

\noindent if we first let $N\to\infty$ and then $L\to\infty$.

\vspace{0.3cm}

On the other hand, as we saw in Proposition \ref{prop:var} that

\begin{equation}
\lim_{L\to\infty}\lim_{N\to\infty}\E\left(\left(\int_0^{2\pi}f(e^{i\theta})(\mu_{N,M}^\beta(d\theta)-\mu_{N,M,L}^\beta(d\theta))\right)^2\right)=0,
\end{equation}

\noindent we have in particular that if we first let $N\to\infty$ and then $L\to\infty$,

\begin{equation}
\int_0^{2\pi}f(e^{i\theta})(\mu_{N,M}^\beta(d\theta)-\mu_{N,M,L}^\beta(d\theta))\stackrel{d}{\to}0.
\end{equation}

We conclude, by Slutsky's theorem, that if $N/M\to 0$ as $N\to\infty$, then 

\begin{equation}
\int_{0}^{2\pi}f(e^{i\theta})\mu_{N,M}^\beta(d\theta)\stackrel{d}{\to}\int_0^{2\pi} f(e^{i\theta})\mu^\beta(d\theta)
\end{equation}

\noindent as $N\to\infty$.

\end{proof}

\section{Discussion}

We now discuss less precisely what might happen to the Fredholm determinant $\det(I+K)$ in the case where $N/M\to q\in(0,1)$ as $N\to\infty$. For simplicity, we focus on the case where the distance between the singularities is bounded away from zero. From our analysis, it looks reasonable that $\mathrm{tr}(K^j)=\mathcal{O}((N/M)^j)$ so if we assume that $q$ is small, to argue that the asymptotics of $\det(I+K)$ differ from one, it should be enough to argue that $\mathrm{tr}(K)\nrightarrow 0$.

\vspace{0.3cm}

Looking at our argument, we see that the most relevant contribution to the trace comes from values of the kernel between $[M^{-1},N^{-1}]$ (close to the point $M^{-1}$ - elsewhere there seems to be a term of the form $e^{-1/q}$). Working a bit harder on the asymptotics of the orthogonal polynomials, one could check whether or not there is some cancellation in this regime (e.g. is there a symmetry causing the integrals along $(1-N^{-1},1-M^{-1})\times w_j$ and $(1+M^{-1},1+N^{-1})\times w_j$ to cancel). As the analysis of the asymptotics was a bit heavy as it is, we choose not to pursue this question further here, but note that it seems unlikely that there would be such a cancellation to all orders in $q$, which implies that it is quite reasonable to expect that the Fisher-Hartwig formula is unlikely to hold as it is for $q>0$. It would be an interesting problem to try to calculate more precise asymptotics of the Fredholm determinant in the $N/M\to q>0$ case. Another interesting questions is how this affects the law of the limiting chaos measure or perhaps the law of the maximum of the characteristic polynomial (for more on this type of questions in the context of characteristic polynomials of log-gases, see \cite{fk,fs})

\vspace{0.3cm}

We wish to underline the interesting contrast between this and the behavior of the linear statistics. If one just has $M-N\to\infty$ as $N\to\infty$, the law of any finite collection of the linear statistics agrees with the continuum case, so the linear statistics are essentially unable to see the discreteness of the model, while the characteristic polynomial does seem to see it. This is presumably due to the fact that the characteristic polynomial essentially contains information of all polynomial linear statistics, so it might be reasonable to expect that the polynomial linear statistics also see the discreteness when one takes the power to grow with $N$ and $M$.


\begin{thebibliography}{99}
\bibitem{ajks} K. Astala, P. Jones, A. Kupiainen, and E. Saksman: Random conformal weldings. Acta Math. 207 (2011), no. 2, 203–254. 
\bibitem{bkmlm} J. Baik, T. Kriecherbauer, K. T.-R. McLaughlin, and P.D. Miller: Discrete orthogonal polynomials. Asymptotics and applications. Annals of Mathematics Studies, 164. Princeton University Press, Princeton, NJ, 2007. 
\bibitem{bl} J. Baik and Z. Liu: Discrete Toeplitz/Hankel determinants and the width of nonintersecting processes. Int. Math. Res. Not. IMRN 2014, no. 20, 5737–5768.
\bibitem{berestycki} N. Berestycki: An elementary approach to Gaussian multiplicative chaos. Preprint arXiv:1506.09113.
\bibitem{bgg} A. Borodin, V. Gorin, and A. Guionnet: Gaussian asymptotics of discrete $\beta$-ensembles. Preprint arXiv:1505.03760.
\bibitem{ck} T. Claeys and I. Krasovsky: Toeplitz determinants with merging singularities. Preprint arXiv:1403.3639.
\bibitem{dik1} P. Deift, A. Its, and I. Krasovsky: Asymptotics of Toeplitz, Hankel, and Toeplitz+Hankel determinants with Fisher-Hartwig singularities. Ann. of Math. (2) 174 (2011), no. 2, 1243–1299.
\bibitem{dik2} P. Deift, A. Its, and I. Krasovsky: On the asymptotics of a Toeplitz determinant with singularities. Preprint arXiv:1206.1292.
\bibitem{ds} P. Diaconis and M. Shahshahani: On the eigenvalues of random matrices. Studies in applied probability. J. Appl. Probab. 31A (1994), 49–62. 
\bibitem{drz} J. Ding, R. Roy, and O. Zeitouni: Convergence of the centered maximum of log-correlated Gaussian fields. Preprint arXiv:1503.04588.
\bibitem{dupshef} B. Duplantier and S. Sheffield: Liouville quantum gravity and KPZ. Invent. Math. 185 (2011), no. 2, 333–393. 
\bibitem{forrester} P.J. Forrester: Log-gases and random matrices. London Mathematical Society Monographs Series, 34. Princeton University Press, Princeton, NJ, 2010. 
\bibitem{fks} Y.V. Fyodorov, B.A. Khoruzhenko, and N.J. Simm: Fractional Brownian motion with Hurst index $H=0$ and the Gaussian Unitary Ensemble. Preprint arXiv:1312.0212.
\bibitem{fk} Y.V. Fyodorov and J.P. Keating: Freezing transitions and extreme values: random matrix theory, and disordered landscapes. Philos. Trans. R. Soc. Lond. Ser. A Math. Phys. Eng. Sci. 372 (2014), no. 2007, 20120503, 32 pp. 
\bibitem{fs} Y.V. Fyodorov and N.J. Simm: On the distribution of maximum value of the characteristic polynomial of GUE random matrices. Preprint arXiv:1503.07110. 
\bibitem{hko} C.P. Hughes, J.P. Keating, and N. O'Connell: On the characteristic polynomial of a random unitary matrix. Comm. Math. Phys. 220 (2001), no. 2, 429–451. 
\bibitem{johansson1} K. Johansson: On fluctuations of eigenvalues of random Hermitian matrices. Duke Math. J. 91 (1998), no. 1, 151–204.
\bibitem{johansson2} K. Johansson: Discrete orthogonal polynomial ensembles and the Plancherel measure. Ann. of Math. (2) 153 (2001), no. 1, 259–296. 
\bibitem{js} J. Junnila and E. Saksman: The uniqueness of the Gaussian multiplicative chaos revisited. Preprint arXiv:1506.05099.
\bibitem{kahane} J.-P. Kahane: Sur le chaos multiplicatif. Ann. Sci. Math. Québec 9 (1985), no. 2, 105–150. 
\bibitem{kallenberg}  O. Kallenberg: Random measures. Third edition. Akademie-Verlag,
Berlin; Academic Press, Inc.
London, 1983.
\bibitem{mad} T. Madaule:Maximum of a log-correlated Gaussian field. Preprint arXiv:1307.1365.
\bibitem{mfmls1} A. Mart\'inez-Finkelshtein, K. T.-R. McLaughlin, E.B. Saff: Szeg\"o orthogonal polynomials with respect to an analytic weight: canonical representation and strong asymptotics. Constr. Approx. 24 (2006), no. 3, 319–363.
\bibitem{mfmls2} A. Mart\'inez-Finkelshtein, K. T.-R. McLaughlin, E.B. Saff: Asymptotics of orthogonal polynomials with respect to an analytic weight with algebraic singularities on the circle. Int. Math. Res. Not. 2006, Art. ID 91426, 43 pp.
\bibitem{mehta} M.L. Mehta: Random matrices. Third edition. Pure and Applied Mathematics (Amsterdam), 142. Elsevier/Academic Press, Amsterdam, 2004.  
\bibitem{rv} R. Rhodes and V. Vargas: Gaussian multiplicative chaos and applications: a review, Probab. Surveys vol 11 (2014),
315-392.
\bibitem{ridvir} B. Rider and B. Vir\'ag: The noise in the circular law and the Gaussian free field. Int. Math. Res. Not. IMRN 2007, no. 2, Art. ID rnm006, 33 pp. 
\bibitem{sheffield} S. Sheffield: Gaussian free fields for mathematicians. Probab. Theory Related Fields 139 (2007), no. 3-4, 521–541. 
\bibitem{shefweld} S. Sheffield: Conformal weldings of random surfaces: SLE and the quantum gravity zipper. Preprint arXiv:1012.4797.
\bibitem{simon} B. Simon: Notes on infinite determinants of Hilbert Space Operators. Advances in Math. 24 (1977), no. 3, 244–273. 
\bibitem{simonsz} B. Simon: The sharp form of the strong Szeg\"o theorem. Geometry,
spectral theory, groups, and dynamics, 253–275, Contemp. Math., 387,
Amer. Math. Soc., Providence, RI, 2005.
\bibitem{webb} C. Webb: The characteristic polynomial of a random unitary matrix and Gaussian multiplicative chaos - The $L^2$-phase. Preprint arXiv:1410.0939
\bibitem{widom} H. Widom: Toeplitz determinants with singular generating functions.
Amer. J. Math. 95 (1973), 333–383.
\end{thebibliography}
\end{document}